\documentclass[11pt]{article} 
\usepackage[leqno]{amsmath}
\usepackage{amssymb,amsthm,url}
\usepackage[letterpaper,margin=1in]{geometry}
\usepackage{graphicx}
\usepackage{color}
\numberwithin{equation}{section}
\numberwithin{figure}{section}
\usepackage{enumitem}		% customizable list environments
\usepackage{stmaryrd}
\usepackage{shuffle}
\usepackage{pict2e}
\usepackage{tikz}
\usetikzlibrary{arrows}
\usepackage{bbm}

\newlength{\lyxlabelwidth}      % auxiliary length 
\definecolor{GREEN}{RGB}{0, 180, 0}
\definecolor{BLUE}{RGB}{0, 0, 180}

\DeclareMathOperator{\Prob}{Prob}
\DeclareMathOperator{\Expect}{Expect}
\newcommand\f{\mathbf{f}}
\newcommand\fo{\overset{\approx}{\f}}
\newcommand\g{\mathbf{g}}
\newcommand\hatk{\check{K}}

\DeclareMathOperator{\End}{End}
\DeclareMathOperator{\id}{id}
\DeclareMathOperator{\Proj}{Proj} 
\DeclareMathOperator{\gr}{gr}
\DeclareMathOperator{\im}{im}
\DeclareMathOperator{\sspan}{span}

\newcommand\calsh{\mathcal{S}}

\DeclareMathOperator{\Anc}{Anc}

\DeclareMathOperator{\std}{std}

\newcommand\T{\mathbf{T}}

\newcommand\cppmap{\theta}
\newcommand\U{\mathbf{U}}
\newcommand\D{\mathbf{D}}

\DeclareMathOperator{\opt}{T} 
\DeclareMathOperator{\oper}{2R} 
\DeclareMathOperator{\opter}{T2R} 
\DeclareMathOperator{\opbinter}{BinT2R} 
\newcommand\ter{\opter}
\newcommand\trer{\opt r \oper}
\newcommand\binter{\opbinter}

\DeclareMathOperator{\opbintob}{BinT/B} 
\DeclareMathOperator{\optober}{T/B2R} 
\DeclareMathOperator{\optrintober}{TrinT/B2R} 
\newcommand\tober{\optober}
\newcommand\bintobrer{\opbintob r \oper}
\newcommand\trintober{\optrintober}

\DeclareMathOperator{\optab}{T+B} 
\DeclareMathOperator{\optaber}{T+B2R} 
\newcommand\taber{\optaber}
\newcommand\tabrer{\optab r \oper}

\newcommand\blp{\beta_{\lambda}^{P}}
\newcommand\bld{\beta_{\lambda}^{D}}

\newcommand\cale{\mathcal{E}}
\newcommand\calp{\mathcal{P}}
\newcommand\calu{\mathcal{U}} 

\newcommand\calh{\mathcal{H}}
\newcommand\calhn{\mathcal{H}_n}
\newcommand\calhdual{\mathcal{H}^*}
\newcommand\calhndual{\mathcal{H}_n^*}

\newcommand\calb{\mathcal{B}}
\newcommand\calbn{\mathcal{B}_n}
\newcommand\hatcalb{\check{\mathcal{B}}}

\newcommand\calbdual{\mathcal{B}^*}

\newcommand\fqsym{\mathbf{FQSym}}
\newcommand\sym{\mathbf{Sym}}
\newcommand\iprod{\cdot}
\newcommand\sn{\mathfrak{S}_n}

\newcommand\sk{\mathfrak{S}_{k}}
\newcommand\skj{\mathfrak{S}_{k-j}}
\newcommand\sj{\mathfrak{S}_j}

\date{}

\theoremstyle{plain}
\newtheorem{thm}[equation]{\protect\theoremname}
  \theoremstyle{definition}
  \newtheorem{defn}[equation]{\protect\definitionname}
  \theoremstyle{plain}
  \newtheorem{prop}[equation]{\protect\propositionname}
  \theoremstyle{plain}
  \newtheorem{cor}[equation]{\protect\corollaryname}
  \theoremstyle{plain}
  \newtheorem{lem}[equation]{\protect\lemmaname}
  \theoremstyle{remark}
  \newtheorem*{rem*}{\protect\remarkname}
  \theoremstyle{remark}
  \newtheorem*{rems*}{\protect\remarksname}
  \theoremstyle{definition}
  \newtheorem{example}[equation]{\protect\examplename}

\newcommand{\lrtreeone}{
\begin{picture}  (9,6) (-2,3)     
\put(2.5,9){\line(-1, -2){4}}
\put(0.5,5){\line(1, -2){2}}
\put(1.5,3){\line(-1, -2){1}}
\put(2.5,9){\line(1, -2){4}}
\put(2.5,9){\color{blue}\circle*{1}}
\put(0.5,5){\color{red}\circle*{1}}
\put(1.5,3){\color{green}\circle*{1}}
\end{picture}
}

\newcommand{\lrtreetwo}{
\begin{picture}  (9,6) (-2,3)     
\put(2.5,9){\line(-1, -2){4}}
\put(0.5,5){\line(1, -2){2}}
\put(2.5,9){\line(1, -2){4}}
\put(2.5,9){\color{blue}\circle*{1}}
\put(0.5,5){\color{green}\circle*{1}}
\end{picture}
}

\newcommand{\lrtreethree}{
\begin{picture}  (9,6) (-2,3)     
\put(2.5,9){\line(-1, -2){4}}
\put(0.5,5){\line(1, -2){2}}
\put(-0.5,3){\line(1, -2){1}}
\put(2.5,9){\line(1, -2){4}}
\put(2.5,9){\color{blue}\circle*{1}}
\put(0.5,5){\color{green}\circle*{1}}
\put(-0.5,3){\color{red}\circle*{1}}
\end{picture}
}
\newcommand{\lrtreefour}{
\begin{picture}  (9,6) (-2,3)    
\put(2.5,9){\line(-1, -2){4}}
\put(0.5,5){\line(1, -2){2}}
\put(1.5,3){\line(-1, -2){1}}
\put(2.5,9){\line(1, -2){4}}
\put(2.5,9){\color{blue}\circle*{1}}
\put(0.5,5){\color{green}\circle*{1}}
\put(1.5,3){\color{red}\circle*{1}}
\end{picture}
}
\newcommand{\lrtreefive}{
\begin{picture}  (9,6) (-2,3)     
\put(2.5,9){\line(-1, -2){4}}
\put(0.5,5){\line(1, -2){1.5}}
\put(4.5,5){\line(-1, -2){1.5}}
\put(2.5,9){\line(1, -2){4}}
\put(2.5,9){\color{blue}\circle*{1}}
\put(0.5,5){\color{green}\circle*{1}}
\put(4.5,5){\color{red}\circle*{1}}
\end{picture}
}

\newcommand{\firing}{
\begin{picture} (15,20) (-3,-9) 
\put(5,7){\circle*{1}} 
\put(5,8){\tiny boss} 
\put(5,7){\line(-3, -4){3}}
\put(2,3){\circle*{1}}
\put(0,3){\tiny A} 
\put(2,3){\line(0, -1){4}}
\put(2,-1){\circle*{1}}
\put(0,-1){\tiny B} 
\put(5,7){\line(3, -4){3}}
\put(8,3){\circle*{1}}
\put(9,3){{\color{blue}\tiny ?}} 
\put(8,3){\line(-3, -4){3}}
\put(5,-1){\circle*{1}}
\put(4.5,-3){\tiny D} 
\put(8,3){\line(0, -1){4}}
\put(8,-1){\circle*{1}}
\put(7.5,-3){\tiny E} 
\put(8,3){\line(3, -4){3}}
\put(11,-1){\circle*{1}}
\put(12,-1.5){\tiny F} 
\put(11,-1){\line(0, -1){3}}
\put(11,-4){\circle*{1}}
\put(12,-4.5){\tiny G} 
\end{picture}
}

\newcommand{\firingone}{
\begin{picture} (15,20) (-3,-9) 
\put(5,7){\circle*{1}} 
\put(5,8){\tiny boss} 
\put(5,7){\line(-3, -4){3}}
\put(2,3){\circle*{1}}
\put(0,3){\tiny A} 
\put(2,3){\line(0, -1){4}}
\put(2,-1){\circle*{1}}
\put(0,-1){\tiny B} 
\put(5,7){\line(3, -4){3}}
\put(8,3){\circle*{1}}
\put(9,3){{\color{blue}\tiny D}}  
\put(8,3){\line(0, -1){4}}
\put(8,-1){\circle*{1}}
\put(7.5,-3){\tiny E} 
\put(8,3){\line(3, -4){3}}
\put(11,-1){\circle*{1}}
\put(12,-1.5){\tiny F} 
\put(11,-1){\line(0, -1){3}}
\put(11,-4){\circle*{1}}
\put(12,-4.5){\tiny G} 
\end{picture}
}

\newcommand{\firingtwo}{
\begin{picture} (15,20) (-3,-9) 
\put(5,7){\circle*{1}} 
\put(5,8){\tiny boss} 
\put(5,7){\line(-3, -4){3}}
\put(2,3){\circle*{1}}
\put(0,3){\tiny A} 
\put(2,3){\line(0, -1){4}}
\put(2,-1){\circle*{1}}
\put(0,-1){\tiny B} 
\put(5,7){\line(3, -4){3}}
\put(8,3){\circle*{1}}
\put(9,3){{\color{blue}\tiny E}}  
\put(8,3){\line(-3, -4){3}}
\put(5,-1){\circle*{1}}
\put(4.5,-3){\tiny D} 
\put(8,3){\line(3, -4){3}}
\put(11,-1){\circle*{1}}
\put(12,-1.5){\tiny F} 
\put(11,-1){\line(0, -1){3}}
\put(11,-4){\circle*{1}}
\put(12,-4.5){\tiny G} 
\end{picture}
}

\newcommand{\firingthree}{
\begin{picture} (15,20) (-3,-9) 
\put(5,7){\circle*{1}} 
\put(5,8){\tiny boss} 
\put(5,7){\line(-3, -4){3}}
\put(2,3){\circle*{1}}
\put(0,3){\tiny A} 
\put(2,3){\line(0, -1){4}}
\put(2,-1){\circle*{1}}
\put(0,-1){\tiny B} 
\put(5,7){\line(3, -4){3}}
\put(8,3){\circle*{1}}
\put(9,3){{\color{blue}\tiny F}}  
\put(8,3){\line(-3, -4){3}}
\put(5,-1){\circle*{1}}
\put(4.5,-3){\tiny D} 
\put(8,3){\line(0, -1){4}}
\put(8,-1){\circle*{1}}
\put(7.5,-3){\tiny E} 
\put(8,3){\line(3, -4){3}}
\put(11,-1){\circle*{1}}
\put(12,-1.5){{\color{blue}\tiny ?}}  
\put(11,-1){\line(0, -1){3}}
\put(11,-4){\circle*{1}}
\put(12,-4.5){\tiny G} 
\end{picture}
}

\newcommand{\firingthreeb}{
\begin{picture} (15,20) (-3,-9) 
\put(5,7){\circle*{1}} 
\put(5,8){\tiny boss} 
\put(5,7){\line(-3, -4){3}}
\put(2,3){\circle*{1}}
\put(0,3){\tiny A} 
\put(2,3){\line(0, -1){4}}
\put(2,-1){\circle*{1}}
\put(0,-1){\tiny B} 
\put(5,7){\line(3, -4){3}}
\put(8,3){\circle*{1}}
\put(9,3){{\color{blue}\tiny F}}  
\put(8,3){\line(-3, -4){3}}
\put(5,-1){\circle*{1}}
\put(4.5,-3){\tiny D} 
\put(8,3){\line(0, -1){4}}
\put(8,-1){\circle*{1}}
\put(7.5,-3){\tiny E} 
\put(8,3){\line(3, -4){3}}
\put(11,-1){\circle*{1}}
\put(12,-1.5){{\color{blue}\tiny G}} 
\end{picture}
}

\newcommand{\firingfour}{
\begin{picture} (15,20) (-3,-9) 
\put(5,7){\circle*{1}} 
\put(5,8){\tiny boss} 
\put(5,7){\line(-3, -4){3}}
\put(2,3){\circle*{1}}
\put(0,3){\tiny A} 
\put(2,3){\line(0, -1){4}}
\put(2,-1){\circle*{1}}
\put(0,-1){\tiny B} 
\put(5,7){\line(3, -4){3}}
\put(8,3){\circle*{1}}
\put(9,3){{\color{blue}\tiny G}}  
\put(8,3){\line(-3, -4){3}}
\put(5,-1){\circle*{1}}
\put(4.5,-3){\tiny D} 
\put(8,3){\line(0, -1){4}}
\put(8,-1){\circle*{1}}
\put(7.5,-3){\tiny E} 
\put(8,3){\line(3, -4){3}}
\put(11,-1){\circle*{1}}
\put(12,-1.5){\tiny F} 
\end{picture}
}

\newcommand{\basisone}{
\begin{picture} (5,10) %(0,0) 
\put(2.5,9){\circle*{1}} 
\end{picture}
}

\newcommand{\basistwo}{
\begin{picture} (5,10) %(0,0) 
\put(2.5,9){\circle*{1}} 
\put(2.5,9){\line(-1, -2){2}}
\put(0.5,5){\circle*{1}}
\end{picture}
}

\newcommand{\basisthree}{
\begin{picture} (5,10) %(0,0) 
\put(2.5,9){\circle*{1}} 
\put(2.5,9){\line(1, -2){2}}
\put(4.5,5){\circle*{1}}
\end{picture}
}

\newcommand{\basisfour}{
\begin{picture} (5,10) %(0,0) 
\put(2.5,9){\circle*{1}} 
\put(2.5,9){\line(1, -2){2}}
\put(4.5,5){\circle*{1}}
\put(4.5,5){\line(0, -1){3}}
\put(4.5,2){\circle*{1}}
\end{picture}
}

\newcommand{\basisfive}{
\begin{picture} (5,10) %(0,0) 
\put(2.5,9){\circle*{1}} 
\put(2.5,9){\line(-1, -2){2}}
\put(0.5,5){\circle*{1}}
\put(2.5,9){\line(1, -2){2}}
\put(4.5,5){\circle*{1}}
\end{picture}
}

\newcommand{\basissix}{
\begin{picture} (5,10) %(0,0) 
\put(2.5,9){\circle*{1}} 
\put(2.5,9){\line(-1, -2){2}}
\put(0.5,5){\circle*{1}}
\put(2.5,9){\line(1, -2){2}}
\put(4.5,5){\circle*{1}}
\put(4.5,5){\line(0, -1){3}}
\put(4.5,2){\circle*{1}}
\end{picture}
}

\newcommand{\treeex}{
\begin{picture} (12,6) %(0,0) 
\put(5,7){\circle*{1}} 
\put(5,7){\line(-3, -4){3}}
\put(2,3){\circle*{1}}
\put(2,3){\line(0, -1){4}}
\put(2,-1){\circle*{1}}
\put(5,7){\line(3, -4){3}}
\put(8,3){\circle*{1}}
\put(8,3){\line(-3, -4){3}}
\put(5,-1){\circle*{1}}
\put(8,3){\line(0, -1){4}}
\put(8,-1){\circle*{1}}
\put(8,3){\line(3, -4){3}}
\put(11,-1){\circle*{1}}
\put(11,-1){\line(0, -1){3}}
\put(11,-4){\circle*{1}}
\end{picture}
}

\newcommand{\coprodone}{
\begin{picture} (12,6) %(0,0) 
\put(5,7){\circle*{1}} 
\put(5,7){\line(-3, -4){3}}
\put(2,3){\circle*{1}}
\put(5,7){\line(3, -4){3}}
\put(8,3){\circle*{1}}
\put(8,3){\line(-3, -4){3}}
\put(5,-1){\circle*{1}}
\put(8,3){\line(0, -1){4}}
\put(8,-1){\circle*{1}}
\put(8,3){\line(3, -4){3}}
\put(11,-1){\circle*{1}}
\put(11,-1){\line(0, -1){3}}
\put(11,-4){\circle*{1}}
\end{picture}
}

\newcommand{\coprodtwo}{
\begin{picture} (12,6) %(0,0) 
\put(5,7){\circle*{1}} 
\put(5,7){\line(-3, -4){3}}
\put(2,3){\circle*{1}}
\put(2,3){\line(0, -1){4}}
\put(2,-1){\circle*{1}}
\put(5,7){\line(3, -4){3}}
\put(8,3){\circle*{1}}
\put(8,3){\line(0, -1){4}}
\put(8,-1){\circle*{1}}
\put(8,3){\line(3, -4){3}}
\put(11,-1){\circle*{1}}
\put(11,-1){\line(0, -1){3}}
\put(11,-4){\circle*{1}}
\end{picture}
}

\newcommand{\coprodthree}{
\begin{picture} (12,6) %(0,0) 
\put(5,7){\circle*{1}} 
\put(5,7){\line(-3, -4){3}}
\put(2,3){\circle*{1}}
\put(2,3){\line(0, -1){4}}
\put(2,-1){\circle*{1}}
\put(5,7){\line(3, -4){3}}
\put(8,3){\circle*{1}}
\put(8,3){\line(-3, -4){3}}
\put(5,-1){\circle*{1}}
\put(8,3){\line(3, -4){3}}
\put(11,-1){\circle*{1}}
\put(11,-1){\line(0, -1){3}}
\put(11,-4){\circle*{1}}
\end{picture}
}

\newcommand{\coprodfour}{
\begin{picture} (12,6) %(0,0) 
\put(5,7){\circle*{1}} 
\put(5,7){\line(-3, -4){3}}
\put(2,3){\circle*{1}}
\put(2,3){\line(0, -1){4}}
\put(2,-1){\circle*{1}}
\put(5,7){\line(3, -4){3}}
\put(8,3){\circle*{1}}
\put(8,3){\line(-3, -4){3}}
\put(5,-1){\circle*{1}}
\put(8,3){\line(0, -1){4}}
\put(8,-1){\circle*{1}}
\put(8,3){\line(3, -4){3}}
\put(11,-1){\circle*{1}}
\end{picture}
}

\newcommand{\hookwalk}{
\begin{picture} (15,20) (-3,-9) 
\put(5,7){\circle*{1}} 
\put(5,7){\line(-3, -4){3}}
\put(2,3){\circle*{1}}
\put(2,3){\line(0, -1){4}}
\put(2,-1){\circle*{1}}
\put(5,7){\line(3, -4){3}}
\put(8,3){\circle*{1}}
\put(9,3){v} 
\put(8,3){\line(-3, -4){3}}
\put(5,-1){{\color{blue}\circle*{1}}}
\put(8,3){\line(0, -1){4}}
\put(8,-1){{\color{blue}\circle*{1}}}
\put(8,3){\line(3, -4){3}}
\put(11,-1){{\color{blue}\circle*{1}}}
\put(11,-1){\line(0, -1){3}}
\put(11,-4){{\color{blue}\circle*{1}}}
\end{picture}
}

\newcommand{\hookwalkone}{
\begin{picture} (15,20) (-3,-9) 
\put(5,7){\circle*{1}} 
\put(5,7){\line(-3, -4){3}}
\put(2,3){\circle*{1}}
\put(2,3){\line(0, -1){4}}
\put(2,-1){\circle*{1}}
\put(5,7){\line(3, -4){3}}
\put(8,3){\circle*{1}}
\put(8,3){\line(-3, -4){3}}
\put(5,-1){\circle*{1}}
\put(4.2,-3.2){v}
\put(8,3){\line(0, -1){4}}
\put(8,-1){\circle*{1}}
\put(8,3){\line(3, -4){3}}
\put(11,-1){\circle*{1}}
\put(11,-1){\line(0, -1){3}}
\put(11,-4){\circle*{1}}
\end{picture}
}

\newcommand{\hookwalktwo}{
\begin{picture} (15,20) (-3,-9) 
\put(5,7){\circle*{1}} 
\put(5,7){\line(-3, -4){3}}
\put(2,3){\circle*{1}}
\put(2,3){\line(0, -1){4}}
\put(2,-1){\circle*{1}}
\put(5,7){\line(3, -4){3}}
\put(8,3){\circle*{1}}
\put(8,3){\line(-3, -4){3}}
\put(5,-1){\circle*{1}}
\put(7.2,-3.2){v} 
\put(8,3){\line(0, -1){4}}
\put(8,-1){\circle*{1}}
\put(8,3){\line(3, -4){3}}
\put(11,-1){\circle*{1}}
\put(11,-1){\line(0, -1){3}}
\put(11,-4){\circle*{1}}
\end{picture}
}

\newcommand{\hookwalkthree}{
\begin{picture} (15,20) (-3,-9) 
\put(5,7){\circle*{1}} 
\put(5,7){\line(-3, -4){3}}
\put(2,3){\circle*{1}}
\put(2,3){\line(0, -1){4}}
\put(2,-1){\circle*{1}}
\put(5,7){\line(3, -4){3}}
\put(8,3){\circle*{1}}
\put(8,3){\line(-3, -4){3}}
\put(5,-1){\circle*{1}}
\put(8,3){\line(0, -1){4}}
\put(8,-1){\circle*{1}}
\put(8,3){\line(3, -4){3}}
\put(11,-1){\circle*{1}}
\put(12,-1.5){v}
\put(11,-1){\line(0, -1){3}}
\put(11,-4){\circle*{1}}
\end{picture}
}

\newcommand{\hookwalkfour}{
\begin{picture} (15,20) (-3,-9) 
\put(5,7){\circle*{1}} 
\put(5,7){\line(-3, -4){3}}
\put(2,3){\circle*{1}}
\put(2,3){\line(0, -1){4}}
\put(2,-1){\circle*{1}}
\put(5,7){\line(3, -4){3}}
\put(8,3){\circle*{1}}
\put(8,3){\line(-3, -4){3}}
\put(5,-1){\circle*{1}}
\put(8,3){\line(0, -1){4}}
\put(8,-1){\circle*{1}}
\put(8,3){\line(3, -4){3}}
\put(11,-1){\circle*{1}}
\put(11,-1){\line(0, -1){3}}
\put(11,-4){\circle*{1}}
\put(12,-4.5){v} 
\end{picture}
}

{\settowidth{\lyxlabelwidth}{#2}
\begin{description}[font=\normalfont,style=sameline,
leftmargin=\lyxlabelwidth,#1]}
{\end{description}}
  \providecommand{\remarksname}{Remarks}

\makeatother

  \providecommand{\corollaryname}{Corollary}
  \providecommand{\definitionname}{Definition}
  \providecommand{\examplename}{Example}
  \providecommand{\lemmaname}{Lemma}
  \providecommand{\propositionname}{Proposition}
  \providecommand{\remarkname}{Remark}
\providecommand{\theoremname}{Theorem}

\begin{document}

\title{Markov Chains from Descent Operators on Combinatorial Hopf Algebras}

\author{C.Y. Amy Pang%
\thanks{amypang@hkbu.edu.hk%
}}

\maketitle
\maketitle
\begin{abstract}
We develop a general theory for Markov chains whose transition probabilities
are the coefficients of descent operators on combinatorial Hopf algebras.
These model the breaking-then-recombining of combinational objects.
Examples include the various card-shuffles of Diaconis, Fill and Pitman,
Fulman's restriction-then-induction chains on the representations
of the symmetric group, and a plethora of new chains on trees, partitions
and permutations. The eigenvalues of these chains can be calculated
in a uniform manner using Hopf algebra structure theory, and there
is a simple expression for their stationary distributions. For an
important subclass of chains analogous to the top-to-random shuffle,
we derive a full right eigenbasis, from which follow exact expressions
for expectations of certain statistics of interest. This greatly generalises
the coproduct-then-product chains previously studied in joint work
with Persi Diaconis and Arun Ram.
\end{abstract}

\section{Introduction}

There has long been interest in using algebra and combinatorics to
study Markov chains \cite{paseplift,sandpile,randomtorandom}. One
highly successful technique is the theory of random walks on groups
\cite{randomwalksongroups,randomwalksongroups2} and on monoids \cite{hyperplanewalk,lrb,rtrivialmonoids}.
The transition matrix of one of these chains is the matrix for a multiplication
map in the group algebra or the monoid algebra. The eigenvalues, and
sometimes eigenfunctions, of the transition matrix can be calculated
in terms of the representation theory of the group or monoid. Such
algebraic data has implications for the long term behaviour of the
chain, such as its stationary distribution and convergence rate.

The purpose of this paper is to execute similar eigen-data analysis
when the transition matrix comes from a \emph{descent operator} on
a \emph{combinatorial Hopf algebra}, instead of from multiplication
in a group or monoid. The chains describe the breaking, or breaking-then-recombining,
of some combinatorial object, such as a graph or a tree. To illustrate
the general theory, this paper will concentrate on the following two
examples.
\begin{example}[A chain on organisational structures]
\label{ex:rooted-trees} A company starts with $n_{0}$ employees
in a tree structure, where each employee except the boss has exactly
one direct superior. The boss is the direct superior of the department
heads. For example, Figure \ref{fig:cktrees-company} is a company
with $n_{0}=8$. It has two departments: A heads the accounting department,
consisting of himself and B, and C heads the consulting department,
of C, D, E, F and G. And C is the direct superior to D, E and F.

Let $q$ be a parameter between 0 and 1. The monthly performance of
each employee is, independently, uniformly distributed between 0 and
1, and each month all employees with performance below $1-q$ are
fired. Each firing causes a cascade of promotions; for the specifics,
see the third paragraph of Section \ref{sec:cktrees} and Figure \ref{fig:cktrees-firing}.
The chain keeps track of the tree structure of the company, but does
not see which employee is taking which position. 

It is clear that the stationary distribution of the chain is concentrated
at the company consisting of only the boss. The eigenvalues of this
chain are 1 and $q^{n'}$, for $2\leq n'\leq n_{0}$, and their multiplicities
are the number of ``connected'' subsets of size $n'$ of the starting
tree containing the boss. (So, for the chain starting at the tree
in Figure \ref{fig:cktrees-company}, the eigenvalue $q^{3}$ has
multiplicity 5, corresponding to the subsets of the boss with A and
B, with A and C, with C and D, with C and E, and with C and F.) Theorem
\ref{thm:spectrum-cktrees}.iii gives a formula for a full right eigenbasis
of this chain. One application of these expressions is as follows:
suppose the company in Figure \ref{fig:cktrees-company} has a project
that requires $s_{1}$ accountants and $s_{2}$ consultants. The expected
number of such teams falls roughly by a factor of $q^{1+s_{1}+s_{2}}$
monthly (see Corollary \ref{cor:expectation-cktrees} for the precise
statement). Such results are obtained by relating this chain to a
decorated version of the Connes-Kreimer Hopf algebra of trees. 
\end{example}
\begin{figure}
\setlength{\unitlength}{2mm} \begin{center} \begin{picture} (15,12) (-3,-5)   \put(5,7){\circle*{1}}  \put(5,8){boss}  \put(5,7){\line(-3, -4){3}} \put(2,3){\circle*{1}} \put(0,3){A}  \put(2,3){\line(0, -1){4}} \put(2,-1){\circle*{1}} \put(0,-1){B}  \put(5,7){\line(3, -4){3}} \put(8,3){\circle*{1}} \put(9,3){C}  \put(8,3){\line(-3, -4){3}} \put(5,-1){\circle*{1}} \put(4.5,-3){D}  \put(8,3){\line(0, -1){4}} \put(8,-1){\circle*{1}} \put(7.5,-3){E}  \put(8,3){\line(3, -4){3}} \put(11,-1){\circle*{1}} \put(12,-1.5){F}  \put(11,-1){\line(0, -1){3}} \put(11,-4){\circle*{1}} \put(12,-4.5){G}  \put(-3,-5){$ \underbrace{\hspace{13mm}}_\text{accounting} $}  \put(5,-5){$ \underbrace{\hspace{13mm}}_\text{consulting} $}  \end{picture} \end{center}\protect\caption{\label{fig:cktrees-company}A company with a tree structure}
\end{figure}
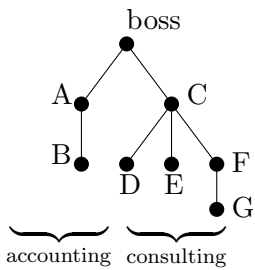

\begin{example}[Relative time on a to-do list]
 \label{ex:fqsym} You keep an electronic to-do list of $n$ tasks.
Each day, you complete the task at the top of the list, and are handed
a new task, which you add to the list in a position depending on its
urgency (more urgent tasks are placed closer to the top). Assume the
incoming tasks are equally distributed in urgency relative to the
$n-1$ tasks presently on the list, so they are each inserted into
the list in a uniform position. To produce from this process a Markov
chain on permutations, relabel the tasks by $\{1,\dots,n\}$ at the
end of each day so that 1 indicates the newest addition to the list,
2 indicates the next newest task, i.e. the newest amonst the $n-1$
tasks not labelled 1, and so on, so that $n$ is the task spending
the longest time on the list. 

Figure \ref{fig:todoexample} shows a possible trajectory of this
chain over three days. On the first day, the top task, labelled 3,
is completed, and tasks 1, 5, 4, 2 remain, in that order. In this
example, the random urgency of the new incoming task places it in
position 3. This incoming task is now the task spending the shortest
time on the list, so it is relabelled 1. The task spending the next
shortest time on the list is that previously labelled 1, and its label
is now changed to 2. Similarly, the task spending the third-shortest
time on the list is that previously labelled 2, and its label is now
changed to 3, and so on. The next two days are similar.

\begin{figure}
\def\arraystretch{1}
\begin{center}
\begin{tikzpicture} 
%\node (X) at (-1.3, 0) {end-of-day}; 
%\node (Y) at (-1.3, -1.8) {middle-of-day};
%\node (Z) at (-1.3, -2.2) {(completed a task, new task not yet arrived)};
\node (A) at (0,0) {$\begin{matrix}3\\1\\5\\4\\2\end{matrix}$}; 
\node (B) at (4,0)  {$\begin{matrix}2\\5\\1\\4\\3\end{matrix}$}; 
\node (C) at (8,0)  {$\begin{matrix}5\\2\\4\\1\\3\end{matrix}$};  
\node (D) at (12,0) {$\begin{matrix}1\\3\\5\\2\\4\end{matrix}$}; 
%\node (AA) at (0,0) {$\rightarrow$}; 
%\node (BB) at (4,0)  {$\rightarrow$}; 
%\node (CC) at (8,0)  {$\rightarrow$};  
%\node (DD) at (12,0) {$\rightarrow$}; 
\node (E) at (2,-2) {$\begin{matrix}1\\5\\ \rightarrow \hspace{5mm}\\4\\2\end{matrix}$}; 
\node (F) at (6,-2)  {$\begin{matrix}5\\1\\4\\ \rightarrow \hspace{5mm}\\3\end{matrix}$}; 
\node (G) at (10,-2)  {$\begin{matrix}\rightarrow \hspace{5mm}\\ 2\\5\\3\\4\end{matrix}$}; 
%\node (E) at (2,-2) {$\begin{matrix}3\\5\\4\\2\end{matrix}$}; 
%\node (F) at (6,-2)  {$\begin{matrix}5\\4\\2\\3\end{matrix}$}; 
%\node (G) at (10,-2)  {$\begin{matrix}2\\5\\3\\4\end{matrix}$}; 
\draw[->] (A) -- (E); 
\draw[->] (E) -- (B); 
\draw[->] (B) -- (F); 
\draw[->] (F) -- (C); 
\draw[->] (C) -- (G); 
\draw[->] (G) -- (D); 
\end{tikzpicture}
\par\end{center}

\protect\caption{\label{fig:todoexample} A possible three-day trajectory of the Markov
chain of ``relative time on a to-do list'', for $n=5$. The horizontal
arrows indicate the (random) positions of incoming tasks.}

\end{figure}
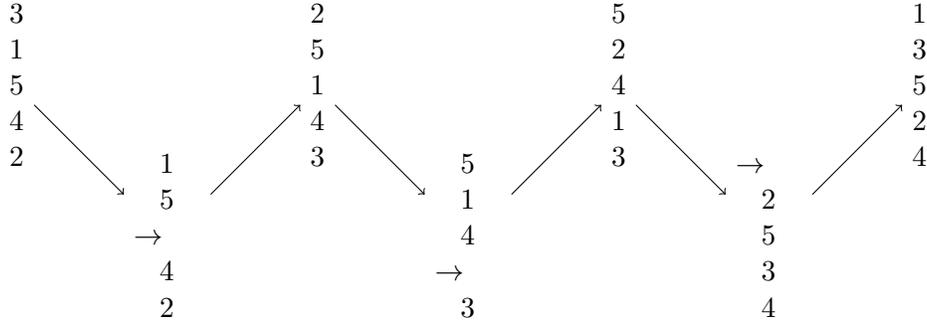

The stationary distribution of this chain is the uniform distribution.
Its eigenvalues are $0,\frac{1}{n},\frac{2}{n},\dots,\frac{n-2}{n},1$,
and the multiplicity of $\frac{j}{n}$ is $(n-j)!-(n-j-1)!$, the
number of permutations of $n$ objects fixing pointwise $1,2,\dots,j$
but not $j+1$. Theorem \ref{thm:spectrum-fqsym} gives a full right
eigenbasis indexed by such permutations. One consequence of such eigenvectors
(Corollary \ref{cor:probestimate-fqsym}) is as follows: assume the
tasks started in ascending order (i.e. newest at the top and oldest
at the bottom). After $t$ days, consider the $n-j$ tasks at the
bottommost positions of the list. The position of the newest task
among these is distributed as follows: 
\begin{align*}
\mbox{position }j+1 & \mbox{ with probability }\frac{1}{n-j}\left(1+\frac{j}{n}(n-j-1)\right);\\
\mbox{positions }j+2,j+3,\dots,n & \mbox{ each with probability }\frac{1}{n-j}\left(1+\frac{j}{n}\right).
\end{align*}

Remarkably, the above formula also gives the distribution of the card
of smallest value after $t$ top-to-random shuffles of \cite{toptorandom}
- this is because the probability distribution of a deck under top-to-random
shuffling, after $t$ steps starting at the identity permutation,
agrees with those of the to-do list chain \cite[Th. 3.14]{hopfchainlift}.
(Note however that the trajectories for the two chains do not correspond.)
So the new to-do list chain opens up a whole new viewpoint to study
the top-to-random shuffle. \cite[Th. 3.11]{hopfchainlift} also relates
the to-do list chain to the restriction-then-induction chains of Fulman
\cite{jasonrepdownupchains} (see below), namely that observing the
RSK shape of the to-do list chain gives the restriction-then-induction
chain. 
\end{example}
The descent operators $m\Delta_{P}$ are variants of the coproduct
$\Delta$ followed by the product $m$, indexed by a probability distribution
$P$ on compositions. In terms of the associated Markov chains, $P$
is the distribution of the piece sizes in the breaking step. For example,
when $P$ is concentrated at the composition $(1,n-1)$, the chain
models removing then reattaching a piece of size 1, analogous to the
``top-to-random'' card-shuffle. The associated descent operator
is $\ter_{n}:=\frac{1}{n}m(\Proj_{1}\otimes\id)\Delta$, where $\Proj_{1}$
denotes projection to the degree 1 subspace. The case where $P$ is
a multinomial distribution was covered in \cite{hopfpowerchains};
the present paper is a vast extension of that framework. As in \cite{hopfpowerchains},
all our results rely heavily on the structure theory of cocommutative
Hopf algebras; however, unlike \cite{hopfpowerchains}, some basic
Markov chain theory plays an essential part in the proofs of entirely
algebraic statements (Theorems \ref{thm:evalues}, \ref{thm:tob2r-evectors-cocommutative}
and \ref{thm:tab2r-evectors-cocommutative}).

This paper is organised as follows: Section \ref{sec:Background}
skims the necessary probabilistic and algebraic background, including
much non-standard, probability-inspired notation regarding descent
operators. Section \ref{sec:Markov-Chains-from-descent-operators}
finds the eigenvalues, multiplicities, and stationary distributions
of all descent operator Markov chains. Section \ref{sec:partsof1}
proves an eigenbasis formula for the chains driven by $\ter_{n}$
and related maps, and establishes a ``recursive lumping'' property
for such chains. Sections \ref{sec:cktrees} and \ref{sec:fqsym}
applies these general results to, respectively, the decorated Connes-Kreimer
Hopf algebra of trees and the Malvenuto-Reutenauer algebra of permutations
to obtain the results of Examples \ref{ex:rooted-trees} and \ref{ex:fqsym}
above. 

The main results of this paper were announced in the extended abstract
\cite{cpmcfpsac}.

Below are some vignettes of $\ter_{n}$ chains on various Hopf algebras,
to demonstrate the diverse range of chains that this framework covers.
Note that the long term behaviour varies greatly depending on the
properties of the associated Hopf algebra. 
\begin{description}
\item [{Card-shuffling}] The descent operators $m\Delta_{P}$ on the shuffle
algebra (see Example \ref{ex:shufflealg-bg}) induce precisely the
``cut and interleave'' shuffles of \cite{cppriffleshuffle}: cut
the deck according to the distribution $P$, then drop the cards one
by one from the bottom of piles chosen with probability proportional
to pile size. In particular, $\ter_{n}$ corresponds to the much-studied
``top-to-random'' shuffle \cite{toptorandom,entropyter}: take the
top card off the deck, then re-insert it at a uniformly chosen position.
(The time-reversal of this chain is the very important Tsetlin library
model \cite{tsetlin} of storage allocation, see \cite{tsetlinlibrary}
for an overview of the many results on this chain and pointers to
diverse applications.) A multi-graded version of Theorem \ref{thm:evalues}
recovers (the unweighted case of) \cite{randomtotopevalues} on the
spectrum of the top-to-random shuffle: if all $n$ cards in the deck
are distinct, then the multiplicity of the eigenvalue $\frac{j}{n}$
is the number of permutations of $n$ objects with $j$ fixed points.
From the eigenbasis formula of Theorem \ref{thm:tob2r-evectors-cocommutative},
one obtains explicit expectation formulae for a variety of ``pattern
functions'' on the bottom $n-j$ cards of the deck. For example,
the case $n-j=2$ gives the following analogue of \cite[Ex. 5.8]{hopfpowerchains}:
after $t$ top-to-random shuffles of a deck of $n$ distinct cards,
starting in ascending order, the probability that the bottommost card
has lower value than the card immediately above is $\left(1-\left(\frac{n-2}{n}\right)^{t}\right)\frac{1}{2}$. 
\item [{Restriction-then-induction,~or~box~moves}] \label{ex:snhookwalk}
The descent operators on the Schur basis of the algebra of symmetric
functions \cite[Sec. 7.10]{stanleyec2} induce the chains of \cite{jasonrepdownupchains},
on the irreducible representations of the symmetric group $\sn$.
In the case of $\ter_{n}$, this chain is: restrict to $\mathfrak{S}_{n-1}$,
induce back to $\sn$, and pick an irreducible constituent with probability
proportional to the dimension of its isotypic component. Because the
Littlewood-Richardson coefficients involved in this case are particularly
simple, this chain has a neat interpretation in terms of partition
diagrams: remove a random box using the hook walk of \cite{hookwalk},
then add a random box according to the complementary hook walk of
\cite{hookwalk2}. See \cite[Sec. 3.2.1]{hopfchainlift} for a detailed
derivation.

The unique stationary distribution of this chain is the ubiquitous
Plancherel measure $\pi(x)=\frac{(\dim x)^{2}}{n!}$. The eigenfunctions
show that, after $t$ moves, the expected character ratio $\frac{\chi(\sigma)}{\chi(\id)}$
on an $n'$-cycle $\sigma$ is $\left(\frac{n-n'}{n}\right)^{t}$
times its initial value. \cite{randomcharratios} uses similar results
to study central limit theorems for character ratios.

\item [{Rock-breaking~and~coupon~collection}] \label{ex:rockchipping}
Work in the elementary basis of the algebra of symmetric functions.
The states of the descent operator chains are partitions, viewed as
a multiset of integers recording the sizes of a collection of rocks.
\cite[Sec. 4]{hopfpowerchains} analysed the chain where each rock
breaks independently according to the binomial distribution. The $\ter_{n}$
chain is as follows: pick a rock in the collection with probability
proportional to its size, and chip off a piece of size 1 from this
rock. As noted in \cite{couponcollectioncompletegraphs}, starting
the $\ter_{n}$ chain from one single rock gives a rephrasing of the
classical coupon-collection problem: the number of chipped-off pieces
correspond to the different coupons collected, and the size of the
large rock corresponds to the number of uncollected coupons.

It is clear that both binomial-breaking and chipping have the same
stationary distribution - concentrated at rocks all of size 1. Indeed,
because the algebra of symmetric functions is both commutative and
cocommutative, both chains have a common right eigenbasis. One consequence
for the $\ter_{n}$ chain (analogous to \cite[Cor. 4.10]{hopfpowerchains})
reads: after $t$ steps starting from $\lambda$, the probability
that there remains a rock of size at least $n'>1$, is at most $\left(\frac{|\lambda|-n'}{|\lambda|}\right)^{t}\sum_{i}\binom{\lambda_{i}}{n'}$.
Or, in the coupon-collection formulation, the probability of still
needing $n'$ or more of the $n$ different coupons is at most $\left(\frac{n-n'}{n}\right)^{t}\binom{n}{n'}$.

\item [{Mining}] \label{ex:subwordcxs} The operator $\ter_{n}$ on the
Hopf algebra of subword complexes \cite{subwordcxs} gives rise to
an intriguing variant of rock-chipping, as sketched in \cite[App. B]{subwordcxs}.
Here is a simplified version which sprouted from discussion with those
authors. The states of this chain are collections of ``gems'', where
each gem is a distinct direction in $\mathbb{R}^{n}$ (or an element
of the projective space $\mathbb{P}^{n}$). The linear dependences
between these directions indicate which gems are ``entangled'' in
rock and must be ``mined''. At each time step, pick a gem $v$ in
the collection to mine, and pick a complementary subspace $W$ spanned
by some other gems in the collection. Then remove from the collection
all gems which are not in $v\cup W$. (The probability of selecting
the ``2-flat decomposition'' $(v,W)$ is proportional to the number
of bases of $\mathbb{R}^{n}$ using the gems in $v\cup W$ - we aim
to find a more natural interpretation in future work.)

This chain is absorbing at any state with $n$ linearly independent
gems, and there are usually multiple such states. For example, starting
with three non-collinear gems in $\mathbb{R}^{2}$, any pair of them
is a possible absorbing state. This is the first instance of a descent
operator Markov chain with multiple absorbing states, which merits
further investigation.

\item [{Phylogenetic~trees}] \label{ex:planarbinarytrees} Following \cite{phylogenetictreesusan},
a phylogenetic tree is a rooted tree on labelled vertices, recording
the ancestry connections of several biological species. Since much
ancestry data is conjectural, it is natural to consider random walks
on the set of such trees. In the simplest models, each species has
at most two descendants, and this can be represented by complete binary
trees (each internal vertex has exactly two children, and only internal
vertices are labelled). If, in addition, left and right children are
distinguished (i.e. the tree is planar), then these trees form a basis
for (a labelled version of) $YSym$, the graded dual to the Loday-Ronco
Hopf algebra \cite{lodayroncotrees,lrstructure}.

The $\ter_{n}$ chain on this algebra is on such trees with $n$ internal
vertices: remove the leftmost leaf and combine its parent with its
sibling, keeping the label of the sibling. Next, uniformly choose
a leaf, give it the label of the parent, and add to it two children.
This is a variant of the chain of \cite{treechainaldous,treechainaldous2},
where we have restricted the edge contraction and insertion operations
to specific edges. Figure \ref{fig:lrtree-chain} demonstrates the
possible moves from a particular state in the case $n=3$.

The unique stationary distribution is the number of ways to label
the vertices with $\{1,\dots,n\}$, each label occuring once, such
that a parent has a smaller label than both its children. Since this
algebra is not commutative, Theorem \ref{thm:tob2r-evectors-cocommutative}
provides many right eigenfunctions but not a complete basis. \cite{lodayroncochains}
proved using one eigenfunction that, after $t$ steps, the probability
of the root having no right children is bounded above by $\left(1+\left(\frac{n-2}{n}\right)^{t}\right)\frac{1}{2}$.

Since $YSym$ is a quotient of the Malvenuto-Reutenauer Hopf algebra
of permutations via taking the decreasing tree, by \cite[Th. 3.6]{hopfchainlift},
the unlabelled version of this chain on trees precisely records the
decreasing trees for permutations under the to-do list chain of Example
\ref{ex:fqsym} / Section \ref{sec:fqsym}. 

\end{description}
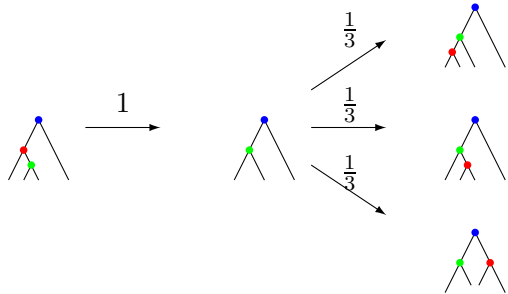
\begin{figure}
\setlength{\unitlength}{1mm} \begin{center} \begin{picture} (58,45) (5,180)  \put(0,190){\lrtreeone} \put(12,195){\vector(1, 0){10}} \put(16,197){$1$} \put(30,190){\lrtreetwo} \put(42,200){\vector(3, 2){10}} \put(46,207){$\frac{1}{3}$} \put(58,205){\lrtreethree} \put(42,195){\vector(1, 0){10}} \put(46,197){$\frac{1}{3}$} \put(58,190){\lrtreefour} \put(42,190){\vector(3, -2){10}} \put(46,188){$\frac{1}{3}$} \put(58,175){\lrtreefive} \end{picture} \end{center}\protect\caption{\label{fig:lrtree-chain}One step of the $\ter_{3}$ chain on phylogenetic
trees }
\end{figure}

Acknowledgements: The author would like to thank Persi Diaconis,
Franco Saliola and Nantel Bergeron for much help throughout this project,
both with the mathematics and the exposition. Thank you also to Sami
Assaf, Thomas Lam, Jean-Christophe Novelli and Nathan Williams for
inspiring this line of research, and to Marcelo Aguiar, Francois Bergeron,
Cesar Ceballos, Alex Chan, Eric Chan, Darij Grinberg, Zachary Hamaker,
Matthieu Josuat-Verges, Servando Pineda and Stephanie van Willengenburg
for many good ideas and helpful discussions. Thanks to Ian Grojnowski
for his help with the first proof of Theorem \ref{thm:evalues}, and
to the anonymous referee of \cite{cpmcfpsac} for supplying most of
the second proof of the same theorem. Some of the ideas in Section
\ref{sec:cktrees} arose in discussion with Daniel Ahlberg, Amir Dembo,
Lea Popovic, Amir Sepehri, Yannic Vargas and James Zhao.

Since this is the author's main piece of work during her postdoctoral
stay in Montreal, this seems a good place to thank all the members
and visitors of LaCIM, and her colleagues and students at McGill,
for an unimaginably wonderful two years. A special mention goes to
Mathieu Guay-Paquet and Franco Saliola for their invaluable advice
and support, both mathematical and personal.

\section{Background and Notation\label{sec:Background}}

Since this paper is intended for multiple audiences, this section
provides both probabilistic (Sections \ref{sub:Markov-chain-background}
and \ref{sub:doobtransform}) and algebraic (Sections \ref{sub:cha-background}
and \ref{sub:Descent-operators-background}) preliminaries. Two points
are of particular interest: Section \ref{sub:doobtransform} outlines
an obscure use of the Doob transform to create transition matrices
out of non-negative matrices, which underlies the construction of
the descent operator chains. Section \ref{sub:Descent-operators-background}
is a non-standard treatment of the descent operators, in order to
facilitate our new probabilistic connections. Readers familiar with
these operators are encouraged to skim this section nonetheless.

\subsection{Linear algebra notation\label{sub:Linear-algebra-background}}

Given a matrix $K$, let $K(x,y)$ denote its entry in row $x$, column
$y$, and write $K^{T}$ for the transpose of $K$, so $K^{T}(x,y)=K(y,x)$.

For a vector space $V$ with basis $\calb$, and a linear map $\T:V\rightarrow V$
, write $\left[\T\right]_{\calb}$ for the matrix of $\T$ with respect
to $\calb$, satisfying
\[
\T(x)=\sum_{y\in\calb}\left[\T\right]_{\calb}(y,x)y.
\]
$V^{*}$ is the \emph{dual vector space} to $V$, the set of linear
functions from $V$ to $\mathbb{R}$. (Because of the probability
applications, take $\mathbb{R}$ to be the ground field of all vector
spaces.) Its natural basis is $\calbdual:=\left\{ x^{*}|x\in\calb\right\} $,
where $x^{*}$ satisfies $x^{*}(x)=1$, $x^{*}(y)=0$ for all $y\in\calb$,
$y\neq x$. The \emph{dual map} to $\T:V\rightarrow V$ is the linear
map $\T^{*}:V^{*}\rightarrow V^{*}$ satisfying $(\T^{*}f)(v)=f(\T v)$
for all $v\in V,f\in V^{*}$. Dualising a linear map is equivalent
to transposing its matrix: $\left[\T^{*}\right]_{\calbdual}=\left[\T\right]_{\calb}^{T}$.

\subsection{Markov chains\label{sub:Markov-chain-background}}

All Markov chains in this work are in discrete time, are time-homogeneous
and have a finite state space $\Omega$. Hence they are each described
by an $|\Omega|$-by-$|\Omega|$ \emph{transition matrix} $K$. We
follow the probability community's convention that the rows index
the source state and the columns index the destination state, so the
probability of moving from $x$ to $y$ is $K(x,y)$. (Combinatorialists
sometimes take the opposite convention, for example \cite{rtrivialmonoids}.)
Note that a matrix $K$ specifies a Markov chain in this manner if
and only if $K(x,y)\geq0$ for all $x,y\in\Omega$, and $\sum_{y\in\Omega}K(x,y)=1$
for each $x\in\Omega$. We refer the reader to \cite{markovmixing}
for the basics of Markov chain theory.

This paper is primarily interested in the \emph{stationary distributions}
of a Markov chain, and its \emph{left and right eigenfunctions}. These
are functions $\pi,\f,\g:\Omega\rightarrow\mathbb{R}$ satisfying
respectively 
\[
\sum_{x\in\Omega}\pi(x)K(x,y)=\pi(y),\quad\sum_{x\in\Omega}\g(x)K(x,y)=\beta\g(y),\quad\sum_{y\in\Omega}K(x,y)\f(y)=\beta\f(x).
\]
(For brevity, we will occasionally write $\beta$-eigenfunction to
mean that the eigenvalue is $\beta$, and similarly for eigenvectors
of linear maps.) So the stationary distributions are precisely the
distributions that are left 1-eigenfunctions. \cite[Sec. 2.1]{hopfpowerchains}
lists many applications of both left and right eigenfunctions; this
work will concentrate on their Use A, which is immediate from the
definitions of expectation and eigenfunction: 
\begin{prop}[Expectations from right eigenfunctions]
\label{prop:eigenfunction-expectation}  The expected value of a
right eigenfunction $\f$ with eigenvalue $\beta$ is 
\[
\Expect(\f(X_{t})|X_{0}=x_{0}):=\sum_{y\in\Omega}K^{t}(x_{0},y)\f(y)=\beta^{t}\f(x_{0}).
\]
\qed
\end{prop}
In Sections \ref{sub:tob2r-efns} and \ref{sub:tab2r}, the computation
of eigenfunctions boils down to finding stationary distributions of
certain related chains, and our main tool for doing so will be detailed
balance:
\begin{lem}
\label{lem:detailed-balance} \cite[Prop. 1.19]{markovmixing} Let
$\{X_{t}\}$ be a Markov chain on the state space $\Omega$ with transition
matrix $K$. If a distribution $\pi$ on $\Omega$ is a solution to
the \emph{detailed balance equation} $\pi(x)K(x,y)=\pi(y)K(y,x)$
for each $x,y\in\Omega$, then $\pi$ is a stationary distribution.\end{lem}
\begin{proof}
\[
\sum_{x\in\Omega}\pi(x)K(x,y)=\sum_{x\in\Omega}\pi(y)K(y,x)=\pi(y)\sum_{x\in\Omega}K(y,x)=\pi(y).
\]

\end{proof}
Note that the detailed balance condition is far from necessary for
a distribution to be stationary; there are plenty of Markov chains
which are not \emph{reversible}, meaning they admit no solutions to
their detailed balance equations. In fact, the descent operator chains
of this paper are in general not reversible; it is only their related
chains, for computing eigenfunctions, which are reversible.

\subsection{The Doob $h$-transform\label{sub:doobtransform}}

This section briefly explains a very general method of constructing
a transition matrix out of a linear map satisfying certain positivity
conditions. Section \ref{sub:construction} will exposit this construction
in full detail in the case where these linear maps are descent operators,
so readers interested solely in descent operator chains should feel
free to skip this section.

The Doob $h$-transform is a very general tool in probability, used
to condition a process on some event in the future \cite{doobtransformoriginal}.
The simple case of relevance here is conditioning a (finite, discrete-time)
Markov chain on non-absorption. The Doob transform constructs the
transition matrix of the conditioned chain out of the transition probabilities
of the original chain between non-absorbing states, or, equivalently,
out of the original transition matrix with the rows and columns for
absorbing states removed. As observed in the multiple references below,
the same recipe essentially works for an arbitrary non-negative matrix.
In the case where this matrix comes from a linear operator, the transform
has an elegant interpretation in terms of a rescaling of the basis.
\begin{thm}[Doob $h$-transform for linear maps]
\label{thm:doob-transform}  \cite[Th. 3.1.1]{mythesis} \cite[Def. 8.11, 8.12]{doobtransformbook} \cite[Sec.17.6.1]{markovmixing}
Let $V$ be a finite-dimensional vector space with basis $\calb$,
and $\T:V\rightarrow V$ be a linear map for which $K:=\left[\T\right]_{\calb}^{T}$
has all entries non-negative. Suppose there is an eigenvector $\eta\in V^{*}$
of the dual map $\T^{*}:V^{*}\rightarrow V^{*}$, with eigenvalue
1, taking only positive values on $\calb$. Then 
\[
\hatk(x,y):=\frac{\eta(y)}{\eta(x)}K(x,y)
\]
defines a transition matrix. Equivalently, $\hatk:=\left[\T\right]_{\hatcalb}^{T}$,
where $\hatcalb:=\left\{ \frac{x}{\eta(x)}|x\in\calb\right\} $.
\end{thm}
The Markov chain with transition matrix $\hatk$ above is called the
\emph{Markov chain on $\calb$ driven by $\T$}.
\begin{proof}
First note that $K=\left[\T^{*}\right]_{\calbdual}$ by definition,
so $\T^{*}\eta=\eta$ translates to $\sum_{y}K(x,y)\eta(y)=\eta(x)$.
(Functions satisfying this latter condition are called \emph{harmonic},
hence the name $h$-transform.) Since $\eta(x)>0$ for all $x$, it
is clear that $\hatk(x,y)\ge0$. It remains to show that the rows
of $\hatk$ sum to 1: 
\[
\sum_{y}\hatk(x,y)=\frac{\sum_{y}K(x,y)\eta(y)}{\eta(x)}=\frac{\eta(x)}{\eta(x)}=1.
\]

\end{proof}
The function $\eta:V\rightarrow\mathbb{R}$ above is the \emph{rescaling
function}. Different choices of $\eta$ for the same linear operator
$\T$ can lead to different Markov chains, but the notation suppresses
the dependence on $\eta$ because, for $\T$ a descent operator, there
is a canonical choice of $\eta$ (Lemma \ref{lem:eta} below).The
assumption that $\eta$ has eigenvalue 1 can be easily relaxed by
scaling the transformation $\T$, see \cite[Th. 2.3]{hopfchainlift}.
We choose to impose this assumption here as it unclutters the eigenfunction
formulae in Proposition \ref{prop:doobtransform-efns} below.

The main advantage of fashioning a transition matrix using the Doob
transform, as opposed to some other manipulation on $K$ (such as
scaling each row separately) is that the diagonalisation of the Markov
chain is equivalent to identifying the eigenvectors of $\T$ and its
dual $\T^{*}$: 
\begin{prop}[Eigenfunctions for Doob transform chains]
\label{prop:doobtransform-efns}   \cite[Prop. 3.2.1]{mythesis}  \cite[Lemma 4.4.1.4]{zhou}  \cite[Lem. 2.11]{doobnotes}Let
$\{X_{t}\}$ be the Markov chain on $\calbn$ driven by $\T:V\rightarrow V$
with rescaling function $\eta$. Then:
\begin{description}
\item [{(L)}] The left $\beta$-eigenfunctions $\g:\calbn\rightarrow\mathbb{R}$
for $\{X_{t}\}$ are in bijection with the $\beta$-eigenvectors $g\in V$
of $\T$, through the vector space isomorphism 
\[
\g(x):=\mbox{coefficient of }x\mbox{ in }\eta(x)g.
\]

\item [{(R)}] The right $\beta$-eigenfunctions $\f:\calbn\rightarrow\mathbb{R}$
for $\{X_{t}\}$ are in bijection with the $\beta$-eigenvectors $f\in V^{*}$
of the dual map $\T^{*}$, through the vector space isomorphism 
\[
\f(x):=\frac{1}{\eta(x)}f(x).
\]

\end{description}

\qed

\end{prop}
\begin{rem*}
In the Markov chain literature, the term ``left eigenvector'' is
often used interchangeably with ``left eigenfunction'', but this
work will be careful to make a distinction between the eigenfunction
$\g:\calb\rightarrow\mathbb{R}$ and the corresponding eigenvector
$g\in V$ (and similarly for right eigenfunctions).
\end{rem*}

\subsection{Combinatorial Hopf algebras\label{sub:cha-background}}

The application of Hopf algebras to combinatorics originated from
\cite{jonirota}; much general theory have since been developed \cite{schmitt,hivertcspolynomialrealisation,abs,hopfmonoids,antipodeformula},
and a plethora of examples constructed and analysed in detail \cite{tableauxhopfalg,grhiscommutative,parkingfnhopfalg,vincenthopfalg}.
Loosely speaking, a combinatorial Hopf algebra is a graded vector
space $\calh=\bigoplus_{n=0}^{\infty}\calh_{n}$ with a basis $\calb=\amalg_{n}\calbn$
indexed by combinatorial objects, such as graphs, trees or permutations.
The grading reflects the ``size'' of the objects. $\calh$ is \emph{connected}
in that $\dim\calh_{0}=1$, spanned by a unique empty object. There
is a product $m:\calh_{i}\otimes\calh_{j}\rightarrow\calh_{i+j}$
and a coproduct $\Delta:\calh_{n}\to\bigoplus_{i=0}^{n}\calh_{i}\otimes\calh_{n-i}$
encoding respectively how to combine two objects and to break an object
into two. These must satisfy associativity, compatibility and various
other axioms; see the survey \cite{vicreinernotes} for details.

Many families of combinatorial objects (graphs, trees) have a single
member of size 1, so $\calh_{1}$ is often one-dimensional. In such
cases, $\bullet$ will denote this sole object of size 1, so $\calb_{1}=\{\bullet\}$.

One simple, instructive, example of a combinatorial Hopf algebra is
the shuffle algebra of \cite{shufflealg}, whose associated Markov
chains describe the cut-and-interleave card-shuffles of \cite{cppriffleshuffle}.
\begin{example}
\label{ex:shufflealg-bg} The \emph{shuffle algebra} $\calsh$, as
a vector space, has basis the set of all words in the letters $\{1,2,\dots,N\}$
(for some $N$, whose exact value is often unimportant). View the
word $\llbracket w_{1}\dots w_{n}\rrbracket$ as the deck of cards
with card $w_{1}$ on top, card $w_{2}$ second from the top, and
so on, so card $w_{n}$ is at the bottom (the bracket notation is
non-standard). The degree of a word is its number of letters, i.e.
the number of cards in the deck. The product of two words, also denoted
by $\shuffle$, is the sum of all their interleavings (with multiplicity),
and the coproduct is deconcatenation, or cutting the deck. For example:
\begin{align*}
m(\llbracket15\rrbracket\otimes\llbracket52\rrbracket)=\llbracket15\rrbracket\shuffle\llbracket52\rrbracket & =2\llbracket1552\rrbracket+\llbracket1525\rrbracket+\llbracket5152\rrbracket+\llbracket5125\rrbracket+\llbracket5215\rrbracket;\\
\Delta(\llbracket316\rrbracket) & =\llbracket\rrbracket\otimes\llbracket316\rrbracket+\llbracket3\rrbracket\otimes\llbracket16\rrbracket+\llbracket31\rrbracket\otimes\llbracket6\rrbracket+\llbracket316\rrbracket\otimes\llbracket\rrbracket.
\end{align*}

\end{example}
Given a graded connected Hopf algebra $\calh=\bigoplus_{n\geq0}\calhn$,
the symmetry of the Hopf axioms allows the definition of a Hopf structure
on the (graded) dual vector space $\calhdual:=\oplus_{n\geq0}\calhndual$:
for $f,g\in\calhdual$, set
\[
m(f\otimes g)(x):=(f\otimes g)(\Delta x),\quad\Delta(f)(w\otimes z)=f(wz),
\]
with $x,z,w\in\calh$. (Here, $(f\otimes g)(a\otimes b)=f(a)g(b)$.)
\begin{example}
\label{ex:freeassalg-bg} The dual of the shuffle algebra $\calsh$
is the \emph{free associative algebra} $\calsh^{*}$, whose basis
is also indexed by words in the letters $\{1,2,\dots,N\}$. The product
in $\calsh^{*}$ is concatenation, and the coproduct is ``deshuffling'';
for example:  
\begin{align*}
m(\llbracket15\rrbracket\otimes\llbracket52\rrbracket) & =\llbracket1552\rrbracket;\\
\Delta(\llbracket316\rrbracket) & =\llbracket\rrbracket\otimes\llbracket316\rrbracket+\llbracket3\rrbracket\otimes\llbracket16\rrbracket+\llbracket1\rrbracket\otimes\llbracket36\rrbracket+\llbracket6\rrbracket\otimes\llbracket31\rrbracket\\
 & \phantom{=}+\llbracket31\rrbracket\otimes\llbracket6\rrbracket+\llbracket36\rrbracket\otimes\llbracket1\rrbracket+\llbracket16\rrbracket\otimes\llbracket3\rrbracket+\llbracket316\rrbracket\otimes\llbracket\rrbracket.
\end{align*}
The associated Markov chains are the (unweighted) pop-shuffles of
\cite[Sec. 2]{hyperplanewalk}, the time-reversal of the cut-and-interleave
shuffles. First, take some cards out of the deck to form a separate
pile, keeping their relative order. Repeat this a few times, then
place the piles one on top of another. This viewpoint is useful for
the proof of the spectrum of descent operators (see Remark 2 after
Lemma \ref{lem:symlemma}).
\end{example}
Because many combinatorial objects have ``symmetric'' assembling
or breaking rules, many combinatorial Hopf algebras are \emph{commutative}
($wz=zw$ for all $w,z\in\calh$) or \emph{cocommutative} (if $\Delta(x)=\sum_{i}w_{i}\otimes z_{i}$,
then $\sum_{i}w_{i}\otimes z_{i}=\sum_{i}z_{i}\otimes w_{i}$). For
example, shuffle algebra is commutative but not cocommutative, and
dualising means that the free associative algebra is noncommutative
and cocommutative. The descent operators are better behaved on such
algebras, so under a (co)commutativity hypothesis, stronger results
hold - for example, the Markov chains are diagonalisable (Theorem
\ref{thm:evalues}).

The eigenvectors of our Markov chains will be constructed from \emph{primitive}
elements: $x\in\calh$ satisfying $\Delta(x)=1\otimes x+x\otimes1.$
It is easy to show that the primitive elements of $\calh$ form a
subspace and a Lie algebra. Write $\calp$ for a basis of this subspace.
Such a basis has been computed for many combinatorial Hopf algebras
 \cite[Sec. 5]{lrstructure} \cite[Sec. 3.1.3]{fisher}.

\subsection{Descent operators\label{sub:Descent-operators-background}}

Here is a non-standard definition of descent operators, which will
be useful for our probabilistic applications. 
\begin{defn}
\label{defn:descent-operator} Let $\calh$ be a graded Hopf algebra.
\begin{enumerate}
\item Given a \emph{weak-composition} $D=(d_{1},\dots,d_{l(D)})$ of $n$
(i.e. non-negative integers $d_{i}$ summing to $n$), define the
\emph{refined coproduct} $\Delta_{D}:\calhn\rightarrow\calh_{d_{1}}\otimes\dots\otimes\calh_{d_{l(D)}}$
to be the composition of the iterated coproduct $\Delta^{[l(D)]}:=(\Delta\otimes\id^{\otimes l(D)-1})\circ\dots(\Delta\otimes\id\otimes\id)\circ(\Delta\otimes\id)\circ\Delta$
followed by the projection $\calh^{\otimes l(D)}\rightarrow\calh_{d_{1}}\otimes\dots\otimes\calh_{d_{l(D)}}$.
\item The \emph{descent operators} are the linear combinations of the composite
maps $m\Delta_{D}:\calhn\rightarrow\calhn$ (abusing notation here
and writing $m:\calh^{l}\rightarrow\calh$ for the multiplication
of arbitrarily many elements).
\item Given a probability distribution $P$ on the set of weak-compositions
of $n$, define $m\Delta_{P}:\calhn\rightarrow\calhn$ as 
\[
m\Delta_{p}:=\sum\frac{P(D)}{\binom{n}{D}}m\Delta_{D},
\]
where $\binom{n}{D}$ is the multinomial coefficient $\binom{n}{d_{1}\dots d_{l(D)}}$.
\end{enumerate}
\end{defn}
On a combinatorial Hopf algebra, $\Delta_{D}$ captures the notion
of breaking an object into pieces of sizes $d_{1},d_{2},\dots,d_{l(D)}$.
So each step of the Markov chain driven by $m\Delta_{p}$ first picks
a weak-composition $D$ according to the distribution $P$, then breaks
the current state into pieces of sizes $d_{1},d_{2},\dots,d_{l(D)}$,
then reassembles these pieces (see Theorem \ref{thm:cppchain} for
a precise statement). For example, when $P$ is the binomial distribution
on weak-compositions with two parts, and zero on all other weak-compositions,
the map $m\Delta_{P}$ is simply $m\Delta$. On the shuffle algebra,
this describes the  Gilbert-Shannon-Reeds model of riffle-shuffling,
as analysed in \cite{originalriffleshuffle}: cut the deck into two
piles binomially, then combine them by repeatedly dropping the bottommost
card from either pile, chosen with probability proportional to the
current pile size.
\begin{example}
\label{ex:descent-operator-shuffle-freeassalg} In the shuffle algebra
$\calsh$,
\begin{align*}
\Delta_{1,1,2}\llbracket1552\rrbracket & =\llbracket1\rrbracket\otimes\llbracket5\rrbracket\otimes\llbracket52\rrbracket.\\
\Delta_{2,0,2}\llbracket1552\rrbracket & =\llbracket15\rrbracket\otimes\llbracket\rrbracket\otimes\llbracket52\rrbracket.\\
m\Delta_{1,3}\llbracket1552\rrbracket & =m(\llbracket1\rrbracket\otimes\llbracket552\rrbracket)=\llbracket1552\rrbracket+\llbracket5152\rrbracket+\llbracket5512\rrbracket+\llbracket5521\rrbracket.
\end{align*}
And in the free associative algebra $\calsh^{*}$,
\begin{align*}
m\Delta_{1,2}\llbracket316\rrbracket & =m(\llbracket3\rrbracket\otimes\llbracket16\rrbracket+\llbracket1\rrbracket\otimes\llbracket36\rrbracket+\llbracket6\rrbracket\otimes\llbracket31\rrbracket)\\
 & =\llbracket316\rrbracket+\llbracket136\rrbracket+\llbracket631\rrbracket.
\end{align*}

\end{example}
The notation $m\Delta_{D}$ is from \cite{hopfmonoids} and is recent;
the same operator is written $B_{D}$ in \cite{descentoperators},
and $\xi_{D}$ in \cite{descentoperatordarij}. \cite[Prop. 88]{descentoperatorhopfmonoids}
gives a version for Hopf monoids. These and other sources mostly consider
$m\Delta_{D}$ when $D$ is a (strict) composition (i.e. no $d_{i}$
is zero). Indeed, on a graded connected Hopf algebra, this is sufficient,
since removing parts of size 0 from a weak-composition $D$ does not
change the map $m\Delta_{D}$. However, the probability distributions
are more natural if parts are allowed to have size 0.

\begin{rems*}
$ $
\begin{enumerate}[label=\arabic*.]
\item Every \emph{positive descent operator} (that is, a non-negative linear
combination of $m\Delta_{D}$) is a multiple of $m\Delta_{P}$ for
some probability distribution $P$. Hence the results in Sections
\ref{sec:Markov-Chains-from-descent-operators} and \ref{sec:partsof1}
concerning $m\Delta_{P}$ have analogues for arbitrary positive descent
operators. 
\item The dual of a descent operator $m\Delta_{P}$ is simply the same operator
on the dual Hopf algebra. This observation will be useful for deriving
right eigenfunctions of the associated Markov chains.
\end{enumerate}
\end{rems*}The descent operators are so named because, on a commutative
or cocommutative Hopf algebra, their composition is equivalent to
the product on Solomon's descent algebra \cite{solomondescentalg}.
For this work, it will be more useful to express the latter as the
internal product $\cdot$ in the algebra of noncommutative symmetric
functions $\sym$ \cite{ncsym}. Let $\cppmap:\sym\rightarrow\End(\calh)$
denote the linear map sending the complete noncommutative symmetric
functions $S^{D}$ to the descent operator $m\Delta_{D}$. (Here,
$\End(\calh)$ is the algebra of linear maps $\calh\rightarrow\calh$;
these maps need not respect the product or coproduct.)
\begin{prop}[Compositions of descent operators]
\label{prop:descentoperator-composition}   \cite[Th. II.7]{descentoperators} \cite[Prop. 5.1]{ncsym}Let
$\calh$ be a graded connected Hopf algebra, and $\cppmap:\sym\rightarrow\End(\calh)$
be the linear map with $\cppmap(S^{D})=m\Delta_{D}$.
\begin{enumerate}
\item If $\calh$ is commutative, then, for any $F,G\in\sym$, the composite
of their images is $\cppmap(F)\circ\cppmap(G)=\cppmap(G\iprod F)$.
\item If $\calh$ is cocommutative, then, for any $F,G\in\sym$, the composite
of their images is $\cppmap(F)\circ\cppmap(G)=\cppmap(F\iprod G)$.
\end{enumerate}

In particular, if $\calh$ is commutative (resp. cocommutative), then
the set of descent operators on $\calh$ is closed under composition.
Indeed, 
\[
m\Delta_{D}\circ m\Delta_{D'}=\sum_{M}m\Delta_{D''(M)},
\]
where the sum runs over all $l(D)$-by-$l(D')$ matrices $M$ of non-negative
integers, with row $i$ summing to $d_{i}$ and column $j$ summing
to $d'_{j}$. And $D''(M)$ is the weak-composition formed from all
the entries of $M$, by reading down each column from the leftmost
column to the rightmost column (resp. by reading left to right across
each row from the top row to the bottom row): 
\begin{align*}
D''(M) & =(M(1,1),M(2,1),\dots,M(l(D),1),M(1,2)\dots,M(1,l(D')),\dots,M(l(D),l(D'))\\
\mbox{(resp.}\quad D''(M) & =(M(1,1),M(1,2),\dots,M(1,l(D')),M(2,1)\dots,M(l(D),1),\dots,M(l(D),l(D'))\quad\mbox{)}.
\end{align*}
 \qed

\end{prop}
(The case $\calh=\calsh$, concerning compositions of cut-and-interleave
shuffles, was proved in \cite[Th. 5.1]{cppriffleshuffle}.) Consequently,
the problem of finding eigenvalues and eigenvectors of descent operators
is closely connected to the determination of orthogonal idempotents
of subalgebras (under the internal product) of $\sym$ (see the remark
after Lemma \ref{lem:idempotents}).
\begin{rem*}
The use of $\iprod$ to denote the internal product is non-standard,
chosen to evoke the standard symbol $\circ$ for composition on $\End(\calh)$.
The usual notation of $*$ is confusing here, as it usually indicates
\emph{convolution product} ($\T*\T':=m(\T\otimes\T')\Delta$), which
corresponds under $\theta$ to the external product on $\sym$, not
the internal product.
\end{rem*}

\section{Markov Chains from Descent Operators\label{sec:Markov-Chains-from-descent-operators}}

Section \ref{sub:construction} applies the Markov chain construction
of the previous section to the descent operators $m\Delta_{P}$ of
Section \ref{sub:Descent-operators-background}. Sections \ref{sub:evalue}
and \ref{sub:Stationary-distribution} give respectively the spectrum
and stationary distributions of these chains, and Section \ref{sub:Absorption}
relates the quasisymmetric invariants of \cite{abs} to the absorption
probabilities of certain chains.

\subsection{Construction \label{sub:construction}}

Recall that, for a probability distribution $P$ on weak-compositions
of a fixed integer $n$, the descent operator $m\Delta_{P}:\calhn\rightarrow\calhn$
is
\[
m\Delta_{p}=\sum\frac{P(D)}{\binom{n}{d_{1}\dots d_{l(D)}}}m\Delta_{D}.
\]
To apply the Doob $h$-transform (Theorem \ref{thm:doob-transform})
to the linear map $m\Delta_{P}$, it is first necessary to find a
basis $\calbn$ of $\calhn$ with respect to which $m\Delta_{P}$
has a nonnegative matrix. One stronger condition that is more natural,
at least for combinatorial Hopf algebras, is to mandate that $\calb$
has non-negative product and coproduct structure constants in the
relevant degrees - this is the essence of conditions i and ii in the
definition below. As for condition iii: informally, this insists that
every combinatorial object indexing the basis may be broken into pieces
of size 1. Such a restriction is necessary since Lemma \ref{lem:eta}
will show that $\eta$ is a suitable rescaling function (in the sense
of Theorem \ref{thm:doob-transform}).
\begin{defn}
\label{defn:state-space-basis} Let $\calh=\bigoplus_{n\geq0}\calhn$
be a graded connected Hopf algebra over $\mathbb{R}$ with each $\calhn$
finite-dimensional. Let $D=(d_{1},\dots,d_{l(D)})$ be a weak-composition
of some fixed integer $n$. A basis $\calb=\amalg_{n\geq0}\calbn$
of $\calh$ is a \emph{state space basis for $D$ }(or \emph{for $m\Delta_{D}$})
if:
\begin{enumerate}
\item for all $z_{1}\in\calb_{d_{1}},z_{2}\in\calb_{d_{2}},\dots,z_{l(D)}\in\calb_{d_{l(D)}}$,
their product is $z_{1}z_{2}\dots z_{l(D)}=\sum_{y\in\calb_{n}}\xi_{z_{1},\dots,z_{l(D)}}^{y}y$
with $\xi_{z_{1},\dots,z_{l(D)}}^{y}\geq0$ (non-negative product
structure constants);
\item for all $x\in\calb_{n}$, its coproduct is $\Delta_{D}(x)=\sum_{z_{i}\in\calb_{d_{i}}}\eta_{x}^{z_{1},\dots,z_{l(D)}}z_{1}\otimes z_{2}\otimes\dots\otimes z_{l(D)}$
with $\eta_{x}^{z_{1},\dots,z_{l(D)}}\geq0$ (non-negative coproduct
structure constants);
\item for all $x\in\calb$, the function 
\[
\eta(x):=\mbox{sum of coefficients (in the }\calb_{1}\otimes\dots\otimes\calb_{1}\mbox{ basis) of }\Delta_{1,\dots,1}(x)
\]
evaluates to a positive number.
\end{enumerate}

If $P$ is a probability distribution on weak-compositions of $n$,
then a basis $\calb=\amalg_{n\geq0}\calbn$ of $\calh$ is a \emph{state
space basis for $P$ }(or \emph{for $m\Delta_{P}$}) if it is a state
space basis for all $D$ with non-zero probability under $P$.

\end{defn}
Note that, if all structure constants of $\calb$ are non-negative
regardless of degree, then $\calb$ is a state space basis for all
distributions $P$. (It is in fact sufficient to check that all $\xi_{z_{1},z_{2}}^{y}$
and all $\eta_{x}^{z_{1},z_{2}}$ are non-negative, because of associativity
and coassociativity, see \cite[Lem. 4.2.1]{mythesis}.) In this case,
\cite[Th. 4.3.7.i]{mythesis} shows that condition iii is equivalent
to $\calb$ not containing any primitive elements of degree greater
than 1. ($\calh$ may contain primitive elements of any degree, so
long as those of degree greater than one are not in the basis $\calb$.)
In general, the absence of primitives in the basis $\calb$ is necessary
but not sufficient.
\begin{example}
\label{ex:t2r-basis} Let $P$ be the distribution that is conentrated
at $(1,n-1)$ - that is, $P((1,n-1))=1$, and $P(D)=0$ for all other
weak-compositions $D$. (Recall from the introduction that, on the
shuffle algebra, this distribution induces the top-to-random card-shuffle.)
Then conditions i and ii in Definition \ref{defn:state-space-basis}
simply require $m:\calh_{1}\otimes\calh_{n-1}\rightarrow\calh_{n}$
and $\Delta_{1,n-1}$ to have non-negative structure constants. (In
other words, the requirement $\xi_{c,z}^{y}\geq0$, $\eta_{x}^{c,z}\geq0$
is only for $z\in\calb_{n-1}$, $c\in\calb_{1}$, $x,y\in\calbn$.) 
\end{example}
All bases of Hopf algebras in this paper have all structure constants
non-negative. For examples which satisfy the conditions in Example
\ref{ex:t2r-basis} and yet have some negative structure constants,
see the plethora of ``schurlike'' bases in noncommutative symmetric
functions \cite{quasischurinncsym,immaculate}.
\begin{rem*}
If $\calh_{1}=\emptyset$ (so there are no combinatorial objects of
size 1), then, according to Definition \ref{defn:state-space-basis},
$\calh$ has no state space bases. However, it is still possible,
at least theoretically, to define descent operator Markov chains on
$\calh$. There are currently no known examples of such chains, so
we do not go into the technical details here - see the last paragraph
of \cite[Sec. 4.3]{mythesis}.
\end{rem*}
Having scrutinised the non-negativity condition on structure constants,
focus now on the function $\eta$, which rigorises the concept of
``number of ways to break into singletons''. It is usually a well-investigated
number: for example, for the irreducible representations of the symmetric
group (inducing Fulman's restriction-then-induction chain), it is
the dimension of the representation \cite[Ex. 4.4.3]{mythesis}. Proposition
\ref{lem:eta} below verifies that $\eta$ is indeed a possible rescaling
function for the Doob transform - in fact, \cite[Th. 4.3.7]{mythesis}
shows that, in some sense, this $\eta$ is the optimal rescaling function.
\begin{lem}
\label{lem:eta} Let $\calh=\bigoplus_{n\geq0}\calhn$ be a graded
connected Hopf algebra over $\mathbb{R}$ with each $\calhn$ finite-dimensional,
and $\calb_{1}$ a basis of $\calh_{1}$. The linear function $\eta:\calh\rightarrow\mathbb{R}$
with 
\[
\eta(x):=\mbox{sum of coefficients (in the }\calb_{1}\otimes\dots\otimes\calb_{1}\mbox{ basis) of }\Delta_{1,\dots,1}(x)
\]
is a $1$-eigenvector of the dual map to the descent operator $m\Delta_{P}:\calhn\rightarrow\calhn$,
for any probability distribution $P$.\end{lem}
\begin{proof}
Let $\bullet^{*}\in\calhdual_{1}$ denote the linear map on $\calh_{1}$
taking value $1$ on each element of $\calb_{1}$. (So, if $\calb_{1}=\{\bullet\}$,
then this map is the dual basis element $\bullet^{*}$, hence the
notation.) Since multiplication in $\calhdual$ is dual to the coproduct
on $\calh$, it is true that $\eta$ restricted to $\calhn$ is $(\bullet^{*})^{n}$.

As noted in Remark 2 of Section \ref{sub:Descent-operators-background},
the dual map to a descent operator $m\Delta_{P}:\calhn\rightarrow\calhn$
is the same descent operator on the dual Hopf algebra $\calhdual_{n}$.
So it suffices to show that $(\bullet^{*})^{n}$ is a 1-eigenvector
of $m\Delta_{P}:\calhndual\rightarrow\calhndual$. By linearity, this
will follow from $(\bullet^{*})^{n}$ being a $\binom{n}{D}$-eigenvector
of $m\Delta_{D}:\calhndual\rightarrow\calhndual$ for each weak-composition
$D$.

Write $l$ for the number of parts in $D$. As $\deg(\bullet^{*})=1$,
the iterated coproduct sends $\bullet^{*}$ to $\Delta^{[l]}(\bullet^{*})=\bullet^{*}\otimes1\otimes\dots\otimes1\ +\ 1\otimes\bullet^{*}\otimes1\otimes\dots\otimes1\ +\ \dots\ +\ 1\otimes\dots\otimes\bullet^{*}$,
i.e. the sum of $l$ terms, each with $l$ tensorands, one of which
is $\bullet^{*}$ and all others are 1. Because of the compatibility
of product and coproduct, 
\begin{align*}
\Delta^{[l]}((\bullet^{*})^{n}) & =\left(\bullet^{*}\otimes1\otimes\dots\otimes1\ +\ 1\otimes\bullet^{*}\otimes1\otimes\dots\otimes1\ +\ \dots\ +\ 1\otimes\dots\otimes\bullet^{*}\right)^{n}\\
 & =\sum_{i_{1},\dots,i_{l}}\binom{n}{i_{1}\dots i_{l}}\left(\bullet^{*}\right)^{i_{1}}\otimes\dots\otimes\left(\bullet^{*}\right)^{i_{l}}.
\end{align*}
Hence $\Delta_{D}((\bullet^{*})^{n})=\binom{n}{D}\left(\bullet^{*}\right)^{d_{1}}\otimes\dots\otimes\left(\bullet^{*}\right)^{d_{l}}$,
so $m\Delta_{D}((\bullet^{*})^{n})=\binom{n}{D}((\bullet^{*})^{n})$.
\end{proof}
So it is indeed possible to apply the Doob transform to $m\Delta_{P}$
in a state space basis, with this choice of $\eta$. 

To obtain a more intuitive interpretation of the Markov chain driven
by $m\Delta_{P}$, appeal to this description of the cut-and-interleave
shuffles of \cite{cppriffleshuffle} (recall from Example \ref{ex:shufflealg-bg}
that this is the case with the shuffle algebra):
\begin{enumerate}[label=\arabic*.]
\item Choose a weak-composition $\left(d_{1},\dots,d_{l(D)}\right)$ of
$n$ according to the distribution $P$.
\item Cut the deck so the first pile contains $d_{1}$ cards, the second
pile contains $d_{2}$ cards, and so on. 
\item Drop the cards on-by-one from the bottom of one of the $l(D)$ piles,
chosen with probability proportional to the current pile size. 
\end{enumerate}
Theorem \ref{thm:cppchain} gives an analogous description of the
chain driven by $m\Delta_{P}$ on any Hopf algebra, separating it
into a breaking part (steps 1 and 2) followed by a recombination (step
3). The probabilities involved in both stages are expressed in terms
of the rescaling function $\eta$ and the \emph{structure constants}
of $\calh$: these are the numbers $\xi_{z_{1,}\dots,z_{l}}^{y},\eta_{x}^{z_{1},\dots,z_{l}}$
defined by 
\[
z_{1}\dots z_{l}=\sum_{y\in\calb}\xi_{z_{1,}\dots,z_{l}}^{y}y,\quad\Delta^{[l]}(x)=\sum_{z_{1},\dots,z_{l}\in\calb}\eta_{x}^{z_{1},\dots,z_{l}}z_{1}\otimes\dots\otimes z_{l},
\]
for $x,y,z_{1},\dots,z_{l}$ in the distinguished basis $\calb$.
\begin{thm}[Definition of descent operator chains]
\label{thm:cppchain}  \cite[Def. 3.1]{cpmcfpsac} Let $P$ be a
probability distribution on the weak-compositions of a fixed integer
$n$. Let $\calh=\bigoplus_{n\geq0}\calhn$ be a graded connected
Hopf algebra over $\mathbb{R}$ with each $\calhn$ finite-dimensional,
and $\calb=\amalg_{n\geq0}\calbn$ a state space basis of $\calh$
for $P$. As above, define functions $m\Delta_{P}:\calhn\rightarrow\calhn$
and $\eta:\calh\rightarrow\mathbb{R}$ by
\[
m\Delta_{p}:=\sum\frac{P(D)}{\binom{n}{d_{1}\dots d_{l(D)}}}m\Delta_{D};
\]
\[
\eta(x):=\mbox{sum of coefficients (in the }\calb_{1}\otimes\dots\otimes\calb_{1}\mbox{ basis) of }\Delta_{1,\dots,1}(x).
\]
Then 
\[
\hatk(x,y):=\frac{\eta(y)}{\eta(x)}\mbox{ coefficient of }y\mbox{ in }m\Delta{}_{P}(x)
\]
gives a transition matrix. Each step of this Markov chain, starting
at $x\in\calbn$, is equivalent to the following three-step process:
\begin{enumerate}[label=\arabic*.]
\item Choose a weak-composition $\left(d_{1},\dots,d_{l(D)}\right)$ of
$n$ according to the distribution $P$.
\item Choose $z_{1}\in\calb_{d_{1}},z_{2}\in\calb_{d_{2}},\dots,z_{l(D)}\in\calb_{d_{l(D)}}$
with probability $\frac{1}{\eta(x)}\eta_{x}^{z_{1},\dots,z_{l(D)}}\eta(z_{1})\dots\eta(z_{l(D)})$.
\item Choose $y\in\calbn$ with probability $\left(\binom{n}{D}\eta(z_{1})\dots\eta(z_{l(D)})\right)^{-1}\xi_{z_{1},\dots,z_{l(D)}}^{y}\eta(y)$.
\end{enumerate}
\end{thm}
Note that the probabilities of the choices in steps 2 and 3 depend
only on the Hopf algebra, not on the probability distribution $P$.
\begin{proof}[Proof of Theorem \ref{thm:cppchain}]
That $\hatk(x,y)$ is a transition matrix follows from Theorem \ref{thm:doob-transform},
the Doob transform for linear maps. What follows will check that the
probabilities under the three-step process agree with $\hatk(x,y)$.
This is easiest using the alternative characterisation of $\eta$
from the proof of Proposition \ref{lem:eta}: $\eta=(\bullet^{*})^{n}$
where $\bullet^{*}\in\calhdual_{1}$ is the linear map sending all
elements of $\calb_{1}$ to 1. Recall also that $\Delta^{[l]}$ is
the iterated coproduct $(\Delta\otimes\id^{\otimes l(D)-1})\circ\dots\circ(\Delta\otimes\id\otimes\id)\circ(\Delta\otimes\id)\circ\Delta$.

First check that, for each weak-composition $D$, the probabilities
in step 2 do sum to 1: 
\begin{align*}
 & \sum_{z_{1}\in\calb_{d_{1}},\dots,z_{l(D)}\in\calb_{d_{l(D)}}}\eta_{x}^{z_{1},\dots,z_{l(D)}}\eta(z_{1})\dots\eta(z_{l(D)})\\
= & \left(\left(\bullet^{*}\right)^{d_{1}}\otimes\dots\otimes\left(\bullet^{*}\right)^{d_{l(D)}}\right)\left(\sum_{z_{1}\in\calb_{d_{1}},\dots,z_{l(D)}\in\calb_{d_{l(D)}}}\eta_{x}^{z_{1},\dots,z_{l(D)}}z_{1}\otimes\dots\otimes z_{l(D)}\right)\\
= & \left(\left(\bullet^{*}\right)^{d_{1}}\otimes\dots\otimes\left(\bullet^{*}\right)^{d_{l(D)}}\right)\left(\Delta_{D}(x)\right)\\
= & \left(\left(\bullet^{*}\right)^{d_{1}}\otimes\dots\otimes\left(\bullet^{*}\right)^{d_{l(D)}}\right)\left(\Delta^{[l(D)]}(x)\right)\\
= & \left(\bullet^{*}\right)^{n}(x)\\
= & \eta(x),
\end{align*}
where the third equality is because $\left(\bullet^{*}\right)^{d}(z_{i})=0$
if $\deg(z_{i})\neq d$, and the fourth equality is by definition
of the product of $\calhdual$. And similarly for the probabilities
in step 3, the combining step:
\begin{align*}
\sum_{y\in\calbn}\xi_{z_{1},\dots,z_{l(D)}}^{y}\eta(y) & =\left(\bullet^{*}\right)^{n}\left(\sum_{y\in\calbn}\xi_{z_{1},\dots,z_{l(D)}}^{y}y\right)\\
 & =\left(\bullet^{*}\right)^{n}(z_{1}\dots z_{l(D)})\\
 & =\Delta^{[l(D)]}((\bullet^{*})^{n})(z_{1}\otimes\dots\otimes z_{l(D)})\\
 & =\left(\sum_{D':l(D')=l(D)}\binom{n}{D'}\left(\bullet^{*}\right)^{d'_{1}}\otimes\dots\otimes\left(\bullet^{*}\right)^{d'_{l(D)}}\right)(z_{1}\otimes\dots\otimes z_{l(D)})\\
 & =\binom{n}{D}\eta(z_{1})\dots\eta(z_{a}),
\end{align*}
where the last equality again relies on the fact that $\left(\bullet^{*}\right)^{d}(z_{i})=0$
if $\deg(z_{i})\neq d$. Finally, the probability of moving from $x$
to $y$ under the three-step process is 
\begin{align*}
 & \sum_{D}P(D)\sum_{z_{i}\in\calb_{i}}\frac{\eta_{x}^{z_{1},\dots,z_{l(D)}}\eta(z_{1})\dots\eta(z_{l(D)})}{\eta(x)}\frac{\xi_{z_{1},\dots,z_{l(D)}}^{y}\eta(y)}{\binom{n}{D}\eta(z_{1})\dots\eta(z_{l(D)})}\\
= & \frac{\eta(y)}{\eta(x)}\sum_{D}\frac{P(D)}{\binom{n}{D}}\sum_{z_{i}\in\calb_{i}}\xi_{z_{1},\dots,z_{l(D)}}^{x}\eta_{y}^{z_{1},\dots,z_{l(D)}}\\
= & \hatk(x,y).
\end{align*}

\end{proof}

\subsection{Eigenvalues and multiplicities\label{sub:evalue}}

Recall from Proposition \ref{prop:doobtransform-efns} that the eigenvalues
for a Markov chain from the Doob transform are simply the eigenvalues
of the associated linear map. Hence, to obtain the spectrum of the
breaking-and-recombination chains of the previous section, it suffices
to calculate the spectrum of the descent operators $m\Delta_{P}$.
The completely general spectrum formula, valid for all $m\Delta_{P}$
and all $\calh$, is rather unsightly, but it simplifies neatly for
many examples of interest, such as Examples \ref{ex:t2r-evalues},
\ref{ex:riffle-binter-evalues}. 

The eigenvalues of $m\Delta_{P}$ are indexed by \emph{partitions}
- these are usually written as tuples $\lambda=(\lambda_{1},\dots,\lambda_{l(\lambda)})$
of integers with $\lambda_{1}\geq\dots\geq\lambda_{l(\lambda)}>0$,
but it will be more convenient here to forget the decreasing ordering
and view them simply as multisets of positive integers. The values
of these eigenvalues themselves are related to \emph{set compositions}
(also known as ordered set partitions): a set composition $B_{1}|\dots|B_{l}$
of a set $S$ is simply an $l$-tuple of disjoint subsets of $S$
with $B_{1}\amalg\dots\amalg B_{l}=S$. The \emph{blocks} $B_{i}$
are allowed to be empty (so perhaps the correct terminology is ``weak
set composition''). The \emph{type} of a set composition is the weak-composition
of cardinalities $(|B_{1}|,\dots,|B_{l}|)$. If $S=\{1,\dots,n\}$,
then the symmetric group $\sn$ acts on the set compositions of $S$
of any given type. For example, $B=\{2,5\}|\ |\{1,4\}|\{3\}$ is a
set composition of $\{1,2,3,4,5\}$ of type $(2,0,2,1)$. The permutation
$\sigma=42351$ sends $B$ to $\sigma(B)=\{\sigma(2),\sigma(5)\}|\ |\{\sigma(1),\sigma(4)\}|\{\sigma(3)\}=\{1,2\}|\ |\{4,5\}|\{3\}$,
and the transpositions $(25)$ and $(14)$ both fix $B$. 
\begin{thm}[Eigenvalues of descent operators]
\label{thm:evalues}  Let $\calh=\bigoplus\calhn$ be a graded connected
Hopf algebra over $\mathbb{R}$, and $P$ a probability distribution
on weak-compositions of a fixed integer n. As usual, write $m\Delta_{P}$
for the associated descent operator 
\[
m\Delta_{P}:=\sum_{D}\frac{P(D)}{\binom{n}{D}}m\Delta_{D}.
\]
The eigenvalues of $m\Delta_{P}:\calhn\rightarrow\calhn$ are 
\[
\beta_{\lambda}^{P}:=\sum_{D}\frac{P(D)}{\binom{n}{D}}\beta_{\lambda}^{D},
\]
where $\beta_{\lambda}^{D}$ is the number of set compositions $B_{1}|\dots|B_{l(D)}$
of $\{1,2,\dots,l(\lambda)\}$ such that, for each $i$, we have $\sum_{j\in B_{i}}\lambda_{j}=d_{i}$.
The multiplicity of the eigenvalue $\beta_{\lambda}^{P}$ is the coefficient
of $x_{\lambda}:=x_{\lambda_{1}}\dots x_{\lambda_{l(\lambda)}}$ in
the generating function $\prod_{i}(1-x_{i})^{-b_{i}}$, where the
numbers $b_{i}$ satisfy 
\[
\sum_{n}\dim\calhn x^{n}=\prod_{i}(1-x^{i})^{-b_{i}}.
\]
Futhermore, $m\Delta_{P}$ is diagonalisable if $\calh$ is commutative
or cocommutative.
\end{thm}
Observe that, under the mild condition $b_{i}>0$ for all $i$ (i.e.
$(\gr\calh)^{*}$ contains primitives in every degree, by five paragraphs
below), the eigenvalues of a descent operator depend only on the associated
probability distribution $P$, not on the Hopf algebra it acts on.
By contrast, in the generic case where all $\beta_{\lambda}^{P}$
are distinct, their multiplicities depend only on the Hopf algebra
(in fact, only on the dimensions of its graded subspaces) and not
on the distribution $P$.

The following two interpretations of the eigenvalues $\beta_{\lambda}^{P}$
are sometimes useful:
\begin{enumerate}[label=\arabic*.]
\item $\bld$ is the number of set compositions of $\{1,2,\dots,n\}$ of
type $D$ which are fixed under the action of any particular permutation
of cycle type $\lambda$ (since this forces each cycle to lie in the
same block). Hence $\blp$ is the probability that a particular permutation
of cycle type $\lambda$ fixes a random set composition chosen in
the following way: choose a weak-composition $D$ according to $P$,
then choose uniformly amongst the set compositions of type $D$. For
many interesting probability distributions $P$, this choice procedure
is not as contrived as it may sound - see Example \ref{ex:riffle-binter-evalues}.
\item By \cite[Prop. 7.7.1, Eq. 7.30]{stanleyec2}, $\beta_{\lambda}^{P}=\langle\underline{S^{P}},p_{\lambda}\rangle$,
the inner product of the power sum symmetric function $p_{\lambda}$
with the commutative image of the noncommutative symmetric function
$S^{P}:=\sum_{D}\frac{P(D)}{\binom{n}{D}}S^{D}$ (i.e. with the linear
combination of complete symmetric functions $\sum_{D}\frac{P(D)}{\binom{n}{D}}h_{D}$).
\end{enumerate}
Note that the numbers $\bld$ depend only on the sizes of the parts
of $D$, not on their order. Also, the eigenvalues $\blp$ need not
be distinct for different partitions $\lambda$; see the example below.
\begin{example}
\label{ex:t2r-evalues} Take $P$ to be concentrated at $(1,n-1)$,
so $m\Delta_{P}$ induces the top-to-random card-shuffle. Then $\beta_{\lambda}^{(1,n-1)}$
is the number of parts of size 1 in $\lambda$, which can be $0,1,\dots,n-2,$
or $n$. So the eigenvalues of a top-to-random chain on any Hopf algebra
are $\blp=\frac{1}{n}\beta_{\lambda}^{(1,n-1)}=0,\frac{1}{n},\frac{2}{n},\dots,\frac{n-2}{n},1$.
 Alternatively, by Interpretation 1 above, $\blp$ is the proportion
of set compositions of type $(1,n-1)$ fixed by any particular permutation
of cycle type $\lambda$ - this is simply the proportion of fixed
points of the permutation, since set compositions of type $(1,n-1)$
are entirely determined by the single element in their first block.

The top-to-random chain is one of the rare examples of a descent operator
chain that admits an explicit diagonalisation on cocommutative Hopf
algebras, see Theorem \ref{thm:tob2r-evectors-cocommutative}.
\end{example}

\begin{example}
\label{ex:riffle-binter-evalues} We apply Interpretation 1 above
to two examples.

First, take $P$ to be the binomial distribution on weak-compositions
with two parts, so $m\Delta_{P}=m\Delta$, inducing the riffle-shuffle
(see the paragraph after Definition \ref{defn:descent-operator}).
Then the process in Interpretation 1 uniformly chooses one of the
$2^{n}$ set compositions with two parts. Since each such set composition
is entirely determined by its first block, Interpretation 1 says that
the eigenvalues $\blp$ are the proportions of subsets of $\{1,\dots,n\}$
fixed by a permutation of cycle type $\lambda$. Being fixed under
the permutation means that these are subsets of its cycles - hence
$\blp=\frac{2^{l(\lambda)}}{2^{n}}$, as shown in \cite[Th. 3.15, 3.16]{hopfpowerchains}.

Now consider a variant where $P$ is supported only on distributions
of the form $(1^{r},n-r)$ for $0\leq r\leq n$, and let $r$ be binomially
distributed. (Here, $1^{r}$ denotes $r$ consecutive parts of size
1.) For this ``binomial-top-to-random'' operator (Definition \ref{defn:tob2r},
with $q=\frac{1}{2}$), $\blp$ is the proportion of subsets fixed
pointwise by a permutation of cycle type $\lambda$. These fixed subsets
are precisely the subsets of the fixed points of the permutation.
So, if there are $j$ fixed points (i.e. $\lambda$ has $j$ parts
of size 1), then the eigenvalue $\blp$ is $\frac{2^{j}}{2^{n}}$.

So both these descent operators have non-positive powers of $2$ as
their eigenvalues, but with different multiplicities. Each fixed partition
$\lambda$ has more parts in total than parts of size 1, so its corresponding
eigenvalue is larger for $m\Delta$ than for the binomial-top-to-random
operator. In the case of card-shuffling, this agrees with intuition:
having cut the deck according to a symmetric binomial distribution,
reinserting the top half of the deck without preserving the relative
order of the cards will randomise the deck faster. \cite[Proof of Cor. 3]{originalriffleshuffle}
made the same comparison; instead of eigenvalues, they looked at the
mixing time, which is $\frac{3}{2}\log_{2}n$ for the riffle-shuffle,
and $\log_{2}n$ for binomial-top-to-random. 
\end{example}
Below are two very different proofs of the spectrum of a descent operator.
The first is probabilistically-inspired, and its key ideas aid in
the construction of eigenvectors in Sections \ref{sub:tob2r-efns}
and \ref{sub:tab2r}. The second comes from assembling known theorems
on noncommutative symmetric functions; this proof was outlined by
the reviewer of \cite{cpmcfpsac}. Both are included in the hope that
they lead to generalisations for different classes of operators.

Both proofs begin by reducing to the case where $\calh$ is cocommutative;
by duality, this will also imply the case for commutative $\calh$.
This reduction follows the argument of \cite[Th. 3]{diagonalisingusinggrh}.
As explained in their Section 1.3, the \emph{coradical filtration}
of a graded connected Hopf algebra $\calh$ is defined as $\calh^{(k)}=\calh_{0}\oplus\bigoplus_{D}\ker\Delta_{D}$,
where the sum ranges over all (strict) compositions $D$ with $k$
parts. The associated graded algebra of $\calh$ with respect to this
filtration, written $\gr(\calh)$, is a Hopf algebra. Every linear
map $\T:\calh\rightarrow\calh$ preserving the coradical filtration
induces a map $\gr(\T):\gr(\calh)\rightarrow\gr(\calh)$ with the
same eigenvalues and multiplicities. Now $m\Delta_{P}$ is a (linear
combination of) convolution product of projections $\Proj_{d_{i}}$
to the graded subspace $\calh_{d_{i}}$. Since $\gr(\Proj_{d})=\Proj_{d}$,
and $\T\rightarrow\gr\T$ preserves convolution products, it must
be that $\gr(m\Delta_{P})=m\Delta_{P}$. So it suffices to show that
$m\Delta_{P}:\gr(\calh)\rightarrow\gr(\calh)$ has the claimed eigenvalues
and multiplicities. By  \cite[Th. 11.2.5.a]{sweedler} \cite[Prop. 1.6]{grhiscommutative},
$\gr(\calh)$ is commutative. (So this argument shows that the eigenvalues
and multiplicities of Theorem \ref{thm:evalues} also apply to any
$\T:\calh\rightarrow\calh$ with $\gr(\T)=m\Delta_{P}$, even if $\T$
itself is not a descent operator.)

\subsubsection*{First proof of Theorem \ref{thm:evalues}: Poincare-Birkhoff-Witt
straightening algorithm and Perron-Frobenius theorem}

By the Cartier-Milnor-Moore theorem \cite[Th. 3.8.1]{cmm}, a graded
connected cocommutative Hopf algebra $\calh$ is the universal enveloping
algebra of its subspace of primitives. Consequently, $\calh$ has
a Poincare-Birkhoff-Witt (PBW) basis: if $(\calp,\preceq)$ is an
ordered basis of the primitive subspace of $\calh$, then $\{p_{1}\dots p_{k}|k\in\mathbb{N},p_{1}\preceq\dots\preceq p_{k}\in\calp\}$
is a basis of $\calh$. The basis element $p_{1}\dots p_{k}$ has
\emph{length} $k$. We will need the following fact, which follows
easily from the ``straightening algorithm'':
\begin{lem}
\label{lem:straightening} \cite[Lem. III.3.9]{pbwref} Let $(\calp,\preceq)$
be an ordered basis of the primitive subspace of $\calh$. If $p_{1},\dots,p_{k}\in\calp$
with \textup{$p_{1}\preceq\dots\preceq p_{k}$,} then, for any $\sigma\in\sk$,
\[
p_{\sigma(1)}\dots p_{\sigma(k)}=p_{1}\dots p_{k}+\mbox{terms of length less than }k.
\]
In particular, the coefficient of the leading term is 1. \qed
\end{lem}
The key to this proof is the following variant of \cite[Th. 3.10]{hopfpowerchains}:
\begin{lem}[Symmetrisation Lemma]
\label{lem:symlemma} Let $p_{1},\dots,p_{k}$ be primitive elements
of $\calh$ and let $\deg(\mathbf{p})$ denote the partition $(\deg p_{1},\dots,\deg p_{k})$.
Then $\sspan\{p_{\sigma(1)}\dots p_{\sigma(k)}|\sigma\in\sk\}$ is
an invariant subspace under $m\Delta_{P}$, and contains a $\beta_{\deg(\mathbf{p})}^{P}$-eigenvector
of the form $\sum_{\sigma\in\sk}\kappa_{\sigma}p_{\sigma(1)}\dots p_{\sigma(k)}$
with all $\kappa_{\sigma}\geq0$. \end{lem}
\begin{proof}
Work first in the free associative algebra generated by primitive
elements $p_{1}',\dots,p_{k}'$, with $\deg p_{i}'=\deg p_{i}$. (Equivalently,
treat $p_{1},\dots,p_{k}$ as formal variables, ignoring any algebraic
relationships between them.) Since the $p_{i}'$ are primitive, 
\begin{equation}
m\Delta_{D}(p'_{1}\dots p'_{k})=\sum_{B_{1},\dots,B_{l(D)}}\left(\prod_{i\in B_{1}}p'_{i}\right)\left(\prod_{i\in B_{2}}p'_{i}\right)\dots\left(\prod_{i\in B_{l(D)}}p'_{i}\right),\label{eq:descent-operator-on-product-of-primitives}
\end{equation}
summing over all set compositions $B_{1}|\dots|B_{l(D)}$ of $\{1,2,\dots,k\}$
such that $\sum_{j\in B_{i}}\deg(p'_{j})=d_{i}$. So each summand
$m\Delta_{D}$ of $m\Delta_{P}$ fixes the subspace $W:=\sspan\{p'_{\sigma(1)}\dots p'_{\sigma(k)}|\sigma\in\sk\}$,
and hence so does $m\Delta_{P}$ itself. Consider the matrix of the
restricted map $m\Delta_{P}|_{W}$ with respect to the basis $\{p'_{\sigma(1)}\dots p'_{\sigma(k)}|\sigma\in\sk\}$.
(This is indeed a basis because the $p_{i}'$ generate a free associative
algebra.) From taking the appropriate linear combination of Equation
\ref{eq:descent-operator-on-product-of-primitives}, we see that the
sum of the entries in the column corresponding to $p'_{1}\dots p'_{k}$
is $\beta_{\deg(\mathbf{p})}^{P}$. Note that the partition $\deg(\mathbf{p})$,
and hence $\beta_{\deg(\mathbf{p})}^{P}$, is independent of the ordering
of the $p'_{i}$, so all columns of the matrix of $m\Delta_{P}|_{W}$
sum to $\beta_{\deg(\mathbf{p})}^{P}$. So the left (row) vector $(1,1,\dots,1)$
is an eigenvector of this matrix with eigenvalue $\beta_{\deg(\mathbf{p})}^{P}$.
Since this vector has all components positive, and all entries of
this matrix of $m\Delta_{P}|_{W}$ are non-negative, the Perron-Frobenius
theorem \cite[Ch. XIII Th. 3]{perronfrob} states that $\beta_{\deg(\mathbf{p})}^{P}$
is the largest eigenvalue of $m\Delta_{P}|_{W}$, and $m\Delta_{P}|_{W}$
has a (right, column) eigenvector of this eigenvalue with non-negative
entries (note that it is in general not unique). Let $\sum_{\sigma\in\sk}\kappa_{\sigma}p'_{\sigma(1)}\dots p'_{\sigma(k)}$
denote this eigenvector, so 
\[
m\Delta_{P}\left(\sum_{\sigma\in\sk}\kappa_{\sigma}p'_{\sigma(1)}\dots p'_{\sigma(k)}\right)=\beta_{\deg(\mathbf{p})}^{P}\left(\sum_{\sigma\in\sk}\kappa_{\sigma}p'_{\sigma(1)}\dots p'_{\sigma(k)}\right).
\]
Apply to both sides the Hopf morphism sending $p_{i}'$ to $p_{i}$;
this shows that $\sum_{\sigma\in\sk}\kappa_{\sigma}p_{\sigma(1)}\dots p_{\sigma(k)}$
is a $\beta_{\deg(\mathbf{p})}^{P}$-eigenvector of $m\Delta_{P}$
on our starting Hopf algebra.
\end{proof}
\begin{rems*}
$ $
\begin{enumerate}[label=\arabic*.]
\item The transpose of the matrix in the above proof, of $m\Delta_{P}|_{W}$
with respect to the basis $\{p'_{\sigma(1)}\dots p'_{\sigma(k)}|\sigma\in\sk\}$,
is the transition matrix of a hyperplane walk \cite{hyperplanewalk},
scaled by $\beta_{\deg(\mathbf{p})}^{P}$. Informally, this walk is
the pop shuffle associated to $m\Delta_{P}$ (see Example \ref{ex:freeassalg-bg}
above) where the primitive $p_{i}'$ behave like $\deg(p_{i}')$ cards
glued together. So the distribution $\pi(\sigma)=\kappa_{\sigma}$
is a stationary distribution of this chain. The idea of expressing
each member of an eigenbasis in terms of the stationary distribution
of a different chain is also integral to the recent left eigenfunction
formulae for hyperplane walks \cite{hyperplanewalkefns1,hyperplanewalkefns2}.

In extremely simple cases, this view of the coefficients $\kappa_{\sigma}$
as the stationary distribution is surprisingly powerful: take $p_{1}=\dots=p_{j}\in\calh_{1}$,
and let $p_{j+1},\dots,p_{k}$ be primitives of degree greater than
1. The hyperplane walk that $m\Delta_{1,n-1}$ induces on $\sspan\{p{}_{\sigma(1)}\dots p{}_{\sigma(k)}\}$
is the ``random-to-top shuffle'': uniformly choose a card to remove
from the deck and place it on top. Since $p_{j+1},\dots,p_{k}$ represent
multiple cards glued together, these cards are never moved, so in
the stationary distribution, they must be at the bottom of the deck,
in the same relative order as they started. And the order of the single
cards $p_{1},\dots,p_{j}$ at the top of the deck is immaterial since
these cards are identical. So $p_{1}\dots p_{k}$ is an eigenvector
for $m\Delta_{1,n-1}$. In the common scenario where $\dim\calh_{1}=1$,
all multisets of $\calp$ have this form, so the simple argument above
produces a full eigenbasis.

Theorem \ref{thm:tob2r-evectors-cocommutative} is a more complex
argument along the same lines, making use of symmetry to simplify
the required hyperplane walk (see point 3 below). However, for general
descent operators $m\Delta_{P}$, the formula for this stationary
distribution \cite[Th. 2b]{hyperplanewalkstationary} is notoriously
difficult to compute with.

\item If $\calh$ is commutative as well as cocommutative, then the Symmetrisation
Lemma shows that any product of primitive elements is an eigenvector
for $m\Delta_{P}$ for all distributions $P$. Hence $\{p_{1}\dots p_{k}|p_{1}\preceq\dots\preceq p_{k}\in\calp\}$
is an eigenbasis for all $m\Delta_{P}$, and all descent operators
commute (which is also clear from Proposition \ref{prop:descentoperator-composition}).
\item The entries of the matrix of $m\Delta_{P}|_{W}$ depend only on the
degrees of the primitives $p_{i}$. Hence permuting the labels of
the $p_{i}$ with the same degree does not change this matrix. So
the eigenvector $\sum_{\sigma\in\sk}\kappa_{\sigma}p_{\sigma(1)}\dots p_{\sigma(k)}$
can be arranged to be symmetric in the primitives of same degree -
in other words, $\kappa_{\sigma}$ depends only on the tuple $(\deg p_{\sigma(1)},\dots,\deg p_{\sigma(k)})$,
not on $(p_{\sigma(1)},\dots,p_{\sigma(k)})$.
\end{enumerate}
\end{rems*}The final step of the proof, to deduce the existence of
an eigenbasis with the claimed eigenvalues and multiplicities, goes
as follows. Apply the Symmetrisation Lemma to each multiset $\{p_{1}\preceq\dots\preceq p_{k}\}\subseteq\calp$
to get an eigenvector, whose highest length term, by Lemma \ref{lem:straightening},
is $\sum_{\sigma\in\sk}\kappa_{\sigma}p_{1}\dots p_{k}$. Since $\sum_{\sigma\in\sk}\kappa_{\sigma}>0$,
this set of eigenvectors is triangular with respect to the PBW basis
of $\calh$, hence is itself a basis of $\calh$. The number of such
eigenvectors with eigenvalue $\beta_{\lambda}^{P}$ is the number
of multisets $\{p_{1}\preceq\dots\preceq p_{k}\}\in\calp$ with $\lambda=(\deg p_{1},\dots,\deg p_{k})$.
Since $b_{i}$ counts the elements of $\calp$ of degree $i$, the
generating function for such multisets is indeed $\prod_{i}(1-x_{i})^{-b_{i}}$.

\subsubsection*{Second proof of Theorem \ref{thm:evalues}: descent algebras and
noncommutative symmetric functions}

Recall from Proposition \ref{prop:descentoperator-composition} that,
on a cocommutative Hopf algebra $\calh$, the composition of descent
operators is equivalent to the internal product of noncommutative
symmetric functions: $\cppmap(f)\circ\cppmap(g)=\cppmap(f\iprod g)$.
(Recall that $\cppmap$ is defined as the linear map sending the complete
noncommutative symmetric function $S^{D}$ to the descent operator
$m\Delta_{D}$.) Write $S^{P}$ for $\sum_{D}\frac{P(D)}{\binom{n}{D}}S^{D}$
(so $\cppmap(S^{P})=m\Delta_{P}$), and focus on the linear map $L_{S^{P}}:\sym\rightarrow\sym$
given by internal product on the left by $S^{P}$. Since the internal
product of $\sym$ is equivalent to the product in the descent algebra,
 \cite[Prop. 3.10]{cppdiagonalisable} \cite[Th. 1]{hyperplanewalkstationary}
asserts that $L_{S^{P}}$ is diagonalisable, and \cite[Prop. 3.12]{ncsym2}
shows that its eigenvalues are the $\beta_{\lambda}^{P}$ in the theorem
statement. (In what follows, assume $P$ is ``generic'' so that
all $\beta_{\lambda}^{P}$ are distinct. This suffices since the characteristic
polynomial of a matrix is continuous in its entries.)

To see that $\beta_{\lambda}^{P}$ are also eigenvalues of $\cppmap(S^{P})=m\Delta_{P}:\calhn\rightarrow\calhn$,
consider the orthogonal projections to each eigenspace of $L_{S^{P}}$.
These are polynomials in $L_{S^{P}}$, and are therefore of the form
$L_{E_{\lambda}^{P}}$ for some $E_{\lambda}^{P}\in\sym$.  Now the
image of $\cppmap(E_{\lambda}^{P}):\calhn\rightarrow\calhn$ consists
of eigenvectors of $m\Delta_{P}$, since 
\begin{align*}
\cppmap(S^{P})(\cppmap(E_{\lambda}^{P})x) & =\cppmap(S^{P}\iprod E_{\lambda}^{P})(x)\\
 & =\beta_{\lambda}^{P}\cppmap(E_{\lambda}^{P})x.
\end{align*}
 Hence , $\beta_{\lambda}^{P}$ are indeed the eigenvalues of $m\Delta_{P}$.

It remains to determine the multiplicities of the eigenvalues (and
deduce by dimension counting that no other eigenvalues can exist).
An extra piece of notation is useful here: consider the linear map
from $\sym$ to the algebra of symmetric functions \cite[Chap. 7]{stanleyec2},
sending $S^{D}$ to the complete symmetric function $h_{D}$. The
image of $F\in\sym$ under this map is its \emph{commutative image}
$\underline{F}$.
\begin{lem}
\cite[Th. 3.21]{ncsym2}\label{lem:idempotents} Suppose $E_{\lambda},E_{\lambda}'$
are two noncommutative symmetric functions, idempotent under the internal
product, whose commutative images $\underline{E_{\lambda}},\underline{E_{\lambda}'}$
are both the normalised power sum $\frac{p_{\lambda}}{z_{\lambda}}$.
Let $\calh$ be a graded connected cocommutative Hopf algebra. Then
the linear map $\cppmap(S^{n}-E_{\lambda}-E_{\lambda}'):\calhn\rightarrow\calhn$
is invertible, and sends the image of $\cppmap(E_{\lambda})$ to the
image of $\cppmap(E_{\lambda}^{'})$. In particular, these two images
have the same dimension. \qed
\end{lem}
(The reference treats only the case where $\calh$ is the free associative
algebra, but the proof - that the eigenvalues $\langle\underline{S^{n}-E_{\lambda}-E_{\lambda}'},p_{\mu}\rangle$
are non-zero - holds for any cocommutative $\calh$.)

Set $E_{\lambda}'$ to be the eigenspace projector $E_{\lambda}^{P}$.
Its commutative image $\underline{E_{\lambda}^{P}}$ is indeed $\frac{p_{\lambda}}{z_{\lambda}}$,
because $L_{\frac{p_{\lambda}}{z_{\lambda}}}$ (the left-internal-product
map on symmetric functions) is the orthogonal projection to the $\blp$-eigenspace
for $L_{\underline{S_{P}}}$.Take $E_{\lambda}$ to be the \emph{Garsia-Reutenauer
idempotents} \cite[Sec. 3.3]{ncsym2}, so the image of $\cppmap(E_{\lambda}):\calhn\rightarrow\calhn$
has basis $\{\sum_{\sigma\in\mathfrak{S}_{l(\lambda)}}p{}_{\sigma(1)}\dots p{}_{\sigma(l(\lambda))}|p_{1}\preceq\dots\preceq p_{l(\lambda)}\in\calp,\deg p_{i}=\lambda_{i}\}$,
where $\calp$ is an ordered basis of primitives of $\calh.$ The
cardinality of these sets are precisely as given by the generating
functions in Theorem \ref{thm:evalues}.
\begin{rem*}
Note that, if $x\in\im\cppmap(E_{\lambda})$, then $\cppmap(S^{n}-E_{\lambda}-E_{\lambda}^{P})x=\cppmap(-E_{\lambda}^{P})x$,
since $\cppmap(S^{n})$ and $\cppmap(E_{\lambda})$ both act as the
identity map on $\im\cppmap(E_{\lambda})$. Hence the proof above
supplies the following eigenbasis for $m\Delta_{P}$: 
\[
\left\{ \cppmap(E_{\deg(\mathbf{p})}^{P})\left(\sum_{\sigma\in\sk}p_{\sigma(1)}\dots p_{\sigma(k)}\right)\middle|k\in\mathbb{N},p_{1}\preceq\dots\preceq p_{k}\in\calp,\deg p_{i}=\lambda_{i}\right\} ,
\]
where $L{}_{E_{\deg(\mathbf{p})}^{P}}$ are the orthogonal projections
to the eigenspaces of $L{}_{S^{P}}$. However, this formula may not
lead to easy computation, since expressions for the $E_{\lambda}^{P}$
are usually fairly complicated, see \cite[Eq. 4.5]{cppriffleshuffle}.
\end{rem*}

\subsection{Stationary distribution\label{sub:Stationary-distribution}}

All descent operator Markov chains on the same state space basis share
the same stationary distributions. These have a simple expression
in terms of the product structure constants and the rescaling function
$\eta$ of Lemma \ref{lem:eta}. Informally, $\pi_{c_{1},\dots,c_{n}}(x)$
enumerates the ways to build $x$ out of $c_{1},\dots,c_{n}$ (in
any order) using the multiplication of the combinatorial Hopf algebra,
and to then break it into singletons. (The theorem below restricts
to probability distributions $P$ taking a non-zero value on some
weak-composition with at least two non-zero parts, so the chain driven
by $m\Delta_{P}$ is not trivial - else every distribution is a stationary
distribution.)
\begin{thm}[Stationary distributions of descent operator chains]
\label{thm:stationarydistribution}  Let $\calh=\bigoplus_{n\geq0}\calhn$
be a graded connected Hopf algebra over $\mathbb{R}$ with each $\calhn$
finite-dimensional, and $\calb=\amalg_{n\geq0}\calbn$ a basis of
$\calh$. Fix an integer $n$, and let $P$ be any probability distribution
on the weak-compositions of $n$ such that $\calb$ is a state space
basis for $P$, and $P$ is non-zero on some weak-composition with
at least two non-zero parts. For each multiset $\{c_{1},\dots,c_{n}\}$
in $\calb_{1}$, define the function $\pi_{c_{1},\dots,c_{n}}(x):\calbn\rightarrow\mathbb{R}$
by 
\[
\pi_{c_{1},\dots,c_{n}}(x):=\frac{\eta(x)}{n!^{2}}\sum_{\sigma\in\sn}\xi_{c_{\sigma(1)},\dots,c_{\sigma(n)}}^{x}=\frac{\eta(x)}{n!^{2}}\sum_{\sigma\in\sn}\mbox{coefficient of }x\mbox{ in the product }c_{\sigma(1)}\dots c_{\sigma(n)}.
\]
If $\pi_{c_{1},\dots,c_{n}}(x)\geq0$ for all $x\in\calbn$, then
$\pi_{c_{1},\dots,c_{n}}$ is a stationary distribution for the Markov
chain on $\calbn$ driven by $m\Delta_{P}$, and any stationary distribution
of this chain can be uniquely written as a linear combination of these
$\pi_{c_{1},\dots,c_{n}}$.
\end{thm}
Many Hopf algebras satisfy $\dim\calh_{1}=1$, in which case the stationary
distribution is unique and given by 
\[
\pi(x):=\frac{\eta(x)}{n!}\xi_{\bullet,\dots,\bullet}^{x}=\frac{\eta(x)}{n!}\mbox{ coefficient of }x\mbox{ in the product }\bullet^{n},
\]
where $\bullet$ denotes the sole element of $\calb_{1}$. This simplifed
formula applies to both extended examples in Sections \ref{sec:cktrees}
and \ref{sec:fqsym}.
\begin{proof}
 By Proposition \ref{prop:doobtransform-efns}.L, the theorem follows
from the following three assertions:
\begin{enumerate}
\item Each function $\pi_{c_{1},\dots,c_{n}}$ has images summing to 1,
so if it takes non-negative values, it is indeed a probability distribution.
\item For any probability distribution $P$ on weak-compositions which is
non-zero on some weak-composition with at least two non-zero parts,
the partition $(1,1,\dots,1)$ is the only $\lambda$ for which $\beta_{\lambda}^{P}=1$.
\item The set of symmetrised products $\left\{ \sum_{\sigma\in\sn}c_{\sigma(1)}\dots c_{\sigma(n)}\right\} $,
over all choices of multisets $\{c_{1},\dots,c_{n}\}\subseteq\calb_{1}$,
gives a basis of the $\beta_{(1,\dots,1)}^{P}$-eigenspace of $m\Delta_{P}$.
In other words, each symmetrised product is a $\beta_{(1,\dots,1)}^{P}$-eigenvector
of $m\Delta_{P}$, and this set is linearly independent and has cardinality
equal to the multiplicity of $\beta_{(1,\dots,1)}^{P}$ specified
in Theorem \ref{thm:evalues}.
\end{enumerate}
For i, to see that $\sum_{x\in\calbn}\pi_{c_{1},\dots,c_{n}}(x)=1$,
appeal to the second displayed equation of the proof of Theorem \ref{thm:cppchain}.
Taking $z_{i}=c_{i}$, it shows that, for each $\sigma\in\sn$, 
\[
\sum_{x\in\calbn}\xi_{c_{\sigma(1)},\dots,c_{\sigma(n)}}^{x}\eta(x)=\binom{n}{\deg c_{\sigma(1)}\dots\deg c_{\sigma(n)}}\eta(c_{1})\dots\eta(c_{n})=n!\cdot1\cdot\dots\cdot1.
\]

Now turn to ii. Recall that $\beta_{\lambda}^{P}:=\sum_{D}\frac{P(D)}{\binom{n}{D}}\beta_{\lambda}^{D}$,
so it suffices to show, for each weak-composition $D$, that $\beta_{(1,\dots,1)}^{D}=\binom{n}{D}$,
and that $\bld\leq\binom{n}{D}$ for all other partitions $\lambda$,
with a strict inequality if $D$ has more than one non-zero part.
The first assertion follows from the definition of $\beta_{(1,\dots,1)}^{D}$
as the number of set compositions of $\{1,2,\dots,n\}$ into $l(D)$
blocks such that block $i$ contains $d_{i}$ elements. Observe that
$\bld$ counts the subsets of these set compositions such that $1,2,\dots,\lambda_{1}$
are in the same block, $\lambda_{1}+1,\lambda_{1}+2,\dots,\lambda_{1}+\lambda_{2}$
are in the same block, and so on. If $\lambda\neq(1,\dots,1)$ and
$D$ has more than one non-zero part, this imposes a non-trivial restriction,
so the count is strictly smaller.

As for iii, recall from the proof of the Symmetrisation Lemma (Lemma
\ref{lem:symlemma}) that, for each weak-composition $D$,
\[
m\Delta_{D}(c{}_{1}\dots c{}_{n})=\sum_{B_{1},\dots,B_{l(D)}}\left(\prod_{i\in B_{1}}c{}_{i}\right)\left(\prod_{i\in B_{2}}c{}_{i}\right)\dots\left(\prod_{i\in B_{l(D)}}c{}_{i}\right),
\]
where the sum runs over all set compositions $B_{1}|\dots|B_{l(D)}$
of $\{1,2,\dots,n\}$ with $d_{i}$ elements in $B_{i}$. So the symmetrised
product $\sum_{\sigma\in\sn}c_{\sigma(1)}\dots c_{\sigma(n)}$ is
a $\beta_{(1,\dots,1)}^{D}$-eigenvector of $m\Delta_{D}$, for all
weak-compositions $D$, and hence is a $\beta_{(1,\dots,1)}^{P}$-eigenvector
of $m\Delta_{P}$. Applying the Poincare-Birkhoff-Witt straightening
algorithm to these symmetrised products give different highest length
terms, so they are linearly independent (see Lemma \ref{lem:straightening}).

It remains to check that the number of such symmetrised products is
equal to the multiplicity of the eigenvalue $\beta_{(1,\dots,1)}^{P}$
as specified by Theorem \ref{thm:evalues}. Clearly the number of
such symmetrised products is $\binom{|\calb_{1}|+n-1}{n}$, the number
of ways to choose $n$ unordered elements, allowing repetition, from
$\calb_{1}$. On the other hand, the eigenvalue multiplicity is $\binom{b_{1}+n-1}{n}$,
since choosing $n$ elements whose degrees sum to $n$ constrains
each element to have degree 1. By equating the coefficient of $x$
in the equality $\prod_{i}\left(1-x^{i}\right)^{-b_{i}}=\sum_{n}\dim\calhn x^{n}$,
it is clear that $b_{1}=\dim\calh_{1}=|\calb_{1}|$. So the number
of such symmetrised products is indeed the multiplicity of the eigenvalue
$\beta_{(1,\dots,1)}^{P}$.
\end{proof}

\subsection{Absorption probabilities and quasisymmetric functions \label{sub:Absorption}}

In this section, assume that $\calh$ is commutative, so no symmetrisation
is necessary in the expression for the stationary distributions in
Theorem \ref{thm:stationarydistribution}:
\[
\pi_{c_{1},\dots,c_{n}}(x)=\frac{\eta(x)}{n!}\mbox{coefficient of }x\mbox{ in the product }c_{1}\dots c_{n}.
\]
Clearly, if $c_{1}\dots c_{n}\in\calb$ (as opposed to being a linear
combination of more than one basis element), then this state is absorbing.
In general, there may be many absorbing states, and also stationary
distributions supported on multiple, non-absorbing states. One sufficient
condition for the absence of the latter is when $\calh$ is \emph{freely
generated} as an algebra (i.e. $\calh=\mathbb{R}[x_{1},x_{2},\dots]$
where $x_{i}$ may have any degree) and $\calb=\left\{ x_{i_{1}}\dots x_{i_{k}}\right\} $
is the set of products in the generators. This is the case for the
chain on organisational tree structures of Example \ref{ex:rooted-trees}
(eventually all employees leave) and also for the rock-chipping model
in the introduction (eventually all rocks have size 1). Under these
conditions, the probability of absorption can be rephrased in terms
of the fundamental Hopf morphism of \cite[Th. 4.1]{abs}. This connection
is a generalisation of \cite[Prop. 3.25]{hopfpowerchains} and \cite[Prop. 5.1.18]{mythesis};
as remarked there, this result does not seem to give an efficient
way to compute these absorption probabilities.

Define a function $\zeta$ on the generators of $\calh$
\[
\zeta(x_{i})=\begin{cases}
1, & \mbox{if }\deg x_{i}=1;\\
0, & \mbox{if }\deg x_{i}>1,
\end{cases}
\]
and extend it linearly and multiplicatively to a character $\zeta:\calh\rightarrow\mathbb{R}$.
So, on $\calb_{n}$, we have $\zeta(x)=1$ if $x$ is an absorbing
state, and 0 otherwise. \cite[Th. 4.1]{abs} asserts that there is
a unique Hopf morphism $\chi$ from $\calh$ to $QSym$, the algebra
of quasisymmetric functions \cite{qsym}, such that $\zeta$ agrees
with the composition of $\chi$ with evaluation at $z_{1}=1,z_{2}=z_{3}=\dots=0$.
(Here, $z_{i}$ are the variables of $QSym$.)
\begin{thm}[Absorption probabilities of certain descent operator chains]
\label{thm:absorption}  Let $\calh$ be a Hopf algebra isomorphic
to $\mathbb{R}[x_{1},x_{2},\dots]$ as an algebra, and $\calb=\left\{ x_{i_{1}}\dots x_{i_{k}}\right\} $.
Let $\{X_{t}\}$ be the descent operator Markov chain on $\calbn$
drivenby $m\Delta_{P}$, started at $x_{0}$. Then the probability
that $\{X_{t}\}$ is absorbed in $t$ steps is 
\[
\frac{n!}{\eta(x_{0})}\langle\underbrace{S^{P}\iprod S^{P}\iprod\dots\iprod S^{P}}_{t\mbox{ factors}},\chi(x_{0})\rangle.
\]

\end{thm}
More explicitly, this absorption probability is: start with the noncommutative
symmetric function $S^{P}:=\sum_{D}\frac{P(D)}{\binom{n}{D}}S^{D}$,
take its $t$-fold internal product with itself, then take its inner
product with $\chi(x_{0})$ (since $QSym\ni\chi(x_{0})$ and $\sym\ni S^{P}$
are dual Hopf algebras), multiply by $n!$ and divide by the rescaling
function evaluated at the initial state.

Note that this theorem only gives the probability of reaching the
set of absorbing states; the above formulation does not calculate
the different probabilities of being absorbed at different states.
A variant which partially addresses this is \cite[Prop. 5.1.19]{mythesis}.
\begin{proof}
The proof is simply a matter of unpacking definitions.

First, reduce to the case of $t=1$ using Proposition \ref{prop:descentoperator-composition}:
since $\calh$ is commutative, $t$ steps of the chain from $S^{P}$
is equivalent to a single step of the chain from $S^{P}\iprod S^{P}\iprod\dots\iprod S^{P}$
(with $t$ factors). So it suffices to show that 
\begin{equation}
\sum_{y}\hatk(x_{0},y)=\frac{n!}{\eta(x_{0})}\langle S^{P},\chi(x_{0})\rangle,\label{eq:absorption}
\end{equation}
where the sum runs over all absorbing states $y$, and $\hatk$ is
the transition matrix of the chain driven by $m\Delta_{P}$. Unravelling
the definition of the Doob transform, 
\begin{align*}
\hatk(x_{0},y) & =\frac{\eta(y)}{\eta(x_{0})}\mbox{coefficient of }y\mbox{ in }m\Delta_{P}(x_{0})\\
 & =\frac{\eta(y)}{\eta(x_{0})}\mbox{coefficient of }y\mbox{ in }\left(\cppmap(S^{P})\right)(x_{0}),
\end{align*}
where $\cppmap$ is the (external product) algebra homomorphism from
noncommutative symmetric functions to descent operators. So Equation
\ref{eq:absorption} is linear in $S^{P}$, and thus it suffices to
work with the complete noncommutative symmetric function $S^{D}$,
i.e. to show
\[
\sum_{y\mbox{ absorbing}}\frac{\eta(y)}{\eta(x_{0})}\mbox{coefficient of }y\mbox{ in }m\Delta_{D}(x_{0})=\frac{n!}{\eta(x_{0})}\langle S^{D},\chi(x_{0})\rangle.
\]

Now the inner product of a quasisymmetric function with $S^{D}$ is
simply its coefficient of $M_{D}$, the monomial quasisymmetric function.
And \cite[Th. 4.1]{abs} defines the $M_{D}$ coefficient of $\chi(x_{0})$
to be $\zeta^{\otimes l(D)}\circ\Delta_{D}(x_{0})$. Since $\zeta$
is multiplicative and linear, this is 
\begin{align*}
\zeta(m\Delta_{D}(x_{0})) & =\sum_{y\in\calbn}\zeta(y)\times\mbox{coefficient of }y\mbox{ in }m\Delta_{D}(x_{0})\\
 & =\sum_{y\mbox{ absorbing}}\mbox{coefficient of }y\mbox{ in }m\Delta_{D}(x_{0}).
\end{align*}
Finally, note that, by the compatibility of product and coproduct,
the rescaling function $\eta$ evaluates to $n!$ on each absorbing
state $c_{1}\dots c_{n}$.
\end{proof}

\section{The Top-or-Bottom-to-Random Chains\label{sec:partsof1}}

This section examines Markov chains driven by various generalisations
of the top-to-random operator $m\Delta_{1,n-1}$. These are of particular
interest because, on many combinatorial Hopf algebras, the refined
coproduct $\Delta_{1,n-1}$ and refined product $m:\calh_{1}\otimes\calh_{n-1}\rightarrow\calh_{n}$
are much easier to understand than the full coproduct and product.
As a result, the chains arising from $m\Delta_{1,n-1}$ (remove and
reattach a piece of size 1) are more natural than those from other
descent operators. Furthermore, these chains have very tractable eigendata
(Theorem \ref{thm:tob2r-evectors-cocommutative}): many of their eigenvalues
collide, resulting in at most $n$ distinct eigenvalues, and there
is an extremely simple formula for many of their eigenvectors.

\subsection{The top-or-bottom-to-random operators\label{sub:tob2r}}

This section defines the operators of interest and details the relationships
between them. In a weak-composition, write $1^{r}$ to denote $r$
consecutive parts of size 1. (Take the convention that $1^{0}$ denotes
a single part of size 0.)
\begin{defn}
\label{defn:tob2r} $ $
\begin{itemize}
\item The \emph{top-to-random} distribution is concentrated at $(1,n-1)$.
\item The \emph{top-$r$-to-random} distribution is concentrated at $(1^{r},n-r)$
(see \cite[Sec. 2]{cppriffleshuffle}).
\item The \emph{binomial-top-to-random} distribution, with parameter $q_{2}$,
assigns probability $\binom{n}{r}(1-q_{2})^{r}q_{2}{}^{n-r}$ to $(1^{r},n-r)$,
for $0\leq r\leq n$ (see \cite[Sec. 3 Exs. 2,3]{cppriffleshuffle}).
\end{itemize}

The related descent operators are:
\begin{align*}
\ter_{n}: & =\frac{1}{n}m\Delta_{1,n-1};\\
\trer_{n}: & =\frac{1}{n(n-1)\dots(n-r+1)}m\Delta_{1^{r},n-r};\\
\binter_{n}(q_{2}): & =\sum_{r=0}^{n}\frac{1}{r!}(1-q_{2})^{r}q_{2}{}^{n-r}m\Delta_{1^{r},n-r}.
\end{align*}

\end{defn}
An alternative definition of $\binter_{n}$, in terms of the formal
series for the exponential function (under the convolution product)
\[
\exp_{*}(x)=1+x+\frac{1}{2!}x*x+\frac{1}{3!}x*x*x+\dots,
\]
is 
\[
\sum_{n}\binter_{n}(q_{2}):=\left(\exp_{*}\left(\frac{1-q_{2}}{q_{2}}\Proj_{1}\right)\right)*\left(\sum_{n}q_{2}^{n}\Proj_{n}\right).
\]
This formulation will not be necessary for what follows. 

We will informally refer to all three operators above as top-to-random
maps. Here is the first way in which they are simpler than the arbitrary
descent operator. Recall that a descent operator induces a Markov
chain on $\calbn$ only if $\calbn$ is a state space basis (Definition
\ref{defn:state-space-basis}), where the relevant product and coproduct
structure constants are non-negative. When using a top-to-random map,
there are fewer structure constants to check. Thanks to associativity
and coassociativity, the following conditions suffice:
\begin{enumerate}
\item for all $c\in\calb_{1},z\in\calb$, their product is $cz=\sum_{y\in\calb}\xi_{c,z}^{y}y$
with $\xi_{c,z}^{y}\geq0$ ;
\item for all $x\in\calb$, its refined coproduct is $\Delta_{1,\deg(x)-1}(x)=\sum_{c,z\in\calb}\eta_{x}^{c,z}c\otimes z$
with $\eta_{x}^{c,z}\geq0$;
\item for each $x\in\calb$, at least one of the $\eta_{x}^{c,z}$ above
is non-zero; in other words, $\Delta_{1,\deg(x)-1}(x)\neq0$. 
\end{enumerate}
Observe that performing a $\trer$ card-shuffle followed by a $\ter$
shuffle results in either a $\opt(r+1)\oper$ shuffle (if the top
card after $\trer$ is one which has not yet been touched) or a $\trer$
shuffle (if this top card was already touched). Thus, as operators
on the shuffle algebra, 
\[
\ter_{n}\circ\trer_{n}=\frac{n-r}{n}\opt(r+1)\oper_{n}+\frac{r}{n}\trer_{n}.
\]
Iterating this shows that the $\trer$ shuffle is a degree $r$ polynomial
(under composition) in the $\ter$ shuffle (see Proposition \ref{prop:tob2r-poly}
below.) And the $\binter$ shuffle is clearly a linear combination
of $\trer$ shuffles. Hence all three types of shuffles can be understood
just by examining the top-to-random shuffle.

In fact, the argument above goes through for all commutative Hopf
algebras, by the composition rule (Proposition \ref{prop:descentoperator-composition}).
It holds more generally for the following maps:
\begin{defn}
\label{defn:qtob2r} $ $
\begin{itemize}
\item The \emph{top-or-bottom-to-random} distribution, with parameter $q$,
assigns probability $q$ to $(1,n-1)$ and probability $1-q$ to $(n-1,1)$.
(\cite[Sec. 6 Ex. 4]{cppriffleshuffle} mentions the symmetric case,
of $q=\frac{1}{2}$.)
\item The \emph{binomial-top-or-bottom-$r$-to-random} distribution, with
parameter $q$, assigns probability $\binom{r}{r_{1}}q^{r_{1}}(1-q)^{r_{3}}$
to $(1^{r_{1}},n-r,1^{r_{3}})$, for $r_{1,}r_{3}\geq0$ with $r_{1}+r_{3}=r$. 
\item The\emph{ trinomial-top-or-bottom-to-random} distribution, with parameters
$q_{1},q_{2},q_{3}$ summing to 1, assigns probability $\binom{n}{r_{1}r_{2}r_{3}}q_{1}^{r_{1}}q_{2}^{r_{2}}q_{3}^{r_{3}}$
to $(1^{r_{1}},r_{2},1^{r_{3}})$, for $r_{1,}r_{2},r_{3}\geq0$ with
$r_{1}+r_{2}+r_{3}=n$. (The corresponding shuffle was termed ``trinomial
top and bottom to random'' in \cite[Sec. 6 Ex. 6]{cppriffleshuffle},
but we change the conjunction from ``and'' to ``or'' to highlight
its relationship to the top-or-bottom-to-random distribution, and
to distinguish it from the top-and-bottom-to-random distributions
of Section \ref{sub:tab2r}.)
\end{itemize}

The related descent operators are:
\begin{align*}
\tober_{n}(q): & =\frac{q}{n}m\Delta_{1,n-1}+\frac{1-q}{n}m\Delta_{n-1,1};\\
\bintobrer_{n}(q): & =\frac{1}{n(n-1)\dots(n-r+1)}\sum_{r_{1}+r_{3}=r}\binom{r}{r_{1}r_{3}}q^{r_{1}}(1-q)^{r_{3}}m\Delta_{1^{r_{1}},n-r,1^{r_{3}}};\\
\trintober_{n}(q_{1},q_{2,}q_{3}): & =\sum_{r_{1}+r_{2}+r_{3}=n}\frac{1}{r_{1}!r_{3}!}q_{1}^{r_{1}}q_{2}^{r_{2}}q_{3}^{r_{3}}m\Delta_{1^{r_{1}},r_{2},1^{r_{3}}}.
\end{align*}

\end{defn}
Observe that 
\begin{align*}
\ter_{n} & =\tober_{n}(1);\\
\trer_{n} & =\bintobrer_{n}(1);\\
\binter_{n}(q_{2}) & =\trintober_{n}(1-q_{2},q_{2},0).
\end{align*}
So it is no surprise that $\trintober_{n}$ also admits an equivalent
definition in terms of $\exp_{*}$ (which again will not be necessary
for this work):
\[
\sum_{n}\trintober_{n}(q_{1},q_{2,}q_{3}):=\left(\exp_{*}\left(\frac{q_{1}}{q_{2}}\Proj_{1}\right)\right)*\left(\sum_{n}q_{2}^{n}\Proj_{n}\right)*\left(\exp_{*}\left(\frac{q_{3}}{q_{2}}\Proj_{1}\right)\right).
\]

This more general triple of operators still enjoy simplified state
space basis axioms, namely the two-sided analogue of the conditions
for top-to-random: $\xi_{c,z}^{y},\xi_{z,c}^{y},\eta_{x}^{c,z},\eta_{x}^{z,c}\geq0$
for all $x,y,z\in\calb$ and all $c\in\calb_{1}$, and $\Delta_{1,\deg(x)-1}(x)\neq0$
for all $x\in\calb$. (An equivalent condition is $\Delta_{\deg(x)-1,1}(x)\neq0$
for all $x\in\calb$.)

The precise relationship between the top-or-bottom-to-random operators,
coming from the composition rule (Proposition \ref{prop:descentoperator-composition}),
is 
\begin{prop}[Relationship between top-or-bottom-to-random operators]
\label{prop:tob2r-poly} On all Hopf algebras,
\[
\trintober_{n}(q_{1},q_{2},q_{3})=\sum_{r=0}^{n}\binom{n}{r}q_{2}{}^{n-r}(1-q_{2})^{r}\bintobrer_{n}\left(\frac{q_{1}}{q_{1}+q_{3}}\right).
\]
And, on all commutative or cocommutative Hopf algebras, $\bintobrer_{n}(q)$
is a polynomial of degree $r$ in $\tober_{n}(q)$, namely the falling
factorial 
\[
\bintobrer_{n}=\left(\frac{n\tober_{n}}{n}\right)\circ\left(\frac{n\tober_{n}-\id}{n-1}\right)\circ\dots\circ\left(\frac{n\tober_{n}-(r-1)\id}{n-(r-1)}\right).
\]
Hence, on all commutative or cocommutative Hopf algebras, $\trintober_{n}(q_{1},q_{2},q_{3})$
is the polynomial
\[
\sum_{r=0}^{n}\binom{x}{r}q_{2}{}^{n-r}(1-q_{2})^{r}
\]
evaluated at $x=\tober_{n}\left(\frac{q_{1}}{q_{1}+q_{3}}\right)$.
\qed
\end{prop}
Note that, if $x$ is an integer between 0 and $n$, then the polynomial
above simplifies to $q_{2}^{n-x}$. (This is false if $x$ is outside
this range.) These powers of $q_{2}$ will turn out to be the eigenvalues
of $\trintober_{n}(q_{1},q_{2},q_{3})$.

\cite[Sec. 6 Ex. 6]{cppriffleshuffle} observed that, if $q_{1}=q_{3}$,
then $\trintober_{n}(q_{1},q_{2},q_{3})$ (on a commutative or cocommutative
Hopf algebra) spans a commutative subalgebra of operators. The polynomial
expression above shows that this is true whenever the ratio between
$q_{1}$ and $q_{3}$ is fixed. Hence the eigendata theorems below
apply to $\trintober_{n}(q_{1},q_{2},q_{3})$ chains where $q_{2}$
varies over time, or chains where different steps are driven by different
top-or-bottom-to-random operators, for instance alternating between
$\tober_{n}(\frac{q_{1}}{q_{1}+q_{3}})$ and $\trintober_{n}(q_{1},q_{2},q_{3})$. 
\begin{rem*}
Interestingly, the expression of $\trer$ as a polynomial in $\ter$
(i.e. the case $q=1$) still holds on a noncommutative and noncocommutative
Hopf algebra, so long as $\dim\calh_{1}=1$. This is a feature of
the theory of dual graded graphs \cite{dgg,dggandcha}, see \cite[Lem. 6.3]{hopfchainliftold}
for a proof in the language of combinatorial Hopf algebras. 
\end{rem*}

\subsection{A construction for eigenvectors\label{sub:tob2r-efns}}

The goal of this section is to prove Theorem \ref{thm:tob2r-evectors-cocommutative},
a particularly simple way to construct many eigenvectors for the top-or-bottom-to-random
maps. On a cocommutative Hopf algebra, all eigenvectors are of this
special form. By Proposition \ref{prop:doobtransform-efns}, this
leads to a full basis of left eigenfunctions for chains on cocommutative
Hopf algebras, and a full basis of right eigenfunctions for chains
on commutative Hopf algebras.

Two special cases of this theorem exist in the literature. For shuffling
a distinct deck of cards, \cite[Sec. 6, Exs. 1 and 4]{cppriffleshuffle}
identified the spectrum for $\trer_{n}$ and $\tober_{n}(\frac{1}{2})$
respectively. Unrelated, the eigenvectors for the top-to-random maps
$(q=1$) follow from applying Schocker's derangement idempotents \cite{derangementidems}
to symmetrised products of primitives, as described in the final remark
of the present Section \ref{sub:evalue}.
\begin{thm}[Eigenvectors for the top-or-bottom-to-random family]
\label{thm:tob2r-evectors-cocommutative}  Let $\calh=\bigoplus_{n\geq0}\calhn$
be a graded connected Hopf algebra over $\mathbb{R}$ with each $\calhn$
finite-dimensional. 
\begin{enumerate}
\item The distinct eigenvalues for the top-or-bottom-to-random operators
are $\beta_{j}$ where:

\begin{itemize}
\item for $\ter_{n}$ and $\tober_{n}(q)$, $\beta_{j}=\frac{j}{n}$ with
$j\in[0,n-2]\cup\{n\}$;
\item for $\trer_{n}$ and $\bintobrer_{n}(q)$, $\beta_{j}=\frac{j(j-1)\dots(j-r+1)}{n(n-1)\dots(n-r+1)}$
with $j\in\{0\}\cup[r,n-2]\cup\{n\}$;
\item for $\binter_{n}(q_{2})$ and $\trintober_{n}(q_{1},q_{2},q_{3})$,
$\beta_{j}=q_{2}^{n-j}$with $j\in[0,n-2]\cup\{n\}$.
\end{itemize}

The multiplicity of the eigenvalue $\beta_{j}$ is the coefficient
of $x^{n-j}y^{j}$ in $\left(\frac{1-x}{1-y}\right)^{\dim\calh_{1}}\sum_{n}\dim\calhn x^{n}$.
(Exception: for $\trer_{n}$ and $\bintobrer_{n}(q)$, the multiplicity
of $\beta_{0}$ is the sum of the coefficients of $x^{n},x^{n-1}y,\dots,x^{n-r+1}y^{r-1}$
in $\left(\frac{1-x}{1-y}\right)^{\dim\calh_{1}}\sum_{n}\dim\calhn x^{n}$.)
In particular, if $\calb_{1}=\{\bullet\}$, then these multiplicities
are $\dim\calh_{n-j}-\dim\calh_{n-j-1}$.

\item Fix $j\in[0,n-2]\cup\{n\}$. For any $p\in\calh_{n-j}$ satisfying
$\Delta_{1,n-j-1}(p)=0$, and any $c_{1},\dots,c_{j}\in\calh_{1}$
(not necessarily distinct), 
\[
\sum_{\sigma\in\sj}c_{\sigma(1)}\dots c_{\sigma(j)}p
\]
 is an eigenvector for:

\begin{itemize}
\item $\ter_{n}$ and $\binter_{n}(q_{2})$, with eigenvalue $\beta_{j}$;
\item $\trer_{n}$, with eigenvalue $\beta_{j}$ if $j\geq r$, and 0 otherwise.
\end{itemize}

For any $p\in\calh_{n-j}$ satisfying $\Delta_{1,n-j-1}(p)=\Delta_{n-j-1,1}(p)=0$,
and any $c_{1},\dots,c_{j}\in\calh_{1}$ (not necessarily distinct),
\[
\sum_{i=0}^{j}\sum_{\sigma\in\sj}\binom{j}{i}q^{i}(1-q)^{j-i}c_{\sigma(1)}\dots c_{\sigma(i)}pc_{\sigma(i+1)}\dots c_{\sigma(j)}
\]
 is an eigenvector for:
\begin{itemize}
\item $\tober_{n}(q)$, with eigenvalue $\beta_{j}$;
\item $\bintobrer_{n}(q)$, with eigenvalue $\beta_{j}$ if $j\geq r$,
and 0 otherwise.
\end{itemize}

and
\[
\sum_{i=0}^{j}\sum_{\sigma\in\sj}\binom{j}{i}q_{1}^{i}q_{3}{}^{j-i}c_{\sigma(1)}\dots c_{\sigma(i)}pc_{\sigma(i+1)}\dots c_{\sigma(j)}
\]
is a $\beta_{j}$-eigenvector for $\trintober_{n}(q_{1},q_{2,}q_{3})$.

\item Let $\calp$ be a (graded) basis of the primitive subspace of $\calh$.
Write $\calp$ as the disjoint union $\calp_{1}\amalg\calp_{>1}$,
where $\calp_{1}$ has degree 1. Set
\begin{align*}
\cale_{j}(1) & :=\left\{ \sum_{\sigma\in\sj}c_{\sigma(1)}\dots c_{\sigma(i)}\sum_{\tau\in\skj}p_{\tau(1)}\dots p_{\tau(k-j)}\right\} ,\\
\cale_{j}(q) & :=\left\{ \sum_{i=0}^{j}\sum_{\sigma\in\sj}\binom{j}{i}q^{i}(1-q)^{j-i}c_{\sigma(1)}\dots c_{\sigma(i)}\left(\sum_{\tau\in\skj}p_{\tau(1)}\dots p_{\tau(k-j)}\right)c_{\sigma(i+1)}\dots c_{\sigma(j)}\right\} ,\\
\cale_{j}(q_{1},q_{2},q_{3}) & :=\left\{ \sum_{i=0}^{j}\sum_{\sigma\in\sj}\binom{j}{i}q_{1}^{i}q_{3}{}^{j-i}c_{\sigma(1)}\dots c_{\sigma(i)}\left(\sum_{\tau\in\skj}p_{\tau(1)}\dots p_{\tau(k-j)}\right)c_{\sigma(i+1)}\dots c_{\sigma(j)}\right\} ,
\end{align*}
where each set ranges over all multisets $\{c_{1},\dots,c_{j}\}$
of $\calp_{1}$, and all multisets $\{p_{1},\dots,p_{k-j}\}$ of $\calp_{>1}$
with $\deg p_{1}+\dots+\deg p_{k-j}=n-j$. Then

\begin{itemize}
\item $\cale_{j}(1)$ (resp. $\cale_{j}(q)$, $\cale_{j}(q_{1},q_{2},q_{3})$)
consists of linearly independent $\beta_{j}$-eigenvectors for $\ter_{n}$
and $\binter_{n}(q_{2})$ (resp. $\tober_{n}(q)$, $\trintober_{n}(q_{1},q_{2},q_{3})$);
\item for $j\geq r$, the set $\cale_{j}(1)$ (resp. $\cale_{j}(q)$) consists
of linearly independent $\beta_{j}$-eigenvectors for $\trer_{n}$
(resp. $\bintobrer_{n}(q)$), and $\amalg_{j=0}^{r-1}\cale_{j}(1)$
(resp. $\amalg_{j=0}^{r-1}\cale_{j}(q)$) consists of linearly independent
$0$-eigenvectors for $\trer_{n}$ (resp. $\bintobrer_{n}(q)$).
\end{itemize}
\item In addition, if $\calh$ is cocommutative, then $\amalg_{j=0}^{n-2}\cale_{j}\amalg\cale_{n}$
is an eigenbasis for the above maps. 
\end{enumerate}
\end{thm}
The symmetrisation of the $p_{i}$ in iii above is unnecessary: its
only advantage is to put all $p_{i}$ in the chosen multiset on equal
footing. In other words, if the basis of primitives $\calp$ admits
a natural order $\preceq$, then setting $\cale_{j}(1)=\left\{ \sum_{\sigma\in\sj}c_{\sigma(1)}\dots c_{\sigma(i)}p_{1}\dots p_{k-j}\right\} $
over all multisets $\{c_{1},\dots,c_{j}\}$ of $\calp_{1}$, and all
multisets $\{p_{1}\preceq\dots\preceq p_{k-j}\}\subseteq\calp_{>1}$
with $\deg p_{1}+\dots+\deg p_{k-j}=n-j$, would also give linearly
independent eigenvectors (and similarly for $\cale_{j}(q)$ and $\cale_{j}(q_{1},q_{2},q_{3})$).
(Indeed, by Lemma \ref{lem:straightening} on the leading term under
the PBW straightening algorithm, it is possible to use any linear
combination of the products $p_{\tau(1)}\dots p_{\tau(k-j)}$ so long
as the coefficient sum is non-zero.) The symmetrisation of the $c_{i}$,
however, is necessary.
\begin{proof}
We start by proving the series of implications ii$\Rightarrow$iii$\Rightarrow$iv$\Rightarrow$i,
and tackle the proof of ii at the end.
\begin{description}
\item [{ii$\Rightarrow$iii:}] Taking $p=\sum_{\tau\in\skj}p_{\tau(1)}\dots p_{\tau(k-j)}$
shows that each $\cale_{j}(q)$ consists of eigenvectors. Their linear
independence follows from a PBW-straightening and triangularity argument
in the subalgebra $\calu(\calp)$, the universal enveloping algebra
of the primitive subspace (see the last paragraph of the first proof
of Theorem \ref{thm:evalues}). 
\item [{iii$\Rightarrow$iv:}] If $\calh$ is cocommutative, then $\calh=\calu(\calp)$
\cite[Th. 3.8.1]{cmm}, so the $\cale_{j}$ also span. 
\item [{iv$\Rightarrow$i:}] Recall from the proof of Theorem \ref{thm:evalues}
that the multiplicity of each eigenvalue on $\calh$ is equal to its
multiplicity on the cocommutative Hopf algebra $\gr(\calh)^{*}$.
Hence these multiplicities are the product of two numbers: the number
of multisets of $\calp_{1}$ (for $\gr(\calh)^{*}$) with $j$ elements,
and the number of multisets of $\calp_{>1}$ (for $\gr(\calh)^{*}$)
with total degree $n-j$. As in Theorem \ref{thm:evalues}, write
$b_{i}$ for the number of elements of degree $i$ in $\calp$ (for
$\gr(\calh)^{*}$), so $\sum_{n}\dim\calhn x^{n}=\prod_{i}(1-x_{i})^{-b_{i}}$.
Then the required multiplicities are the coefficients of $x^{n-j}y^{j}$
in $(1-y)^{-b_{1}}\prod_{i>1}(1-x_{i})^{-b_{i}}=\left(\frac{1-x}{1-y}\right)^{b_{1}}\sum_{n}\dim\calhn x^{n}$,
and $b_{1}=\dim\calh_{1}$, as noted in the proof of Theorem \ref{thm:stationarydistribution}.
\item [{Proof~of~ii:}] This can be checked by direct calculation, but
that doesn't explain where the formula comes from. So here's a more
circutous proof to demonstrate how one might discover such a formula.
(It is a more complicated version of the $q=1$ argument in Remark
1 after Lemma \ref{lem:symlemma}.) The proof first concentrates on
$\tober_{n}$. The result for $\bintobrer_{n}$ and $\trintober_{n}$
are not then immediate - Proposition \ref{prop:tob2r-poly} doesn't
apply as there is no commutativity or cocommutativity hypothesis.
A more careful argument (the final sentence of the proof) is necessary
to make this extension.

The Symmetrisation Lemma (Lemma \ref{lem:symlemma}) asserts that
there is an eigenvector for each multiset of primitives. Experiment
with a simplest case where this multiset is $\{c_{1},\dots,c_{j},p\}$
with $\deg(p)>1$. From an example such as 
\[
m\Delta_{1,n-1}(c_{1}c_{2}c_{3}pc_{4}c_{5})=c_{1}c_{2}c_{3}pc_{4}c_{5}+c_{2}c_{1}c_{3}pc_{4}c_{5}+c_{3}c_{1}c_{2}pc_{4}c_{5}+c_{4}c_{1}c_{2}c_{3}pc_{5}+c_{5}c_{1}c_{2}c_{3}pc_{4}
\]
(see Equation \ref{eq:descent-operator-on-product-of-primitives}
for the general formula of a descent operator acting on a product
of primitives), it is not hard to see that
\begin{align*}
 & m\Delta_{1,n-1}\left(\sum_{\sigma\in\sj}c_{\sigma(1)}\dots c_{\sigma(i)}pc_{\sigma(i+1)}\dots c_{\sigma(j)}\right)\\
= & i\left(\sum_{\sigma\in\sj}c_{\sigma(1)}\dots c_{\sigma(i)}pc_{\sigma(i+1)}\dots c_{\sigma(j)}\right)+(j-i)\left(\sum_{\sigma\in\sj}c_{\sigma(1)}\dots c_{\sigma(i+1)}pc_{\sigma(i+2)}\dots c_{\sigma(j)}\right),
\end{align*}
and 
\begin{align*}
 & m\Delta_{n-1,1}\left(\sum_{\sigma\in\sj}c_{\sigma(1)}\dots c_{\sigma(i)}pc_{\sigma(i+1)}\dots c_{\sigma(j)}\right)\\
= & i\left(\sum_{\sigma\in\sj}c_{\sigma(1)}\dots c_{\sigma(i-1)}pc_{\sigma(i)}\dots c_{\sigma(j)}\right)+(j-i)\left(\sum_{\sigma\in\sj}c_{\sigma(1)}\dots c_{\sigma(i)}pc_{\sigma(i+1)}\dots c_{\sigma(j)}\right).
\end{align*}
So, treating $\omega_{i}:=\sum_{\sigma\in\sj}c_{\sigma(1)}\dots c_{\sigma(i)}pc_{\sigma(i+1)}\dots c_{\sigma(j)}$
as a state, the linear map $\tober_{n}(q)=\frac{q}{n}m\Delta_{1,n-1}+\frac{1-q}{n}m\Delta_{n-1,1}$,
acting on $\sspan\{\omega_{i}|0\leq i\leq j\}$, induces (a multiple
by $\frac{j}{n}$ of) a \emph{birth-and-death process} \cite[Sec. 2.5]{markovmixing},
with transition ``probabilities''
\begin{align*}
\Prob(\omega_{i}\rightarrow\omega_{i+1}) & =\frac{q}{n}(j-i),\\
\Prob(\omega_{i}\rightarrow\omega_{i}) & =\frac{q}{n}i+\frac{1-q}{n}(j-i),\\
\Prob(\omega_{i}\rightarrow\omega_{i-1}) & =\frac{1-q}{n}i.
\end{align*}
The standard formula \cite[Prop. 2.8]{markovmixing} for the ``stationary
distribution'' of such a chain then simplifies to $\pi(\omega_{i})=\binom{j}{i}q^{i}(1-q)^{j-i}$.
So $\sum_{i}\pi(\omega_{i})\omega_{i}$ is an eigenvector for $\tober_{n}(q)$,
of eigenvalue $\frac{j}{n}$. Even without using the theory of birth-and-death
processes, it is not hard to guess a solution to the detailed balance
equations
\[
\pi(\omega_{i})\frac{q}{n}(j-i)=\pi(\omega_{i+1})\frac{1-q}{n}(i+1),\quad0\leq i<j.
\]
Note that this argument did not require $x$ to be primitive - thanks
to the compatibility between product and coproduct, $c_{\sigma(1)}\dots c_{\sigma(i)}xc_{\sigma(i+1)}\dots c_{\sigma(j)}$
will behave like a product of primitives under $\tober_{n}$, $\bintobrer_{n}$
and $\trintober_{n}$ (i.e. satisfying Equation \ref{eq:descent-operator-on-product-of-primitives})
so long as $\frac{q}{n}\Delta_{1,n-1}(x)=\frac{1-q}{n}\Delta_{n-1,1}(x)=0$.

\end{description}
\end{proof}
Note that, unless $c_{1}=\dots=c_{j}$, the chain on the $\omega_{i}$,
with $j+1$ states, is not the chain from the Symmetrisation Lemma,
which is on all orderings of the multiset $\{\bullet_{1},\dots,\bullet_{j},p\}$,
so in general has $(j+1)!$ states. The latter chain is not reversible
and does not have solutions to the detailed balance equation.
\begin{rem*}
In the case where $\dim\calh_{1}=1$ (so $\calb_{1}=\{\bullet\}$),
part iii of the Theorem above also follows from the eigenspace algorithm
for dual graded graphs \cite{dgg} since, on cocommutative Hopf algebras,
the maps $\U_{n}:\calhn\rightarrow\calh_{n+1}$ and $\D_{n}:\calhn\rightarrow\calh_{n-1}$
defined by 
\[
\U_{n}(x):=q\bullet x+(1-q)x\bullet;
\]
\[
\Delta_{1,n-1}(x)=\bullet\otimes\D_{n}(x)
\]
(so $\Delta_{n-1,1}(x)=\D_{n}(x)\otimes\bullet$ by cocommutativity)
satisfy the relation $\D_{n+1}\U_{n}-\U_{n-1}\D_{n}=\id$. This is
a $q$-deformation of the canonical dual graded graph structure on
a combinatorial Hopf algebra, as detailed in \cite{dggandcha}. Then
$\tober_{n}(q)=\U_{n-1}\D_{n}$. In the case $q=1$ (or $q=0$), cocommutativity
is not necessary - hence, whenever $\dim\calh_{1}=1$, part iii gives
an eigenbasis for $\ter_{n}$, $\trer_{n}$ and $\binter_{n}(q_{2})$.
\end{rem*}

\subsection{Eigenvectors for top-and-bottom-to-random chains\label{sub:tab2r}}

Much of the analysis in the last two sections generalises to a wider
class of chains, where the distribution $P$ is non-zero only on weak-compositions
with at most one part of size larger than 1. (That is, $P$ is non-zero
only on $(1^{r_{1}},r_{2},1^{r_{3}})$ and ``paddings'' of these
compositions by parts of size zero.) These chains model removing pieces
of size 1 from either ``end'' of a combinatorial object then reattaching
them. 

The conditions for $\calb$ to be a state space basis for these distributions
$P$ are the same simplified conditions for the top-or-bottom-to-random
maps. The chains again have at most $n$ eigenvalues, with the same
multiplicities as in Theorem \ref{thm:tob2r-evectors-cocommutative}.i.
The eigenvector algorithm goes through to a lesser extent: $m\Delta_{P}$
would still induce (a multiple of) a Markov chain on $W:=\sspan\{\omega_{i}:=\sum_{\sigma\in\sj}c_{\sigma(1)}\dots c_{\sigma(i)}pc_{\sigma(i+1)}\dots c_{\sigma(j)}|0\leq i\leq j\}$,
whose ``stationary distribution'' $\pi$ gives the eigenvector 
\[
\sum_{i=0}^{j}\sum_{\sigma\in\sj}\pi(\omega_{i})c_{\sigma(1)}\dots c_{\sigma(i)}pc_{\sigma(i+1)}\dots c_{\sigma(j)}.
\]
(As above, $p\in\ker\Delta_{1,n-j-1}\cap\ker\Delta_{n-j-1,1}$, and
$c_{1},\dots,c_{j}\in\calh_{1}$.) In theory, this would again give
a full eigenbasis on cocommutative Hopf algebras. However, the chains
on $W$ are in general not birth-and-death processes, and it appears
to be rather rare for such ``stationary distributions'' to have
as simple an expression as in Theorem \ref{thm:tob2r-evectors-cocommutative}.
We conclude this sectio with one exception.
\begin{defn}
\label{defn:tab2r} $ $
\begin{itemize}
\item The \emph{top-and-bottom-to-random} distribution is concentrated at
$(1,n-2,1)$.
\item The \emph{top-and-bottom-$r$-to-random} distribution is concentrated
at $(1^{r},n-2r,1^{r})$.
\end{itemize}

The related descent operators are:
\begin{align*}
\taber_{n}: & =\frac{1}{n(n-1)}m\Delta_{1,n-2,1};\\
\tabrer_{n}: & =\frac{1}{n(n-1)\dots(n-2r+1)}m\Delta_{1^{r},n-2r,1^{r}}.
\end{align*}

\end{defn}
Unlike the top-or-bottom-to-random maps, $\tabrer_{n}$ are not polynomials
and specialisations of a single descent operator. (Indeed, the eigenvectors
in Theorem \ref{thm:tab2r-evectors-cocommutative} below depend on
$r$.) Polynomials in $\taber_{n}$ involve compositions $(1^{r_{1}},r_{2},1^{r_{3}})$
where $r_{1}\neq r_{3}$, and the distributions associated to these
polynomials don't seem natural.

The detailed balance equations for the ``Markov chain'' that $\tabrer_{n}$
induces on $\sspan\{\omega_{i}:=\sum_{\sigma\in\sj}c_{\sigma(1)}\dots c_{\sigma(i)}pc_{\sigma(i+1)}\dots c_{\sigma(j)}|0\leq i\leq j\}$
read
\[
\pi(\omega_{i})\binom{j-i}{r+i'-i}\binom{i}{r-i'+i}=\pi(\omega_{i'})\binom{j-i'}{r+i-i'}\binom{i'}{r-i+i'},
\]
and admit the solution
\[
\pi(\omega_{i})=\binom{j-r}{i}\binom{j-r}{i-r}.
\]
(This $\pi$ is not normalised to be a distribution - i.e. $\sum_{i=0}^{j}\pi(\omega_{i})\neq1$
- but that is immaterial for the eigenvector construction.) This generates
the eigenvectors in the theorem below. In the case of card-shuffling,
the eigenvalues of $\taber_{n}$ previously appeared in \cite[Sec. 6 Ex. 5]{cppriffleshuffle}.
\begin{thm}[Eigenvectors for the top-and-bottom-to-random family]
\label{thm:tab2r-evectors-cocommutative} Let $\calh=\bigoplus_{n\geq0}\calhn$
be a graded connected Hopf algebra over $\mathbb{R}$ with each $\calhn$
finite-dimensional. 
\begin{enumerate}
\item The distinct eigenvalues for $\tabrer_{n}$ are $\beta_{j}=\frac{j(j-1)\dots(j-2r+1)}{n(n-1)\dots(n-2r+1)}$
with $j\in\{0\}\cup[2r,n-2]\cup\{n\}$. For $j\neq0$, the multiplicity
of $\beta_{j}$ is the coefficient of $x^{n-j}y^{j}$ in $\left(\frac{1-x}{1-y}\right)^{\dim\calh_{1}}\sum_{n}\dim\calhn x^{n}$.
The multiplicity of $\beta_{0}$ is the sum of the coefficients of
$x^{n},x^{n-1}y,\dots,x^{n-r+1}y^{r-1}$ in $\left(\frac{1-x}{1-y}\right)^{\dim\calh_{1}}\sum_{n}\dim\calhn x^{n}$.

\item For any $p\in\calh_{n-j}$ satisfying $\Delta_{1,n-j-1}(p)=\Delta_{n-j-1,1}(p)=0$,
and any $c_{1},\dots,c_{j}\in\calh_{1}$ (not necessarily distinct):

\begin{itemize}
\item if $j<2r$, then
\[
\sum_{\sigma\in\sj}c_{\sigma(1)}\dots c_{\sigma(j)}p
\]
 is a 0-eigenvector for $\tabrer_{n}$.
\item if $j\geq2r$, then
\[
\sum_{i=r}^{j-r}\sum_{\sigma\in\sj}\binom{j-r}{i}\binom{j-r}{i-r}c_{\sigma(1)}\dots c_{\sigma(i)}pc_{\sigma(i+1)}\dots c_{\sigma(j)}
\]
 is a $\beta_{j}$-eigenvector for $\tabrer_{n}$.
\end{itemize}
\item Let $\calp$ be a (graded) basis of the primitive subspace of $\calh$.
Write $\calp$ as the disjoint union $\calp_{1}\amalg\calp_{>1}$,
where $\calp_{1}$ has degree 1. Set
\[
\cale_{j}:=\left\{ \left(\sum_{\sigma\in\sj}c_{\sigma(1)}\dots c_{\sigma(j)}\right)\left(\sum_{\tau\in\skj}p_{\tau(1)}\dots p_{\tau(k-j)}\right)\right\} 
\]
if $j<2r$, and 
\[
\cale_{j}:=\left\{ \sum_{i=r}^{j-r}\sum_{\sigma\in\sj}\binom{j-r}{i}\binom{j-r}{i-r}c_{\sigma(1)}\dots c_{\sigma(i)}\left(\sum_{\tau\in\skj}p_{\tau(1)}\dots p_{\tau(k-j)}\right)c_{\sigma(i+1)}\dots c_{\sigma(j)}\right\} ,
\]
if $j\geq2r$, ranging (in both cases) over all multisets $\{c_{1},\dots,c_{j}\}$
of $\calp_{1}$, and all multisets $\{p_{1},\dots,p_{k-j}\}$ of $\calp_{>1}$
with $\deg p_{1}+\dots+\deg p_{k-j}=n-j$. Then $\amalg_{j=0}^{2r-1}\cale_{j}$
consists of linearly independent 0-eigenvectors for $\tabrer_{n}$,
and, for $j>2r$, $\cale_{j}$ consists of linearly independent $\beta_{j}$-eigenvectors
for $\tabrer_{n}$. 
\item In addition, if $\calh$ is cocommutative, then $\amalg_{j=0}^{n-2}\cale_{j}\amalg\cale_{n}$
is an eigenbasis for $\tabrer_{n}$.\qed 
\end{enumerate}
\end{thm}
As in the case of the top-or-bottom-to-random shuffles (Theorem \ref{thm:tob2r-evectors-cocommutative}),
the symmetrisation of the high degree primitives $p_{i}$ in part
iii is unnecessary. The symmetrisation of the $c_{i}$ in the $j<2r$
case, in both parts ii and iii, are also unnecessary - indeed, any
linear combination of products in $c_{i}$ and $p_{i}$ is a 0-eigenvector
when $j<2r$.

\subsection{Recursive lumping property}

Return to the top-to-random chains of Definition \ref{sub:tob2r}.
Assume that the underlying Hopf algebra $\calh$ satisfies $\calh_{1}=\sspan\{\bullet\}$,
together with a new condition: for each $x\in\calbn$, we have $\Delta_{1,n-1}(x)=\bullet\otimes x'$
for $x'\in\calb_{n-1}$. (Note this forces $\eta\equiv1$.) Then,
there is a well-defined map $D:\calbn\rightarrow\calb_{n-1}$ satisfying
$\Delta_{1,n-1}(x)=\bullet\otimes D(x)$. (This $D$ is the down operator
of the dual graded graph associated to $\calh$ in \cite{dggandcha}'s
construction.) Iterates $D^{n-k}$ of $D$ satisfy $\Delta_{1^{n-k},k}(x)=\bullet\otimes\dots\otimes\bullet\otimes D^{n-k}(x)$.
View $D^{n-k}$ as a ``forgetful function'' on $\calbn$, observing
only a size $k$ part of these size $n$ objects.

Theorem \ref{thm:recursivelumping} below proves that the image under
$D^{n-k}$ of the top-to-random chain on $\calbn$ is a lazy version
of the analogous chain on $\calb_{k}$. See Theorem \ref{thm:recursivelumping-fqsym}
for an intuitive card-shuffling example. Informally, observing a subobject
of the chain gives a smaller copy of the same chain (with laziness),
hence ``recursive''. In order to state this result precisely, the
following definitions are necessary: 
\begin{defn}
\label{def:lazychain} Given a Markov chain $\{X_{t}\}$ with transition
matrix $K$ and a number $\alpha\in[0,1]$, the \emph{$\alpha$-lazy
version of $\{X_{t}\}$} is the Markov chain with transition matrix
$\alpha\id+(1-\alpha)K$.
\end{defn}
Equivalently, at each time step, the $\alpha$-lazy version of $\{X_{t}\}$
stays at the same state with probability $\alpha$, and with probability
$1-\alpha$ it evolves according to $\{X_{t}\}$.

From the ``basis scaling'' interpretation of the Doob transform
($\hatk:=\left[\T\right]_{\hatcalb}^{T}$), it is immediate that the
Doob transform ``commutes with lazying'' in the following sense:
\begin{lem}
\label{lem:dooblazy} Let $\{X_{t}\}$ be the Markov chain on $\calb$
driven by $\T$ with rescaling function $\eta$. The $\alpha$-lazy
version of $\{X_{t}\}$ is driven by $\alpha\id+(1-\alpha)\T$, with
the same rescaling function $\eta$.\qed
\end{lem}
Since $\id:\calhn\rightarrow\calhn$ may be expressed as the descent
operator $m\Delta_{(n)}$, a lazy version of a descent operator chain
is also a descent operator chain.
\begin{defn}
\label{def:lumping} Let $\{X_{t}\},\{\bar{X}_{t}\}$ be Markov chains
on state spaces $\Omega,\bar{\Omega}$ respectively, and let $\theta:\Omega\rightarrow\bar{\Omega}$
be a surjection. Then $\{X_{t}\}$ is said to \emph{lump via $\theta$
to }$\{\bar{X}_{t}\}$ if the process $\{\theta(X_{t})\}$ is a Markov
chain with the same transition matrix as $\{\bar{X}_{t}\}$. \end{defn}
\begin{thm}[Recursive lumping of top-to-random chains]
 \label{thm:recursivelumping} Let $\calh$ be a graded connected
Hopf algebra with basis $\calb$ such that $\calb_{1}=\{\bullet\}$,
and $\Delta_{1,n-1}(x)=\bullet\otimes D(x)$ for some $D:\calbn\rightarrow\calb_{n-1}$.
Then: 
\begin{enumerate}
\item the Markov chain on $\calbn$ driven by $\ter_{n}$ lumps via $D^{n-k}$
to the $\frac{n-k}{n}$-lazy version of the chain on $\calb_{k}$
driven by $\ter_{k}$;
\item the Markov chain on $\calbn$ driven by $\binter_{n}(q_{2})$ lumps
via $D^{n-k}$ to the chain on $\calb_{k}$ driven by $\binter_{k}(q_{2})$.
\end{enumerate}
\end{thm}
\begin{proof}
For i: by Lemma \ref{lem:dooblazy}, we wish to prove that the chain
driven by $\ter_{n}=\frac{1}{n}m\Delta_{1,n-1}:\calbn\rightarrow\calbn$
lumps via $D^{n-k}$ to the chain driven by $\frac{n-k}{n}\id+\frac{k}{n}\ter_{k}=\frac{n-k}{n}\id+\frac{1}{n}m\Delta_{1,k-1}:\calb_{k}\rightarrow\calb_{k}$.
By \cite[Th. 2.7]{hopfchainlift}, it suffices to show that 
\begin{equation}
D^{n-k}\circ\left(\frac{1}{n}m\Delta_{1,n-1}\right)=\left(\frac{n-k}{n}\id+\frac{1}{n}m\Delta_{1,k-1}\right)\circ D^{n-k};\label{eq:recursivelumping}
\end{equation}
the other hypotheses are clearly satisfied. As noted at the beginning
of this section, $D$ is the down operator of a dual graded graph
following the general construction in \cite{dggandcha}. The corresponding
up operator $U$ satisfies $m_{1,n-1}(\bullet\otimes x)=U(x)$. Hence
Equation \ref{eq:recursivelumping} may be rephrased in terms of dual
graded graphs as follows:
\[
\frac{1}{n}D^{n-k}\circ U\circ D=\frac{1}{n}\left((n-k)\id+U\circ D\right)\circ D^{n-k}.
\]
And this is proved by repeated application of the dual graded graph
condition $DU=UD+\id$:
\begin{align*}
D^{n-k}\circ U\circ D & =D^{n-k-1}\circ(UD+\id)\circ D\\
 & =D^{n-k-2}\circ(DU+\id)\circ D^{2}=D^{n-k-2}\circ((UD+\id)+\id)\circ D^{2}\\
 & \hphantom{=D^{n-k-2}\circ(DU+\id)\circ D^{2}}=D^{n-k-3}\circ(DU+2\id)\circ D^{3}=\dots\\
 & \hphantom{=D^{n-k-2}\circ(DU+\id)\circ D^{2}=D^{n-k-3}\circ(DU+2\id)\circ D^{3}}=\left((n-k)\id+U\circ D\right)\circ D^{n-k}.
\end{align*}

For ii: by \cite[Th. 2.7]{hopfchainlift}, we wish to show 
\[
D^{n-k}\circ\binter_{n}=\binter_{k}\circ D^{n-k}
\]
which, by Proposition \ref{prop:tob2r-poly}, is equivalent to
\begin{align}
 & D^{n-k}\circ\left(\sum_{r=0}^{n}\binom{n}{r}q_{2}{}^{n-r}(1-q_{2})^{r}\left(\frac{n\ter_{n}}{n}\right)\circ\left(\frac{n\ter_{n}-\id}{n-1}\right)\circ\dots\circ\left(\frac{n\ter_{n}-(r-1)\id}{n-(r-1)}\right)\right)\nonumber \\
= & \left(\sum_{r=0}^{k}\binom{k}{r}q_{2}{}^{k-r}(1-q_{2})^{r}\left(\frac{k\ter_{k}}{k}\right)\circ\left(\frac{k\ter_{k}-\id}{k-1}\right)\circ\dots\circ\left(\frac{k\ter_{k}-(r-1)\id}{k-(r-1)}\right)\right)\circ D^{n-k}.\label{eq:binterlumping}
\end{align}
From Equation \ref{eq:recursivelumping}: 
\[
D^{n-k}\circ\left(n\ter_{n}-i\id\right)=\left((n-k-i)\id+k\ter_{k}\right)\circ D^{n-k},
\]
and so the left hand side of Equation \ref{eq:binterlumping} is 
\begin{align*}
 & \left(\sum_{r=0}^{n}\frac{1}{r!}q_{2}{}^{n-r}(1-q_{2})^{r}\left((n-k)\id+k\ter_{k}\right)\circ\left((n-k-1)\id+k\ter_{k}\right)\circ\dots\circ\left((n-k-(r-1))\id+k\ter_{k}\right)\right)\circ D^{n-k}\\
= & \left(\sum_{r=0}^{n}\binom{n-k+x}{r}q_{2}{}^{n-r}(1-q_{2})^{r}\right)\circ D^{n-k}\mbox{ evaluated at }x=k\ter_{k}.
\end{align*}
For every integer $x>k$, we have $\binom{n-k+x}{r}=\sum_{r'=0}^{r}\binom{n-k}{r-r'}\binom{x}{r'}$,
so this must be a polynomial identity in $x$. Hence the above is
\begin{align*}
= & \left(\sum_{r=0}^{n}\sum_{r'=0}^{r}\binom{n-k}{r-r'}\binom{x}{r'}q_{2}{}^{n-r}(1-q_{2})^{r}\right)\circ D^{n-k}\mbox{ evaluated at }x=k\ter_{k}\\
= & \left(\sum_{r'=0}^{n}\binom{x}{r'}q_{2}{}^{k-r'}(1-q_{2})^{r'}\sum_{r=r'}^{n}\binom{n-k}{r-r'}q_{2}{}^{n-k-(r-r')}(1-q_{2})^{r-r'}\right)\circ D^{n-k}\mbox{ evaluated at }x=k\ter_{k}\\
= & \left(\sum_{r'=0}^{k}\binom{x}{r'}q_{2}{}^{k-r'}(1-q_{2})^{r'}\sum_{r-r'=0}^{r-r'=n-r'}\binom{n-k}{r-r'}q_{2}{}^{n-k-(r-r')}(1-q_{2})^{r-r'}\right)\circ D^{n-k}\mbox{ evaluated at }x=k\ter_{k},
\end{align*}
where we have neglected the terms with $r'>k$ because then evaluating
$\binom{x}{r'}$ at $x=k\ter_{k}$ gives a factor $k\ter_{k}-i\id$
for all the eigenvalues $i=0,1,\dots k$ of $k\ter_{k}$, hence the
evaluation is 0. And when $r'\leq k$, the second sum runs up to $r-r'=n-r'\geq n-k$,
but $\binom{n-k}{r-r'}=0$ for $r-r'>n-k$, so we can truncate the
sum at $n-k$. Then this second sum yields 1, and the remainder of
the expression is exactly the right hand side of Equation \ref{eq:binterlumping}.

\end{proof}

\section{A Chain on Organisational Structures\label{sec:cktrees}}

The goal of this section is to employ the present Hopf-algebraic framework
to analyse a ``leaf-removal'' or ``employee firing'' chain, as
outlined in Example \ref{ex:rooted-trees}. We note that, as a result
of \cite[Th. 1]{treehookwalk} (repeated as Theorem \ref{thm:treehookwalk}
below), our leaf-removal step also occurs in the chain of \cite{treechainjason},
where it is followed by a leaf-attachment step that is absent here.

The situation is as follows: A company has a tree structure, so each
employee except the boss has exactly one direct superior. Each month,
some employees are fired (details in three paragraphs below), and
each firing independently causes a cascade of promotions: first, someone
further down the chain of superiority from the fired employee is uniformly
selected to replace him. Then, if the promoted employee was superior
to anyone, then one of those is uniformly selected and promoted to
his position. This process continues until someone who is not superior
to anyone (a leaf) is promoted. Figure \ref{fig:cktrees-firing} shows
the probabilities of all the possible scenarios after C is fired.

\begin{figure}
\setlength{\unitlength}{1mm} \begin{center} \begin{picture} (150,80) (15,135)   \put(62.5,190){\firing}  \put(65,195){\vector(-4, -5){20}} \put(52,184){$\frac{1}{4}$} \put(30,150){\firingone}  \put(70,195){\vector(-2, -5){10}} \put(62,184){$\frac{1}{4}$} \put(52,150){\firingtwo}  \put(75,195){\vector(2, -5){10}} \put(81,184){$\frac{1}{4}$} \put(78,150){\firingthree} \put(85,155){\vector(0, -1){12}} \put(88,150){$1$} \put(78,125){\firingthreeb}  \put(80,195){\vector(4, -5){20}} \put(91,184){$\frac{1}{4}$} \put(96,150){\firingfour}  \end{picture} \end{center}\protect\caption{\label{fig:cktrees-firing}All possible promotion scenarios after
an employee is fired, and their respective probabilities.}
\end{figure}
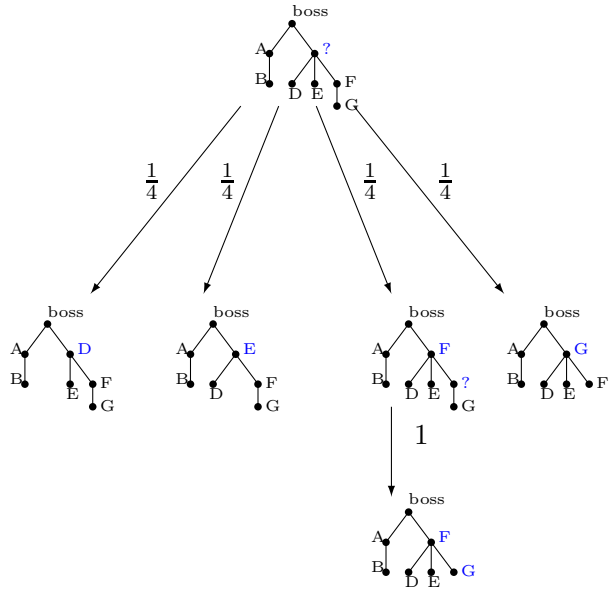

The chain keeps track of the tree structure of the company, but does
not know which employee is taking which position. For example the
rightmost two possibilities in Figure \ref{fig:cktrees-firing} represent
the same state. More specifically, if $T_{0}$ denotes the starting
state, then the state space of the chain is the set of \emph{rooted
subtrees} of $T_{0}$ - that is, the connected subsets of vertices
of $T_{0}$ which include the root (the boss). Let $T_{t}$ denote
the structure of the company after $t$ months, and $n_{t}$ the number
of employees after $t$ months (the number of vertices in $T_{t}$).

We consider two ways in which employees are fired: 
\begin{itemize}
\item The \emph{single model}: with probability $\frac{n_{t}}{n_{0}}$,
one uniformly chosen employee is fired at the start of month $t+1$.
(It does not matter whether the fired employee is chosen from all
$n_{t}$ employees or from only the $n_{t}-1$ non-boss employees;
the promotion cascades ensure that both options result in the same
chain.) With the complementary probability $\frac{n_{0}-n_{t}}{n_{0}}$,
there is no change to the company structure. 
\item The \emph{binomial model:} the monthly performance of each employee
(including the boss) is, independently, uniformly distributed between
0 and 1. For a fixed parameter $q_{2}$, all employees with performance
below $1-q_{2}$ are fired one-by-one in a random order. Hence the
number of fired employees follows a binomial distribution with parameter
$1-q_{2}$.
\end{itemize}
In both models, no firing occurs when only the boss remains - this
is the unique absorbing state. Below are the transition matrices for
these two models, starting from the four-person company on the far
right. (All empty entries are zeroes.){\arraycolsep=2pt\def\arraystretch{1.5}

\[
\begin{array}{cccccc}
\basisone & \basistwo & \basisthree & \basisfour & \basisfive & \basissix\\
\hline 1\\
\frac{1}{2} & \frac{1}{2}\\
\frac{1}{2} &  & \frac{1}{2}\\
 &  & \frac{3}{4} & \frac{1}{4}\\
 & \frac{3}{8} & \frac{3}{8} &  & \frac{1}{4}\\
 &  &  & \frac{1}{3} & \frac{2}{3}
\end{array}
\]
}
{\def\arraystretch{1.2}
\[
\begin{array}{cccccc}
\basisone & \basistwo & \basisthree & \basisfour & \basisfive & \basissix\\
\hline 1\\
(1-q)(1+q) & q^{2}\\
(1-q)(1+q) &  & q^{2}\\
(1-q)^{2}(1+3q) &  & 3q^{2}(1-q) & q^{3}\\
(1-q)^{2}(1+3q) & \frac{3}{2}q^{2}(1-q) & \frac{3}{2}q^{2}(1-q) &  & q^{3}\\
(1-q)^{3}(1+4q) & 2q^{2}(1-q)^{2} & 4q^{2}(1-q)^{2} & \frac{4}{3}q^{3}(1-q) & \frac{8}{3}q^{3}(1-q) & q^{4}
\end{array}
\]
}

Section \ref{sub:chain-cktrees} recasts these chains as the $\ter_{n_{0}}$
and $\binter_{n_{0}}$ chains repsectively on a decorated variant
of the Connes-Kreimer Hopf algebra of rooted trees. Section \ref{sub:ebasis-cktrees}
applies Theorem \ref{thm:tob2r-evectors-cocommutative} to produce
a full basis of right eigenfunctions (Theorem \ref{thm:spectrum-cktrees}),
using which we bound the expected numbers of ``inter-departmental
teams'' (Corollary \ref{cor:expectation-cktrees}). Section \ref{sub:trintober-cktrees}
extends the binomial model to a $\trintober$ chain, where employees
with outstanding work are promoted to a special status, independent
of the firings.

Another possible extension, not discussed here, is to give the company
a more general poset structure, where each employee may have multiple
direct supervisors. This uses a Hopf algebra of (unranked) posets
\cite[Ex. 2.3]{abs} which contains the Connes-Kreimer algebra of
trees. If the initial company structure $T_{0}$ and all its order
ideals are \emph{$d$-complete} \cite{dcompleteposets}, then the
promotion cascade algorithm above agrees with the $\ter$ and $\binter$
chains on the poset algebra. However, for more general posets, there
is no known method to generate a uniform linear extension, and it
is unclear how to define a firing/promotion process to keep the Hopf-algebraic
connection. 

We note also that much of the following analysis can be easily adapted
for Hopf algebras whose state space basis is a free-commutative monoid
(what \cite[Chap. 5]{mythesis} termed a ``free-commutative'' state
space basis.) The class of associated chains includes the rock-chipping
process of the introduction (a $\ter$ analogue of \cite[Sec. 4]{hopfpowerchains})
and the mining variant. They do not have a combining step. A few results
below require the additional hypothesis that $\Delta_{1,n-1}$ ``preserves
atomicity'' (i.e. all terms in $\Delta_{1,n-1}$ of a tree has a
tree in the second factor, as opposed to forests with more than one
connected component.) This is true for rock-chipping, but not for
mining.

\subsection{A connection to a decorated Connes-Kreimer Hopf algebra of trees\label{sub:chain-cktrees}}

The Connes-Kreimer Hopf algebra of trees arose independently from
the study of renormalisation in quantum field theory \cite{cktrees}
and of Runge-Kutta numerical methods of solving ordinary differential
equations \cite{rungekutta}. Relevant here is the decorated variant
from \cite[end of Sec. 2.1]{decoratedcktrees}, where each vertex
is labelled. In the present application, the labels are the job positions,
not the employees currently holding each position. These labels are
necessary to distinguish abstractly isomorphic trees without reference
to the starting state $T_{0}$. 

Take a quotient of this decorated tree algebra so the root becomes
unlabelled. This ensures that there is a unique rooted forest on one
vertex - call it $\bullet$ - which will simplify the notation slightly.
Below are the aspects of this Hopf algebra relevant to the present
Markov chain application; its full Hopf structure follows easily from
\cite[Sec. 2]{cktrees2}. (It is probably possible to run the subsequent
analysis using a Hopf monoid \cite[Sec. 13.3.1]{hopfmonoids} instead.)
\begin{itemize}
\item A basis is the set of all decorated rooted forests - that is, each
connected component has a distinguished root vertex, and each non-root
vertex is assigned one of a finite set of labels. (In contrast to
the concept of increasing labelling below, each label can appear on
multiple vertices of a forest, or not at all.)
\item The degree of a forest is its number of vertices.
\item The product of two forests is their disjoint union (preserving all
labels). Hence this Hopf algebra is commutative.
\item The partial coproduct of a forest $x$ on $n$ vertices is $\Delta_{1,n-1}(x)=\sum_{v}\bullet\otimes x\backslash v$
(preserving all labels in the second factor), where the sum runs over
all leaves $v$ of $x$. For example, 
\[
\Delta_{1,7}\left(\treeex\right)=\bullet\otimes\left(\coprodone+\coprodtwo+\coprodthree+\coprodfour\right).
\]

\end{itemize}

Because of coassociativity, $\Delta_{1,\dots,1,n-r}(x)=\sum\bullet^{\otimes r}\otimes x\backslash\{v_{1},\dots v_{r}\}$,
summing over all choices of $v_{1},\dots,v_{r}$ such that $v_{i}$
is a leaf of $x\backslash\{v_{1},\dots,v_{i-1}\}$. In particular,
the coefficient of $\bullet^{\otimes n}$ in $\Delta_{1,\dots,1}(x)$
enumerates the ways to successively remove leaves from $x$, or, equivalently,
the \emph{increasing labellings} of the vertices of $x$ - that is,
each of the labels $1,2,\dots,n$ occur once, and the label of a parent
is less than the label of the child. An example is in Figure \ref{fig:increasinglabelling}.
By Theorem \ref{thm:cppchain}, this coefficient is the reweighting
function $\eta$. (In the more general case of posets, it is the number
of linear extensions.) \cite[Eq. 1.1]{treehookwalk} gives a hook
length formula for the number of increasing labellings of a forest
$x$: 
\begin{equation}
\eta(x)=\frac{\deg x!}{\prod_{v\in x}h(v)},\label{eq:eta-cktree}
\end{equation}
where $h(v)$ is the number of vertices in the \emph{hook} $H(v)$
of vertex $v$: the hook consists of $v$, its children, its grandchildren,
... ; see the top of Figure \ref{fig:treehookwalk}. (In the interpretation
of a tree as a company structure, the hook of employee $v$ consists
of everyone further down in superiority from $v$, including $v$
himself.) One proof of this formula goes via an algorithm for uniformly
generating an increasing labelling, which corresponds to the ``promotion
cascade'' process described above.

\begin{figure}
\setlength{\unitlength}{2mm} 
\begin{center} 
\begin{picture} (15,12) (-3,-3)   
\put(5,7){\circle*{1}}  
\put(5,8){1}  
\put(5,7){\line(-3, -4){3}} 
\put(2,3){\circle*{1}} 
\put(0,3){3}  
\put(2,3){\line(0, -1){4}} 
\put(2,-1){\circle*{1}} 
\put(0,-1){8}  
\put(5,7){\line(3, -4){3}} 
\put(8,3){\circle*{1}} \put(9,3){2}  
\put(8,3){\line(-3, -4){3}} 
\put(5,-1){\circle*{1}} 
\put(4.5,-3){5}  
\put(8,3){\line(0, -1){4}} 
\put(8,-1){\circle*{1}} 
\put(7.5,-3){7}  
\put(8,3){\line(3, -4){3}} 
\put(11,-1){\circle*{1}} 
\put(12,-1){4}  
\put(11,-1){\line(0, -1){3}} 
\put(11,-4){\circle*{1}} 
\put(12,-4){6}  
\end{picture} 
\end{center}\protect\caption{\label{fig:increasinglabelling}An increasing labelling of a tree.}
\end{figure}
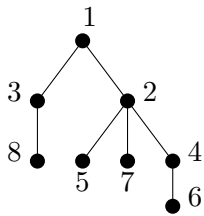

\begin{thm}[Sagan-Yeh hook walk]
 \label{thm:treehookwalk} \cite[Th. 1]{treehookwalk} For a tree
$T$ on $n$ vertices, the \emph{Sagan-Yeh hook walk} is the following
recursive process:
\begin{enumerate}[label=\arabic*.]
\item  Choose a vertex $v$ of $T$ uniformly.
\item If $v$ is a leaf, assign the label $n$ to $v$.
\item Else, uniformly choose a vertex $w$ in the hook of $v$, and return
to step 2 with $w$ in place of $v$.
\end{enumerate}

Hence the walk terminates with some leaf receiving the label $n$.
Remove this leaf and re-apply the walk to the remaining tree on $n-1$
vertices, and repeat until the root is assigned the label 1. This
generates an increasing labelling of $T$ uniformly.\qed

\end{thm}
\begin{figure}
\setlength{\unitlength}{1mm} \begin{center} \begin{picture} (150,50) (15,160)   \put(62.5,190){\hookwalk}  \put(65,195){\vector(-4, -5){20}} \put(52,184){$\frac{1}{4}$} \put(30,152){\hookwalkone}  \put(70,195){\vector(-2, -5){10}} \put(62,184){$\frac{1}{4}$} \put(52,152){\hookwalktwo}  \put(75,195){\vector(1, -5){5}} \put(78,184){$\frac{1}{4}$} \put(72,152){\hookwalkthree} \put(90,165){\vector(1, 0){8}} \put(93,166){$1$} \put(80,195){\vector(4, -5){20}} \put(91,184){$\frac{1}{4}$} \put(96,152){\hookwalkfour}  \end{picture} \end{center}\protect\caption{\label{fig:treehookwalk}Steps 2 and 3 of the Sagan-Yeh hook walk,
for this choice of vertex $v$. The blue vertices in the top diagram
are the hook of $v$.}
\end{figure}
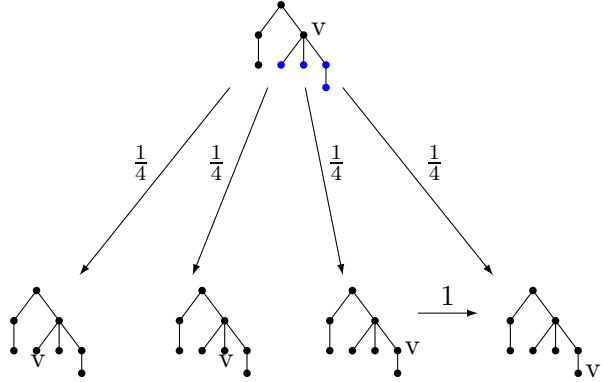

The remainder of this section will prove the following.
\begin{thm}[Interpretation of top-to-random chains on trees]
 \label{thm:chain-cktrees} The chains driven by $\ter_{n_{0}}$
and $\binter_{n_{0}}(q_{2})$ on the decorated Connes-Kreimer Hopf
algebra of trees, starting from a tree $T_{0}$ on $n_{0}$ vertices,
remain in the subset of states given by 
\[
\{\bullet^{n_{0}-n}\amalg T|T\mbox{ is a rooted subtree of }T_{0}\mbox{ on }n\mbox{ vertices}\}.
\]
Under the equivalence $\bullet^{n_{0}-n}\amalg T\rightarrow T$, these
chains have the ``employee firing'' description given at the beginning
of Section \ref{sec:cktrees}.
\end{thm}
It is crucial that all Hopf-algebraic calculations use the form $\bullet^{n_{0}-n}\amalg T$
as opposed to $T$. And note that the equivalence sends $\bullet^{n_{0}}$
to $\bullet$, not to the empty tree.
\begin{proof}
Begin by substituting the definition of product and coproduct into
the 3-step description of Theorem \ref{thm:cppchain}. For $\ter_{n_{0}}$,
the first and third steps are trivial, and the second translates to:
for each leaf $v$ of a forest $x$ on $n_{0}$ vertices, move from
$x$ to $v\amalg x\backslash v$ with probability $\frac{\eta(v)\eta(x\backslash v)}{\eta(x)}=\frac{\eta(x\backslash v)}{\eta(x)}$.
(It is impossible to move to forests not of the form $x\backslash v$.)
Notice that the increasing labellings of $x\backslash v$ are precisely
the increasing labellings of $x$ where $v$ has label $n_{0}$. Hence
each step of the chain is: uniformly pick an increasing labelling
of $x$, and isolate the vertex receiving label $n_{0}$. Similarly,
the $\binter_{n_{0}}(q_{2})$ chain uniformly picks an increasing
labelling of $x$, then isolates the vertices with labels $n_{0},n_{0}-1,\dots,n_{0}-r+1$,
where $r\in[0,n_{0}]$ has a binomial distribution with parameter
$1-q_{2}$.

Now specialise to the case where the rooted forest $x$ is of the
form $\bullet^{n_{0}-n}\amalg T$, for $T$ a tree on $n$ vertices.
The key observation is this: a uniform choice of an increasing labelling
of $\bullet^{n_{0}-n}\amalg T$ is equivalent to a uniform (ordered)
choice of $n_{0}-n$ distinct labels from $\{1,\dots,n_{0}\}$ for
the singletons, plus an independent uniform choice of increasing labelling
for $T$ (the ``standardisation'' of the original labelling for
$\bullet^{n_{0}-n}\amalg T$). 

In the case of $\ter_{n_{0}}$, there is a probability of $\frac{n_{0}-n}{n_{0}}$
that the label $n_{0}$ is amongst those chosen for the singletons:
in this case, the chain remains at $\bullet^{n_{0}-n}\amalg T$. With
the complementary probability $\frac{n}{n_{0}}$, the chain isolates
the vertex with the largest label in a random increasing labelling
of $T$, which, by Theorem \ref{thm:treehookwalk}, is precisely the
firing of a uniformly chosen employee and the subsequent promotion
cascade.

As for the $\binter_{n_{0}}$ chain: if $r'$ of the $r$ labels $n_{0},n_{0}-1,\dots,n_{0}-r+1$
were assigned to singletons, then the chain would remove $r-r'$ vertices
from $T$ according to the hook walk / promotion cascade. Hence it
suffices to show that, if $r\in[0,n_{0}]$ follows a binomial distribution
with parameter $1-q_{2}$, then $r-r'\in[0,n]$ is also binomial with
parameter $1-q_{2}$. Because the choice of labels for the singletons
is uniform (independent of the value of the label), the situation
has this alternative description: pick $n-n_{0}$ distinct labels
from $n$ labels, and $r'$ is the number of $r$ ``labels of interest''
that were picked. $r$ having a binomial distribution is equivalent
to each label independently having probability $1-q_{2}$ of being
a ``label of interest''. Thus $r'$ is binomial with parameter $1-q_{2}$,
and so is the number of unpicked labels of interest $r-r'$.

(More generally, for forests, $\ter_{n_{0}}$ selects a tree with
probability proportional to its number of vertices and removes a vertex
as per the hook walk. $\binter_{n_{0}}$ removes a binomial number
of vertices independently from each tree in the forest.)
\end{proof}

\subsection{The spectrum and right eigenbasis\label{sub:ebasis-cktrees}}

This section seeks to prove the following:
\begin{thm}[Eigenfunctions of employee-firing chains]
 \label{thm:spectrum-cktrees}Consider the chain on organisational
structures, started from $T_{0}$ with $n_{0}$ employees/vertices.
\begin{enumerate}
\item The eigenvalues of the single and binomial models are $\frac{j}{n_{0}}$
and $q_{2}^{n_{0}-j}$ respectively, for $0\leq j\leq n_{0}-2$ and
$j=n_{0}$.
\item The eigenvalue 1 ($j=n_{0}$) has multiplicity 1. For $j<n_{0}$,
the multiplicity of the eigenvalue indexed by $j$ is the number of
rooted subtrees of $T_{0}$ on $n_{0}-j$ vertices.
\item For $T'\neq\bullet$, the right eigenfunction $\f_{T'}$ corresponding
to the rooted subtree $T'$ on $n'$ vertices is
\[
\f_{T'}:=\binom{n}{n'}\Prob\left(\substack{\mbox{After firing }n-n'\mbox{ uniformly chosen employees (and the subsequent }\\
\mbox{promotion cascades), the remaining company structure is }T'
}
\right).
\]
These, together with the constant function 1, form an eigenbasis.
\end{enumerate}
\end{thm}
For the chain explicited at the start of Section \ref{sec:cktrees},
the basis of right eigenfunctions constructed in part iii above are
the columns of the following table (all empty entries are zeroes):
\[
\begin{array}{cccccc}
1 & \f{}_{\basistwo} & \f{}_{\basisthree} & \f{}_{\basisfour} & \f{}_{\basisfive} & \f{}_{\basissix}\\
\hline 1\\
1 & 1\\
1 &  & 1\\
1 &  & 3 & 1\\
1 & \frac{3}{2} & \frac{3}{2} &  & 1\\
1 & 2 & 4 & \frac{4}{3} & \frac{8}{3} & 1
\end{array}
\]
The corresponding eigenvalues for the single model are $1,\frac{1}{2},\frac{1}{2},\frac{1}{4},\frac{1}{4},0$,
and for the binomial model are $1,q^{2},q^{2},q^{3},q^{3},q^{4}$.

Here is a quick way to see parts i and ii of the theorem: the transition
matrices for both chains are triangular if the states are ordered
by the number of employees. Then the eigenvalues are the diagonal
entries, i.e. the probability of remaining at a state $T'$ (on $n'$
vertices), which is $\frac{n_{0}-n'}{n_{0}}$ for the single model
and $q_{2}^{n'}$ for the binomial model. The proof of iii is at the
end of this section.

Note that part ii does not follow from the general result on the spectrum
of descent operator chains (Theorem \ref{thm:evalues}): that gives
the eigenvalue multiplicities on the full state space basis of all
rooted forests on $n_{0}$ vertices, of which the states of the employee-firing
chains are a proper subset. Similarly, part iii is not an immediate
consequence of Theorem \ref{thm:tob2r-evectors-cocommutative}.iv.
Instead, the proof below will use the general theory to construct
the eigenfunctions, then show they stay linearly independent when
restricted to this smaller state space.

Here's an application before the proof. Recall from Example \ref{ex:rooted-trees}
that a department head is an employee whose direct superior is the
boss, and a department is everyone further down the chain of superiority
from a department head. (So a department is a connected component
of the tree with the root removed). Write $n^{(i)}$ for the number
of employees in department $i$. For example, the company in Figure
\ref{fig:cktrees-company} has two departments, with $n^{(1)}=2$
and $n^{(2)}=5.$ Because the boss is not in any department, it always
holds that $1+\sum n^{(i)}=n$. 

Suppose the company receives a project that requires a team of $s_{i}$
employees from department $i$. The number of ways to choose such
a team is $\prod_{i}\binom{n^{(i)}}{s_{i}}$.
\begin{cor}
\label{cor:expectation-cktrees} Let $s_{i}$ be any sequence of integers,
and use the $n^{(i)}$ notation for department sizes as above.
\begin{enumerate}
\item $\f(T):=n\prod_{i}\binom{n^{(i)}}{s_{i}}$ is a right eigenfunction
of eigenvalue $\frac{n_{0}-1-\sum_{i}s_{i}}{n_{0}}$ (resp. $q_{2}^{1+\sum s_{i}}$)
for the single (resp. binomial) model.
\item After $t$ months under the single (resp. binomial) model, starting
from a company of $n_{0}$ employees,
\[
\Expect\left(n_{t}\prod_{i}\binom{n_{t}^{(i)}}{s_{i}}\right)=\beta^{t}n_{0}\prod_{i}\binom{n_{0}^{(i)}}{s_{i}},
\]
where $\beta^{t}=\frac{n_{0}-1-\sum_{i}s_{i}}{n_{0}}$ for the single
model, and $\beta^{t}=q_{2}^{1+\sum s_{i}}$ for the binomial model.
\item In particular, the expected number of teams consisting of $s_{i}$
employees in department $i$ satisfies 
\[
\beta^{t}\prod_{i}\binom{n_{0}^{(i)}}{s_{i}}\leq\Expect\left(\prod_{i}\binom{n_{t}^{(i)}}{s_{i}}\right)\leq\beta^{t}\frac{n_{0}}{1+\sum s_{i}}\prod_{i}\binom{n_{0}^{(i)}}{s_{i}}
\]
for both the single and binomial models. Also, in the single model,
\[
\Expect\left(\prod_{i}\binom{n_{t}^{(i)}}{s_{i}}\right)\leq\beta^{t}\frac{n_{0}}{n_{0}-t}\prod_{i}\binom{n_{0}^{(i)}}{s_{i}}.
\]
 
\end{enumerate}
\end{cor}
\begin{proof}
We claim that, up to scaling, $\f$ is the sum of the $\f_{T'}$ of
Theorem \ref{thm:spectrum-cktrees}.iii, over all $T'$ with exactly
$s_{i}$ employees in department $i$. This sum is
\[
\left(\sum\f_{T'}\right)(T)=\binom{n}{1+\sum s_{i}}\Prob\left(\substack{\mbox{After firing }n-\sum s_{i}-1\mbox{ uniformly chosen employees (and the subsequent }\\
\mbox{promotion cascades), there are }s_{i}\mbox{ employees in department }i
}
\right).
\]
The key observation is that the promotions always occur within the
department of the fired employee, so it is unnecessary to consider
the promotion cascades. The probability of interest is simply that
of uniformly picking $\sum s_{i}$ employees (those who are not fired)
so that $s_{i}$ of them come from department $i$. Hence 
\[
\left(\sum\f_{T'}\right)(T)=\binom{n}{1+\sum s_{i}}\frac{\prod_{i}\binom{n^{(i)}}{s_{i}}}{\binom{n-1}{\sum s_{i}}}=\frac{n}{1+\sum s_{i}}\prod_{i}\binom{n^{(i)}}{s_{i}}.
\]
Part ii then follows by applying Proposition \ref{prop:eigenfunction-expectation}.
To see part iii, note that $n_{t}\leq n_{0}$, and, if $\f(T_{t})\neq0$,
it must be that $1+\sum s_{i}\leq n_{t}$ . The last statement uses
the alternative upper bound $n_{t}\geq n_{0}-t$ for the single model.
\end{proof}

\begin{proof}[Proof of Theorem \ref{thm:spectrum-cktrees}]
 The proof follows the argument of \cite[Proof of Th. 3.19]{hopfpowerchains}
to show that the functions $\f_{T'}$, as defined in part iii of the
theorem, are right eigenfunctions for the Markov chain, of the claimed
eigenvalues, and are linearly independent. Parts i and ii will then
follow immediately.

Recall that right eigenfunctions of a descent operator chain come
from eigenvectors in the dual Hopf algebra. By the ``duality of
primitives and generators'' \cite[Sec. 3.8.24]{primitivegeneratordual},
the elements $T^{*}$ that are dual to trees are primitive in this
dual algebra. Hence, for any tree $T'\neq\bullet$ on $n'$ vertices,
Theorem \ref{thm:tob2r-evectors-cocommutative}.iii asserts that $\bullet^{*n_{0}-n'}T'^{*}$
is a $\frac{n_{0}-n'}{n_{0}}$-eigenvector of $\ter_{n_{0}}$, and
a $q_{2}^{n'}$-eigenvector of $\binter_{n_{0}}(q_{2})$. The corresponding
right eigenfunction, by Proposition \ref{prop:doobtransform-efns}.R,
is 
\begin{align*}
\f_{T'}(\bullet^{n_{0}-n}\amalg T) & =\frac{1}{\eta(\bullet^{n_{0}-n}\amalg T)}\bullet^{*n_{0}-n'}T'^{*}\mbox{ evaluated on }\bullet^{n_{0}-n}\amalg T\\
 & =\frac{1}{\eta(\bullet^{n_{0}-n}\amalg T)}(\bullet^{*}\otimes\dots\otimes\bullet^{*}\otimes T'^{*})\Delta_{1,\dots,1,n'}(\bullet^{n_{0}-n}\amalg T).
\end{align*}
This is the probability that a uniformly chosen increasing labelling
of $\bullet^{n_{0}-n}\amalg T$ has the vertices labelled $1,2,\dots,n'$
forming a copy of $T'$. As before, view each increasing labelling
of $\bullet^{n_{0}-n}\amalg T$ as a (ordered) choice of $n_{0}-n$
labels for the singletons, together with an increasing labelling of
$T$. Since $T'\neq\bullet$, the desired condition is equivalent
to the singletons all having labels greater than $n'$, and the smallest
$n'$ labels in $T$ being assigned to the vertices of $T'$. These
two subconditions are independent; the first happens with probability
\[
\frac{(n_{0}-n')(n_{0}-n'-1)\dots(n-n'+1)}{n_{0}(n_{0}-1)\dots(n+1)}=\binom{n}{n'}/\binom{n_{0}}{n'},
\]
and the second is the probability that, after $n-n'$ firings, the
remaining company structure is $T'$. So the product of these two
numbers give an eigenfunction. Since any scalar multiple of an eigenfunction
is again an eigenfunction, we can multiply this by $\binom{n_{0}}{n'}$
(which is independent of $T$) and obtain $\f_{T'}$.

The linear independence of $\f_{T'}$ comes from a triangularity argument.
Clearly both factors in $\f_{T'}(T)$ are zero if $T$ has fewer than
$n'$ vertices (i.e. fewer vertices than $T'$). And if $T$ has exactly
$n'$ vertices, then $\f_{T'}(T)$ is the probability that $T=T'$,
so it evaluates to 1 at $T$ and 0 otherwise. 
\end{proof}

\subsection{An employee-firing chain with VPs \label{sub:trintober-cktrees}}

This section studies the $\trintober_{n}(q_{1},q_{2},q_{3})$ chains
on the Connes-Kreimer algebra (where $q_{1}+q_{2}+q_{3}=1$). This
chain involves both $\Delta_{1,n-1}$, the firing operator from above,
as well as a new operator $\Delta_{n-1,1}$, which removes from a
forest the root of one of its connected components. So now, in addition
to firing employees whose performance is in the interval $[0,q_{1}],$
there is a reward for employees performing in the interval $[q_{1}+q_{2},1]=[1-q_{3},1]$
- a special VP status is given to the person highest in the chain
of superiority above this good employee, who isn't yet a VP. The chain
keeps track of the forest structure of the non-VP employees of the
company; since VP status is for life, the chain ignores any positions
once it becomes VP status.

To state a result using the eigenfunctions of Theorem \ref{thm:tob2r-evectors-cocommutative}.ii,
more tree terminology is necessary. Recall that, if $v$ is a vertex
of the forest $x$, then $H(v)$ is the hook of $v$, consisting of
$v$, its children, its grandchildren, ... , and $h(v)$ is the number
of vertices in $H(v)$. For the company interpretation, $H(v)$ is
the set of employees under indirect supervision of $v$. A complementary
idea is the \emph{ancestors} $A(v)$ of $v$, consisting of $v$,
its parent, its grandparent, ... ; let $a(v)$ denote the number of
ancestors of $v$. This is the length of the superiority chain from
$v$ upwards, not including any VPs. $A(v)\backslash\{v\}$ is the
\emph{strict ancestors} of $v$.

Theorem \ref{thm:trintoberexpectation-cktrees} below gives a family
of functions $\fo_{j}$ whose average valuefalls by approximately
$q_{2}^{n_{0}-j}$ with each step of the chain. $\fo_{j}$ is the
expected number of teams, of size $n_{0}-j$, that some employee $u$
can assemble from his/her ``hook''. The project leader $u$ is chosen
with probability proportional to $\left(\frac{q_{3}}{q_{1}+q_{3}}\right)^{a(u)-1}\left(\frac{q_{1}}{q_{1}+q_{3}}\right)^{h(u)}$.
One may justify these probabilities as follows: we disfavour a project
leader with large $h(u)$ as he/she has many employees to manage,
that will take time away from this project. Large $a(u)$ is also
undesirable as he/she may have many other projects from these supervisors.
\begin{thm}[Approximate eigenfunction for the employee-firing chain with VPs]
 \label{thm:trintoberexpectation-cktrees} Let $\{X_{t}\}$ denote
the employee-firing chian with VPs as detailed above, starting with
a company of $n_{0}$ employees (and no VPs). For each integer $j\in[0,n-2]$,
define the following functions on the forest structure of the non-VP
employees: 
\[
\fo_{j}(x):=\sum_{u\in x}\binom{h(u)}{n_{0}-j}\left(\frac{q_{3}}{q_{1}+q_{3}}\right)^{a(u)-1}\left(\frac{q_{1}}{q_{1}+q_{3}}\right)^{h(u)}.
\]
(The binomial coefficient is 0 if $h(u)<n_{0}-j$.) Then 
\[
\Expect\left\{ \fo_{j}(X_{t})\right\} \leq q_{2}^{(n_{0}-j)t}\fo_{j}(X_{0})\max_{u\in X_{0}:h(u)\geq n_{0}-j}\left\{ \binom{n_{0}}{a(u)-1}\right\} .
\]

\end{thm}
The proof requires yet more definitions: a \emph{trunk} of a forest
consists of rooted subtrees of its constituent trees. For example,
in Figure \ref{fig:trintoberexpectationpf}, $S$ is a trunk of $x$,
and $S'$ and $S'\cup\{w\}$ are both trunks of $x'$. $T'$ is not
a trunk of $x$, but rather a trunk of $x\backslash S$. Since we
will consider the hook of $v$ both within the full forest $x$ and
of a trunk $S$ (or of other subtrees), write $H_{x}(v)$ and $H_{S}(v)$
for these respectively, and similarly $h_{x}(v)$ and $h_{S}(v)$.
For example, in Figure \ref{fig:trintoberexpectationpf}, $h_{x}(v_{2})=10$
whilst $h_{S}(v_{2})=3$.

For the proof of this theorem, work in the non-decorated Connes-Kreimer
algebra, where the vertices are unlabelled. The following coassociativity
result will be useful:
\begin{lem}
\label{lem:coassociativity-cktrees} Suppose $x$ is a forest of degree
$n$, and $T'$ is a tree of degree $n-j$. Fix an integer $i\in[0,j]$.
Then a ratio of coproduct coefficients may be expressed as follows:
\[
\frac{\eta_{x}^{\overbrace{\bullet,\dots,\bullet}^{j-i},T',\overbrace{\bullet,\dots,\bullet}^{i}}}{\eta(x)}=\frac{1}{(n-j)!\binom{n}{i\ n-j\ j-i}}\sum_{S,T'}\left(\prod_{v\in S}\frac{h_{x}(v)}{h_{S}(v)}\right)\prod_{v\in T'}h_{x}(v),
\]
where the sum runs over all trunks $S$ of $x$ with degree $i$,
and all copies of $T'$ within $x$ that are trunks of $x\backslash S$.
In particular, taking $T'=\emptyset$ (so $n=j$) shows that 
\[
\binom{n}{i}=\sum_{S}\prod_{v\in S}\frac{h_{x}(v)}{h_{S}(v)},
\]
where the sum runs over all trunks $S$ of $x$ with degree $i$.\end{lem}
\begin{proof}
Recall that the numerator on the left hand side is the coefficient
of $\bullet^{\otimes j-i}\otimes T'\otimes\bullet^{\otimes i}$ in
$\Delta_{1,\dots,1,n-j,1,\dots,1}(x)$. By coassociativity, and then
the definition of $\eta$: 
\begin{align*}
\Delta_{1,\dots,1,n-j,1,\dots,1}(x) & =\left(\Delta_{1,\dots,1}\otimes\id\otimes\Delta_{1,\dots,1}\right)\circ\Delta_{j-i,n-j,i}(x)\\
 & =\left(\sum_{S,T}\eta(x\backslash(S\cup T))\eta(S)\right)\bullet^{\otimes j-i}\otimes T\otimes\bullet^{\otimes i},
\end{align*}
summing over all trunks $S$ of $x$ with degree $i$, and all trunks
$T$ of $x\backslash S$ with degree $n-j$. So it suffices to show
that, for a specific trunk $S$ of $x$ with degree $i$ and a copy
of $T'$ within $x$ that is a trunk of $x\backslash S$, 
\begin{equation}
\eta(x\backslash(S\cup T'))\eta(S)\frac{1}{\eta(x)}=\frac{1}{(n-j)!\binom{n}{i\ n-j\ j-i}}\left(\prod_{v\in S}\frac{h_{x}(v)}{h_{S}(v)}\right)\prod_{v\in T'}h_{x}(v).\label{eq:trinomialtreepf1}
\end{equation}

The key is Equation \ref{eq:eta-cktree}: 
\[
\eta(x)=\frac{\deg x!}{\prod_{v\in x}h_{x}(v)},
\]
so the left hand side of Equation \ref{eq:trinomialtreepf1} is 
\[
\frac{(j-i)!}{\prod_{v\in x\backslash(S\cup T')}h_{x\backslash(S\cup T')}(v)}\frac{i!}{\prod_{v\in S}h_{S}(v)}\frac{\prod_{v\in x}h_{x}(v)}{n!}.
\]
Observe that each $v\in x$ is exactly one set out of $S$, $T'$
and $x\backslash(S\cup T')$. Further, if $v\in x\backslash(S\cup T')$,
then all descendants of $v$ within $x$ are within $x\backslash(S\cup T')$,
so $h_{x}(v)=h_{x\backslash(S\cup T)}(v)$. This proves Equation \ref{eq:trinomialtreepf1}.

To see the ``in particular'' claim in the lemma, note that $\eta_{x}^{\overbrace{\bullet,\dots,\bullet}^{n-i},\emptyset,\overbrace{\bullet,\dots,\bullet}^{i}}=\eta_{x}^{\bullet,\dots,\bullet}=\eta(x)$,
by coassociativity.
\end{proof}

\begin{proof}[Proof of Theorem \ref{thm:trintoberexpectation-cktrees}]
 The first step is to calculate a right eigenfunction by applying
Theorem \ref{thm:tob2r-evectors-cocommutative}.ii to the dual algebra.
Let the kernel element $p$ be $T'^{*}$ where $T'$ is a tree, of
degree $j$. Then the associated eigenfunction is 
\begin{align}
\f_{T'}(x) & =\frac{1}{\eta(x)}\sum_{i=0}^{j}\binom{j}{i}q_{1}^{i}q_{3}{}^{j-i}\bullet^{*i}T'^{*}\bullet^{*j-i}\mbox{ evaluated on }x\nonumber \\
 & =\frac{1}{\eta(x)}\sum_{i=0}^{j}\binom{j}{i}q_{1}^{j-i}q_{3}{}^{i}\eta_{x}^{\overbrace{\bullet,\dots,\bullet}^{j-i},T',\overbrace{\bullet,\dots,\bullet}^{i}}\quad\mbox{(renaming }j-i\mbox{ as }i\mbox{)}\nonumber \\
 & =\sum_{i=0}^{j}\binom{j}{i}q_{1}^{j-i}q_{3}{}^{i}\frac{1}{(n_{0}-j)!\binom{n_{0}}{i\ n-j\ j-i}}\sum_{S,T'}\left(\prod_{v\in S}\frac{h_{x}(v)}{h_{S}(v)}\right)\prod_{v\in T'}h_{x}(v),\label{eq:trinomialtreepf2}
\end{align}
summing over all trunks $S$ of $x$ with degree $i$, and all copies
of $T'$ within $x$ that are trunks of $x\backslash S$.

We exchange the order of summation in $S$ and $T'$ - rather than
first choosing the trunk $S$, and then letting $T'$ be a trunk of
$x\backslash S$, we instead start by specifying the copy of $T'$
in $T$, then let $u$ denote its root, and let $S$ be a trunk of
$x\backslash T'$. The condition that $T'$ is a trunk of $x\backslash S$
translates to $S$ containing all strict ancestors of $u$. Note also
that $S$ cannot contain any vertices in $H(u)$. Thus the degree
$i$ of $S$ must range between $a(u)-1$ and $n-h_{x}(u)$.

It would be ideal to simplify Equation \ref{eq:trinomialtreepf2}
using the second statement of Lemma \ref{lem:coassociativity-cktrees}.
This requires removing the condition $S\supseteq A(u)\backslash\{u\}$.
To do so, define $\bar{x}:=x\backslash(A(u)\cup H(u))$, and $\bar{S}=S\backslash(A(u)\backslash\{u\})$;
see Figure \ref{fig:trintoberexpectationpf}. Then the trunks $S$
of $x$ that contain all strict ancestors of $u$ are in bijection
with the trunks $\bar{S}$ of $\bar{x}.$ Moreover, $h_{\bar{S}}(v)=h_{S}(v)$
for all $v\in\bar{S}$, and $h_{\bar{x}}(v)=h_{x}(v)$ if $v\in\bar{x}$.
And for $v\in S\backslash\bar{S}=A(u)\backslash\{u\}$, it is true
that $h_{x}(v)>h_{S}(v)$. So 
\begin{equation}
\prod_{v\in S}\frac{h_{x}(v)}{h_{S}(v)}=\left(\prod_{v\in\bar{S}}\frac{h_{\bar{x}}(v)}{h_{\bar{S}}(v)}\right)\left(\prod_{v\in A(u)\backslash\{u\}}\frac{h_{x}(v)}{h_{S}(v)}\right)\geq\left(\prod_{v\in\bar{S}}\frac{h_{\bar{x}}(v)}{h_{\bar{S}}(v)}\right).\label{eq:trinomialtreepf3}
\end{equation}
Now sum over all trunks $S$ of $x$ with degree $i$ and containing
all strict ancestors of $u$:
\[
\sum_{S}\prod_{v\in S}\frac{h_{x}(v)}{h_{S}(v)}\geq\sum_{\bar{S}}\left(\prod_{v\in\bar{S}}\frac{h_{\bar{x}}(v)}{h_{\bar{S}}(v)}\right)=\binom{\deg\bar{x}}{\deg\bar{S}}=\binom{n_{0}-a(u)-h_{x}(u)+1}{i-a(u)+1}
\]
using the second statement of Lemma \ref{lem:coassociativity-cktrees}.
Substitute into Equation \ref{eq:trinomialtreepf2}, keeping in mind
that $i=\deg S$ ranges between $a(u)-1$ and $n_{0}-h_{x}(u)$ (and
viewing $u$ as a function of $T'$): 
\begin{align*}
\f_{T'}(x) & \geq\sum_{T'}\prod_{v\in T'}h_{x}(v)\sum_{i=a(u)-1}^{n_{0}-h_{x}(u)}\binom{j}{i}q_{1}^{j-i}q_{3}{}^{i}\frac{1}{(n_{0}-j)!\binom{n_{0}}{i\ n-j\ j-i}}\binom{n_{0}-a(u)-h_{x}(u)+1}{i-a(u)+1}\\
 & =\sum_{T'}\prod_{v\in T'}h_{x}(v)\frac{1}{(n_{0}-j)!\binom{n_{0}}{j}}q_{3}^{a(u)-1}q_{1}^{-n_{0}+h_{x}(u)+j}\sum_{i=a(u)-1}^{n_{0}-h_{x}(u)}q_{1}^{n_{0}-h_{x}(u)-i}q_{3}{}^{i-a(u)+1}\binom{n_{0}-a(u)-h_{x}(u)+1}{i-a(u)+1}\\
 & =\sum_{T'}\prod_{v\in T'}h_{x}(v)\frac{1}{(n_{0}-j)!\binom{n_{0}}{j}}q_{3}^{a(u)-1}q_{1}^{-n_{0}+h_{x}(u)+j}(q_{1}+q_{3})^{n_{0}-a(u)-h_{x}(u)+1}.
\end{align*}

\begin{figure}
\begin{center} 
\setlength{\unitlength}{2mm}
\begin{picture} (15,24) (-3,-15)   
\put(5,7){\color{red}{\circle*{1}}}
\put(2.5,7){$v_1$}  
\thicklines
\put(5,7){\color{red}{\line(0, -1){4}} }
\put(5,3){\color{red}{\circle*{1}} }
\put(2.5,3.2){$v_2$} 
\thinlines
\put(5,3){\line(-6, -4){6}} 
\put(-1,-1){\circle*{1}} 
\put(-3,-0.5){$w$} 
\put(-1,-1){\line(0, -1){4}} 
\put(-1,-5){\circle*{1}}
\thicklines
\put(5,3){\color{red}{\line(0, -1){4}} }
\put(5,-1){\color{red}{\circle*{1}} }
\put(2.5,-1){$v_3$}  
\put(5,3){\color{red}{\line(3, -4){3}} }
\put(8,-1){\color{red}{\circle*{1}} }
\thinlines
\put(5,-1){\line(0, -1){4}} 
\put(5,-5){\color{blue}{\circle*{1}} }
\put(2.5,-5){$u$} 
\thicklines
\put(5,-5){\color{blue}{\line(-3, -4){3}} }
\put(2,-9){\color{blue}{\circle*{1}} }
\put(5,-5){\color{blue}{\line(0, -1){4}} }
\put(5,-9){\color{blue}{\circle*{1}} }
\thinlines
\put(5,-9){\line(0, -1){4}} 
\put(5,-13){\circle*{1}} 
\put(5,-5){\line(3, -4){3}} 
\put(8,-9){\circle*{1}} 
\put(6.5,1.5){\color{red}{{$S$}}}
\put(1,-7){\color{blue}{{$T'$}}}
\put(-3,-15){$ \underbrace{\hspace{30mm}}_x $}
\end{picture} 
\qquad \qquad \qquad
\begin{picture} (15,24) (-3,-15)   
\put(-1,-1){\circle*{1}} 
\put(-3,-0.5){$w$}
\put(-1,-1){\line(0, -1){4}} 
\put(-1,-5){\circle*{1}}
\put(8,-1){\color{red}{\circle*{1}} }
\put(6.5,0){\color{red}{{$S'$}}}
\put(-3,-15){$ \underbrace{\hspace{30mm}}_{x'} $}
\end{picture} 
\end{center}\protect\caption{\label{fig:trintoberexpectationpf} To illustrate the notation in
the proof of Theorem \ref{thm:trintoberexpectation-cktrees}: a tree
$x$ with one possibility of $S$ and $T'$, the corresponding $x'$
and $S'$, and the vertices $u$ and $v_{1},\dots,v_{a(u)-1}$.}
\end{figure}
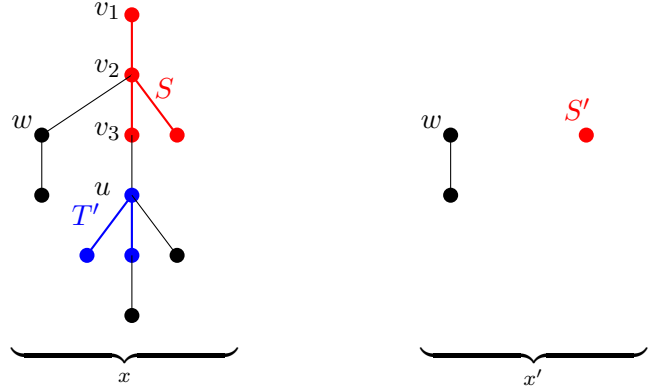

$\f_{T'}$ has eigenvalue $q_{2}^{n_{0}-j}$ whenever $\deg T'=n_{0}-j$.
Take the following linear combination of such $T'$, over all $T'$
with degree $n_{0}-j$ (as opposed to only over copies within $x$
of a fixed $T'$):
\begin{align*}
\sum_{T'}(n_{0}-j)!\binom{n_{0}}{j}\frac{q_{1}^{n_{0}-j}}{q_{1}+q_{3}{}^{n}}\frac{1}{\prod_{v\in T'}h_{T'}(v)}\f_{T'}(x) & \geq\sum_{T'}\prod_{v\in T'}\frac{h_{x}(v)}{h_{T'}(v)}q_{3}^{a(u)-1}q_{1}^{h_{x}(u)}(q_{1}+q_{3})^{n_{0}-a(u)-h_{x}(u)+1}\\
 & =\sum_{u\in x}\binom{h_{x}(u)}{n_{0}-j}\left(\frac{q_{3}}{q_{1}+q_{3}}\right)^{a(u)-1}\left(\frac{q_{1}}{q_{1}+q_{3}}\right)^{h_{x}(u)}=\fo_{j}(x)
\end{align*}
where obtaining the second line uses Equation \ref{eq:trinomialtreepf2},
for a sum over trunks $T'$ of $H_{x}(u)$ with degree $n_{0}-j$.

So, by Proposition \ref{prop:eigenfunction-expectation}, 
\begin{align*}
\Expect(\fo_{j}(X_{t})|X_{0}=x_{0}) & \leq\Expect(\fo_{j}(X_{t})|X_{0}=x_{0})\\
 & =q_{2}^{(n_{0}-j)t}\sum_{T'}(n_{0}-j)!\binom{n_{0}}{j}\frac{q_{1}^{n_{0}-j}}{q_{1}+q_{3}{}^{n}}\frac{1}{\prod_{v\in T'}h_{T'}(v)}\f_{T'}(x_{0})
\end{align*}
and it remains to upper-bound the right-hand side. We do so by finding
an upper bound $C$ for $\prod_{v\in A(u)\backslash\{u\}}\frac{h_{x}(v)}{h_{S}(v)}$;
then we have 
\[
\prod_{v\in S}\frac{h_{x}(v)}{h_{S}(v)}=\left(\prod_{v\in\bar{S}}\frac{h_{\bar{x}}(v)}{h_{\bar{S}}(v)}\right)\left(\prod_{v\in A(u)\backslash\{u\}}\frac{h_{x}(v)}{h_{S}(v)}\right)\geq C\left(\prod_{v\in\bar{S}}\frac{h_{\bar{x}}(v)}{h_{\bar{S}}(v)}\right)
\]
and using this in place of Equation \ref{eq:trinomialtreepf3} in
the first four paragraphs of this proof will show 
\[
q_{2}^{(n_{0}-j)t}\sum_{T'}(n_{0}-j)!\binom{n_{0}}{j}\frac{q_{1}^{n_{0}-j}}{q_{1}+q_{3}{}^{n}}\frac{1}{\prod_{v\in T'}h_{T'}(v)}\f_{T'}(x_{0})\leq q_{2}^{(n_{0}-j)t}C\fo_{j}(X_{0}).
\]

To obtain the upper bound $C$, let $v_{1},v_{2},\dots v_{a(u)-1}$
be the vertices of $A(u)\backslash\{u\}$, such that $v_{1}$ is a
root, $v_{2}$ is a child of $v_{1}$, $v_{3}$ is a child of $v_{2}$,
... . So $\Anc(v_{k})=\{v_{1},\dots,v_{k}\}$, and, except $v_{k}$
itself, these are not in the hook of $v_{i}$ - hence $h_{x}(v_{k})\leq n_{0}-(k-1)$.
For the denominator, note that $H_{S}(v_{k})\supseteq v_{k},v_{k+1},\dots v_{a(u)-1}$,
so $h_{S}(v_{i})\geq a(u)-k$. Thus 
\[
\prod_{v\in A(u)\backslash\{u\}}\frac{h_{x}(v)}{h_{S}(v)}\leq\frac{n_{0}(n_{0}-1)\dots(n_{0}-a(u)+2)}{(a(u)-1)(a(u)-2)\dots1}=\binom{n_{0}}{a(u)-1}\leq\max_{u\in X_{0}:h(u)\geq n_{0}-j}\left\{ \binom{n_{0}}{a(u)-1}\right\} ,
\]
which proves the theorem.
\end{proof}

\section{Relative time on a to-do list\label{sec:fqsym}}

This section applies the Hopf-algebraic framework of Sections \ref{sec:Markov-Chains-from-descent-operators}-\ref{sec:partsof1}
to the Malvenuto-Reutenauer Hopf algebra to analyse variations of
the ``to-do list'' chain of Example \ref{ex:fqsym}, also known
as \emph{top-to-random-with-standardisation}. As shown in \cite{hopfchainlift},
the significance of this family of chains is two-fold: first, the
distribution after $t$ steps is equal to that of a card shuffle if
started from the identity permutation. So any result below phrased
solely in terms of the distribution at time $t$ (and starting at
the identity) also applies to the corresponding shuffle; see Theorem
\ref{thm:recursivelumping-shuffle} and Corollary \ref{cor:probestimate-fqsym}.
Second, information about these chains can be used to analyse descent
operator Markov chains on the numerous subquotients of the Malvenuto-Reutenauer
Hopf algebra, which are all lumpings of the present chains. These
lumped chains include the phylogenetic tree example from the introduction
(Loday-Ronco algebra of binary trees), an unbump-and-reinsert chain
on tableaux (Poirier-Reutenauer algebra) and a remove-and-readd-a-box
chain on partitions (symmetric functions with Schur basis) \cite[Sec. 3.2-3.3]{hopfchainlift}.

\subsection{Markov chains from the Malvenuto-Reutenauer Hopf algebra}

Following \cite{sym6}, let $\fqsym$ (free quasisymmetric functions)
denote the Malvenuto-Reutenauer Hopf algebra of permutations \cite[Sec. 3]{fqsym}
(the other common notation is $\mathfrak{S}Sym$ \cite{fqsymstructure}).
Its basis in degree $n$ consists of the permutations in $\sn$, written
in one-line notation: $\sigma=(\sigma_{1}\dots\sigma_{n})$. The $\sigma_{i}$
are called letters, in line with the terminology of words.

We work in the fundamental basis of $\fqsym$. The product $\sigma\tau$
of two permutations is computed as follows: first, add to each letter
in $\tau$ the degree of $\sigma$, then sum over all interleavings
of the result with $\sigma$. For example: 
\begin{align*}
(312)(21) & =(312)\shuffle(54)\\
 & =(31254)+(31524)+(31542)+(35124)+(35142)\\
 & \hphantom{=}+(35412)+(53124)+(53142)+(53412)+(54312).
\end{align*}
The coproduct is ``deconcatenate and standardise'': 
\[
\Delta(\sigma_{1}\dots\sigma_{n})=\sum_{i=0}^{n}\std(\sigma_{1}\dots\sigma_{i})\otimes\std(\sigma_{i+1}\dots\sigma_{n}),
\]
where the standardisation map $\std$ converts an arbitrary string
of distinct letters into a permutation by preserving the relative
order of the letters. For example:
\begin{align*}
 & \hphantom{=}\Delta(4132)\\
 & =()\otimes(4132)+\std(4)\otimes\std(132)+\std(41)\otimes\std(32)+\std(413)\otimes\std(2)+(4132)\otimes()\\
 & =()\otimes(4132)+(1)\otimes(132)+(21)\otimes(21)+(312)\otimes(1)+(4132)\otimes().
\end{align*}
Since, for every permutation $\sigma\in\sn$, its coproduct $\Delta_{1,n-1}(\sigma)$
is of the form $(1)\otimes\tau$ for some permutation $\tau\in\mathfrak{S}_{n}$,
we see inductively that the rescaling function is $\eta(\sigma)\equiv1$,
i.e. no rescaling is required. Thus, by Theorem \ref{thm:cppchain},
each timestep of the top-$r$-to-random-with-standardisation chains,
driven by $\trer_{n}=\frac{1}{n(n-1)\dots(n-r+1)}m\Delta_{1^{r},n-r}$,
has the following three-part description:
\begin{enumerate}
\item Remove the first $r$ letters of $\sigma$.
\item Replace the letters of the result by $r+1,\dots,n$ such that they
have the same relative order as before.
\item Insert the letters $1,2,\dots,r$ in a uniformly chosen position.
\end{enumerate}
Figure \ref{fig:todoexample-twotask} shows a possible trajectory
for $r=2$, $n=5$.

\begin{figure}
\begin{center}
\begin{tikzpicture} 
\node (A) at (0,0) {$(23541)$}; 
\node (B) at (4,0)  {$(15423)$}; 
\node (C) at (8,0)  {$(52341)$};  
\node (D) at (12,0) {$(42153)$}; 
\node (E) at (1,-2) {$(541)$}; 
\node (EE) at (3,-2) {$(543)$}; 
\node (F) at (5,-2)  {$(423)$}; 
\node (FF) at (7,-2)  {$(534)$}; 
\node (G) at (9,-2)  {$(341)$}; 
\node (GG) at (11,-2)  {$(453)$}; 
\draw[->] (A) -- (E); 
\draw[->] (E) -- (EE);
\draw[->] (EE) -- (B); 
\draw[->] (B) -- (F); 
\draw[->] (F) -- (FF);
\draw[->] (FF) -- (C); 
\draw[->] (C) -- (G); 
\draw[->] (G) -- (GG); 
\draw[->] (GG) -- (D); 
\end{tikzpicture}
\par\end{center}

\protect\caption{\label{fig:todoexample-twotask} A possible three-day trajectory of
the Markov chain of ``relative time on a to-do list'', for $n=5$
and $r=2$.}
\end{figure}
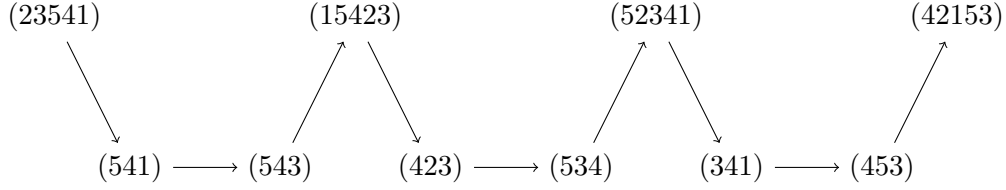

We give three interpretations of this chain. First, there is the to-do
list formulation: every day you complete $r$ tasks at the top of
a to-do list comprising $n$ tasks, then add $r$ new tasks independently,
each at a uniformly chosen position. You encode each daily to-do list
as a permutation, writing 1 for the latest addition to the list, 2
indicates the next newest addition excluding task 1, and so on, so
that $n$ is the task spending the longest time on the list. See Example
\ref{ex:fqsym}. In other words, you record only the relative times
that the tasks have spent on the list. In addition to fixed values
of $r$, we can let $r$ be a binomial variable with parameter $q_{2}$.
This corresponds to the descent operator $\binter$, and models $n$
managers independently handing you a task each with probability $q_{2}$.
The variable work speed (that each day the number of completed tasks
equals this vairable number of incoming tasks) can be explained by
procrastination when the workload is low, and a panic to overwork
when workload is high, to avoid a growing to-do list.

The second formulation is in terms of card-shuffling - take $r$ cards
from the top of an $n$-card deck, and reinsert them independently
each at a uniformly chosen position. Such shuffles were studied in
\cite{cppriffleshuffle}, but here we make one crucial modification:
instead of recording the values of the cards in each position as they
are shuffled, we record the relative last times that each card was
touched, 1 being the most recently moved card. So the number assigned
to each card changes at each shuffle - more specifically, if the top,
removed, card is labelled $i$, then we change its label to 1 during
its reinsertion, and the cards previously labelled $1,2,\dots,i-1$
now increase their labels each by 1, becoming $2,3,\dots,i$ respectively. 

The third formulation models a player's hand during a card game. At
each turn, the player plays the $r$ leftmost cards in his hand (thus
removing them from the hand), then draws $r$ new cards to maintain
a hand of $n$ cards. The newly-drawn cards are inserted into the
hand in uniformly chosen positions, depending on when the player plans
to play them (cards to be played sooner are placed towards the left).
Again, the chain tracks the relative times that the cards have spent
in the player's hand.

This chain has a unique stationary distribution, namely the uniform
distribution $\pi(\sigma)\equiv\frac{1}{n!}$. Indeed, since $\bullet=(1)$
is the unique permutation of 1, the comment after Theorem \ref{thm:stationarydistribution}
applies: the stationary distribution is given by $\pi(\sigma)=\frac{1}{n!}\eta(\sigma)\times$coefficient
of $\sigma$ in $\bullet^{n}$. And the required coefficient is 1
as there is a unique way of inserting the letters $1,2,\dots,n$ in
that order to obtain a given permutation.

\subsection{Relationship to card-shuffles}

For the readers' convenience, we reproduce below a theorem and proof
from \cite{hopfchainlift}, which allows results for the to-do list
chain to apply to cut-and interleave card-shuffles (as described after
Lemma \ref{lem:eta}).
\begin{thm}[Equidistribution under descent-operator chains on $\fqsym$ and $\calsh$]
 \cite[Th. 3.14]{hopfchainlift} \label{thm:multistepprobofshuffles}
The distribution on permutations after $t$ iterates of the $m\Delta_{P}$-chain
on $\fqsym$ (to-do list chain) is the same as that after $t$ iterates
of the $m\Delta_{P}$-chain on the shuffle algebra $\calsh$ (card-shuffling),
if both are started from the identity permutation.\end{thm}
\begin{proof}
First consider the case $t=1$. In $\fqsym$, for any weak-composition
$D$, 
\begin{align*}
m\Delta_{D}(1\dots n) & =m\left(\std(12\dots d_{1})\otimes\std((d_{1}+1)\dots(d_{1}+d_{2}))\otimes\dots\otimes\std((d_{1}+\dots+d_{l(D)-1}+1)\dots n)\right)\\
 & =m\left(12\dots d_{1}\otimes12\dots d_{2}\otimes\dots\otimes12\dots d_{l(D)}\right)\\
 & =1\dots d_{1}\shuffle(d_{1}+1)\dots(d_{1}+d_{2})\shuffle\dots\shuffle(d_{1}+\dots+d_{l(D)-1}+1)\dots n,
\end{align*}
and $m\Delta_{D}$ calculated in $\calsh$ gives the same result (under
the identification of $\sigma$ with $\llbracket\sigma\rrbracket$).
By linearity, $m\Delta_{P}(1\dots n)=m\Delta_{P}(\llbracket1\dots n\rrbracket)$
for all distributions $P$. This proves the equidistribution in the
case $t=1$.

The key to showing equidistribution for larger $t$ is to express
$t$ iterates of $m\Delta_{P}$, in either $\fqsym$ or $\calsh$,
as a single application of $m\Delta_{P''}$ for the same distribution
$P''$. This uses the identification of the descent operator $m\Delta_{D}$
with the homogeneous noncommutative symmetric function $S^{D}$, see
Section \ref{sub:Descent-operators-background}. Let $S^{P}=\sum_{D}\frac{P(D)}{\binom{n}{D}}S^{D}$
be the noncommutative symmetric function associated to $m\Delta_{P}$.

On a commutative Hopf algebra, such as $\calsh$, the composition
$\left(m\Delta_{P}\right)\circ\left(m\Delta_{P'}\right)$ corresponds
to the internal product of noncommutative symmetric functions $S^{P'}\cdot S^{P}$
(see Proposition \ref{prop:descentoperator-composition}). So $t$
iterates of $m\Delta_{P}$ on $\calsh$ correspond to $S^{P}\cdot S^{P}\cdot\dots\cdot S^{P}$,
with $t$ factors.

Now consider $m\Delta_{P}$ on $\fqsym$. Note that the coproduct
of the identity permutation is 
\begin{equation}
\Delta(1\cdots n)=\sum_{r=0}^{n}(1\cdots r)\otimes(1\dots n-r),\label{eq:fqsymidcoprod}
\end{equation}
and each tensor-factor is an identity permutation of shorter length.
Thus the subalgebra of $\fqsym$ generated by identity permutations
of varying length is closed under coproduct, and repeated applications
of $m\Delta_{P}$, starting at the identity, stays within this sub-Hopf-algebra.
Equation \ref{eq:fqsymidcoprod} shows that this sub-Hopf-algebra
is cocommutative, so by Proposition \ref{prop:descentoperator-composition},
the composition $\left(m\Delta_{P}\right)\circ\left(m\Delta_{P'}\right)$
corresponds to the internal product $S^{P}\cdot S^{P'}$, in the opposite
order from for $\calsh$. However, we only concern the case $P=P'$
where the order is immaterial: $t$ iterates of the $m\Delta_{P}$-chain
on $\fqsym$ are also driven by $S^{P}\cdot S^{P}\cdot\dots\cdot S^{P}$,
with $t$ factors.
\end{proof}

\subsection{Recursive lumping}

Since $\Delta_{1,n-1}$ sends a permutation to $\bullet\otimes$another
permutation, the recursive lumping of Theorem \ref{thm:recursivelumping}
applies. Note that the lumping map $D^{n-k}:\calbn\rightarrow\calb_{k}$
is simply the observation of the bottom $k$ items on the to-do list,
or the relative last-moved times of the bottom $k$ cards of the deck.
Hence we have
\begin{thm}[Recursive lumping for to-do list chains]
 \label{thm:recursivelumping-fqsym}Observing the last $k$ items
under the top-to-random-with-standardisation chain on $n$ items gives
a $\frac{n-k}{n}$-lazy version of the top-to-random-with-standardisation
chain on $k$ items. Observing the last $k$ items under the binomial-top-to-random-with-standardisation
chain on $n$ items gives the binomial-top-to-random-with-standardisation
chain on $k$ items, with the same parameter $q_{2}$.\qed
\end{thm}
It is possible to see these two lumpings via an elementary argument.
For each step of the top-to-random-with-standardisation chain on $n$
cards, one of these two scenarios occur:
\begin{itemize}
\item With probability $\frac{n-k}{n}$, the removed top card is reinserted
in one of positions $1,2,\dots,n-k$ (counting from the top). When
observing only the bottommost $k$ cards, we see no change.
\item With the complementary probability of $\frac{k}{n}$, the removed
top card is reinserted in one of positions $n-k+1,n-k+2,\dots n$,
chosen uniformly. The insertion pushes the $n-k+1$th card into the
$n-k$th position, so when observing only the bottommost $k$ cards,
the top card amongst these $k$ appears removed. Although the inserted
card is not the ``top'' card of this apparent removal, it is nevertheless
the most recently touched card, whether we are observing the entire
deck or just the bottommost $k$ cards. And the chain tracks only
the relative times that cards are last touched, so this difference
in card is invisible.
\end{itemize}
As for binomial-top-to-random: view each step of the chain in terms
of reinserted positions as follows:
\begin{enumerate}[label=\arabic*.]
\item Select each position out of $\{1,2,\dots,n\}$ independently with
probability $(1-q_{2})$. Let $r$ denote the number of selected positions.
\item Select uniformly a permutation $(\tau_{1}\dots\tau_{r})\in\mathfrak{S}_{r}$.
\item Remove the top $r$ cards and reinsert into the positions chosen in
1., in the relative order chosen in 2. (i.e. the card previously in
position $\tau_{1}$ is reinserted to the topmost selected position,
the card previously in position $\tau_{2}$ is reinserted to the second-topmost
selected position, ...). 
\end{enumerate}
When observing only the bottom $k$ positions, we see that they are
each independently selected as in 1.. The relative order of reinserted
cards is that of $\tau_{r'+1}\tau_{r'+2}\dots\tau_{r}$ for some $r'$,
and this is uniform on $\mathfrak{S}_{r-r'}$ when $\tau$ is uniform.
So the change in the bottom $k$ positions follows exactly steps 1.,
2., 3. above, with $k$ in place of $n$. (As discussed above for
the top-to-random shuffle, the cards that are pushed above and outside
the observation area are not the cards that are reinserted into the
observation area, but since they are the last-touched cards, the chain
does not notice this difference.)

In view of Theorem \ref{thm:multistepprobofshuffles}, this recursive-lumping
theorem can be restated in terms of top-to-random shuffles (although
the relationship is no longer a lumping):
\begin{thm}[Distribution of bottommost cards under top-to-random shuffles]
 \label{thm:recursivelumping-shuffle} After $t$ top-to-random shuffles
(resp. binomial-top-to-random shuffles) of a deck of $n$ cards, the
probability distribution on $\mathfrak{S}_{k}$ given by standardising
the values of the bottommost $k$ cards is equal to the distribution
on a deck of $k$ cards after $t$ steps of an $\frac{n-k}{n}$-lazy
version of the top-to-random shuffle (resp. after $t$ steps of a
binomial-top-to-random shuffle, with the same parameter $q_{2}$).\qed
\end{thm}

\subsection{Eigenvalues and eigenvectors}

An easy application of Theorem \ref{thm:tob2r-evectors-cocommutative}.i
shows that
\begin{prop}
The eigenvalues of the top-to-random-with-standardisation and binomial-top-to-random-with-standardisation
chains are $\beta_{j}=\frac{j}{n}$ and $\beta_{j}=q_{2}^{n-j}$ respectively
($j\in[0,n-2]\cup\{n\}$), and the multiplicity of $\beta_{j}$ is
$\dim\calh_{n-j}-\dim\calh_{n-j-1}=(n-j)!-(n-j-1)!$.
\end{prop}
Observe that the multiplicity $(n-j)!-(n-j-1)!$ is precisely the
number of permutations in $\sn$ fixing pointwise $1,2,\dots,j$ but
not $j+1$.  Indeed, we can associate an eigenvector of eigenvalue
$\beta_{j}$ to each such permutation $\tau$:
\begin{thm}[Eigenbasis of top-to-random-with-standardisation chains]
 \label{thm:spectrum-fqsym}Given a permutation $\tau\in\sn$, let
$j+1$ be the smallest number in $\{1,2,\dots n\}$ not fixed pointwise
by $\tau$. So the number $i$ defined by $\tau_{i}=j+1$ satisfies
$i>j+1$. Define the function $\f_{\tau}:\sn\rightarrow\mathbb{R}$
by 
\[
\f_{\tau}(\sigma)=\begin{cases}
1 & \mbox{if }(\sigma_{j+1}\sigma_{j+2}\dots\sigma_{n})\mbox{ is in the same relative order as }(\tau_{j+1}\tau_{j+2}\dots\tau_{n});\\
-1 & \mbox{if }(\sigma_{j+1}\sigma_{j+2}\dots\sigma_{n})\mbox{ is in the same relative order as ((}j+1)\tau_{j+1}\tau_{j+2}\dots\tau_{i-1}\tau_{i+1}\tau_{i+2}\dots\tau_{n});\\
0 & \mbox{otherwise.}
\end{cases}
\]
Then $\f_{\tau}$ is a right eigenfunction of top-to-random-with-standardisation
and binomial-top-to-random-with-standardisation with eigenvalue $\beta_{j}$,
and $\{\f_{\tau}|\tau\in\sn\}$ is a basis.
\end{thm}
The proof is at the end of this section.
\begin{example}
Let $\tau=12534$, so $j=2$ and $i=4$ (because $\tau_{4}=3$). Then
$\f_{\tau}(\sigma)$ is 1 if the last three letters of $\sigma$ are
in the same relative order as $534$, i.e. ``high then low then middle''.
And $\f_{\tau}(\sigma)$ is $-1$ if the last three letters of $\sigma$
are in the same relative order as $354$, i.e. ``low then high then
middle''. For example $\f_{\tau}(35{\color{red}412})=1$, $\f_{\tau}(24{\color{red}153})=1$,
$\f_{\tau}(25{\color{red}431})=0$.
\end{example}
The eigenbasis of Theorem \ref{thm:spectrum-fqsym} yields the probability
distribution of the position of the newest task amongst the bottom
$n-j$ tasks, or the last card that was touched in those bottom $n-j$
cards.
\begin{cor}
\label{cor:probestimate-fqsym} After $t$ iterates of top-to-random
(resp. binomial-top-to-random), with or without standardisation, starting
from the identity permutation, the probability that the smallest value
among $\{\sigma_{j+1},\sigma_{j+2},\dots,\sigma_{n}\}$ is $\sigma_{j+k}$
is 
\[
\begin{cases}
\frac{1}{n-j}\left(1+\beta_{j}^{t}(n-j-1)\right) & \mbox{ if }k=1;\\
\frac{1}{n-j}\left(1-\beta_{j}^{t}\right) & \mbox{ if }k\neq1;
\end{cases}
\]
where $\beta_{j}=\frac{j}{n}$ (resp. $\beta_{j}=q_{2}^{n-j}$).\end{cor}
\begin{proof}
Fix $j$. Let $\bar{\sigma}=\std(\sigma_{j+1}\dots\sigma_{n})\in\mathfrak{S}_{n-j}$,
so the desired probabilities are of $\bar{\sigma}_{k}=1$.

Consider first the case $k=1$. Let $\f_{1}=\sum\f_{\tau}$, summing
over all $\tau$ that fix $1,2,\dots,j$ pointwise, but not $j+1$.
We show 
\[
\f_{1}(\sigma)=\begin{cases}
-(n-j-1) & \mbox{if }\bar{\sigma}_{1}=1;\\
1 & \mbox{otherwise.}
\end{cases}
\]
If $\bar{\sigma}_{1}\neq1$ (i.e. $\sigma_{j+1}$ is not the smallest
amongst $\sigma_{j+2},\dots,\sigma_{n}$), then the only summand contributing
to $\f(\sigma)$ is $\tau=12\dots j(\bar{\sigma}_{1}+j)(\bar{\sigma}_{2}+j)\dots(\bar{\sigma}_{n-j}+j)$,
and $\f_{\tau}(\sigma)=1$. If $\bar{\sigma}_{1}=1$, then the only
non-zero summands in $\f_{1}(\sigma)$ come from $\tau$ of the form
$12\dots j(\bar{\sigma}_{2}+j)\dots(\bar{\sigma}_{\bar{i}}+j)(1+j)(\bar{\sigma}_{\bar{i}+1}+j)\dots(\bar{\sigma}_{n-j}+j)$,
where $\bar{i}\in(2,n-j)$. Hence there are $n-j-1$ terms all contributing
$-1$.

Writing $\mathbbm{1}$ for an indicator function, we have 
\begin{align}
\Prob(\bar{\sigma}_{1}=1) & =\Expect(\mathbbm{1}\{\bar{\sigma}_{1}=1\})\nonumber \\
 & =\Expect\left(\frac{1}{n-j}(1-\f_{1})\right)\nonumber \\
 & =\frac{1}{n-j}\left(1-\beta_{j}^{t}\f_{1}(x_{0})\right)\nonumber \\
 & =\frac{1}{n-j}\left(1+\beta_{j}^{t}(n-j-1)\right).\label{eq:prob-fqsym}
\end{align}
(The third equality uses the linearity of expectations and that $\f_{1}$
is an eigenvector with eigenvalue $\beta_{j}$.)

For $k>1$, consider $\f_{k}=\sum\f_{\tau}$ summing over all $\tau$
that fix $1,2,\dots,j$ pointwise, and satisfy $\tau_{k+j}=j+1$.
Then 
\[
\f_{k}(\sigma)=\begin{cases}
1 & \mbox{ if }\bar{\sigma}_{k}=1;\\
-1 & \mbox{ if }\bar{\sigma}_{1}=1;\\
0 & \mbox{ otherwise},
\end{cases}
\]
with the contributions in the first two cases coming from $\tau=12\dots j(\bar{\sigma}_{1}+j)(\bar{\sigma}_{2}+j)\dots(\bar{\sigma}_{n-j}+j)$
and $\tau=12\dots j(\bar{\sigma}_{2}+j)\dots(\bar{\sigma}_{k}+j)(1+j)(\bar{\sigma}_{k+1}+j)\dots(\bar{\sigma}_{n-j}+j)$
respectively. So 
\[
\Prob(\bar{\sigma}_{k}=1)-\Prob(\bar{\sigma}_{1}=1)=\Expect(\f_{k})=\beta_{j}^{t}\f_{k}(x_{0})=-\beta_{j}^{t},
\]
and substitute for $\Prob(\bar{\sigma}_{1}=1)$ from (\ref{eq:prob-fqsym}).
\end{proof}
\begin{rem*}

Corollary \ref{cor:probestimate-fqsym} also follows from the following
elementary argument. By the recursive lumpings of Theorem \ref{thm:recursivelumping-fqsym},
it suffices to calculate the distribution of the card/task labelled
1 under a $\frac{j}{n}$-lazy version of the chain driven by $\ter_{n-j}$,
or under the chain driven by $\binter_{n-j}(q_{2})$. Note that, because
a reinserted card is relabelled 1 and placed in a uniformly chosen
position, the card labelled 1 would be uniformly distributed in position
at time $t$ as long as at least one reinsertion has happened by time
$t$. Let $P(t)$ be the probability that no reinsertion has happened
by time $t$; then card 1 is at the top at time $t$ with probability
$P(t)+(1-P(t))\frac{1}{n-j}$, and in any other position with probability
$(1-P(t))\frac{1}{n-j}$. To complete the proof, 
\[
P(t)=P(1)^{t}=\begin{cases}
\left(\frac{j}{n}\right)^{t} & \mbox{ for the }\frac{j}{n}\mbox{-lazy version of the chain driven by }\ter_{n-j};\\
\left(q_{2}^{n-j}\right)^{t} & \mbox{ for the chain driven by }\binter_{n-j}(q_{2}).
\end{cases}
\]

\end{rem*}
\begin{proof}[Proof of Theorem \ref{thm:spectrum-fqsym}]
 To calculate right eigenfunctions, we must work in the dual $\fqsym^{*}$,
where $\Delta_{1,n-1}$ is the removal of the letter 1 (followed by
standardisation) and $m_{1,n-1}(\bullet^{*}\otimes\sigma^{*})$ is
the sum over all permutations whose last $n-1$ letters standardise
to $\sigma^{*}$.

Given $\tau,j$ as in the theorem statement, let $\bar{\tau}=\std(\tau_{j+1}\dots\tau_{n})\in\mathfrak{S}_{n-j}$.
Then the $\bar{i}$ with $\bar{\tau}{}_{\bar{i}}=1$ satisfies $\bar{i}>1$.
Note that, in $\fqsym^{*}$, we have $\Delta_{1,n-j-1}(\bar{\tau}^{*})=(\bar{\tau}_{1}\dots\bar{\tau}_{\bar{i}-1}\bar{\tau}_{\bar{i}+1}\dots\bar{\tau}{}_{n-j})^{*}=\Delta_{1,n-j-1}(1\bar{\tau}_{1}\dots\bar{\tau}_{\bar{i}-1}\bar{\tau}_{\bar{i}+1}\dots\bar{\tau}{}_{n-j})^{*}$.
So $p:=\bar{\tau}^{*}-(1\bar{\tau}_{1}\dots\bar{\tau}_{\bar{i}-1}\bar{\tau}_{\bar{i}+1}\dots\bar{\tau}{}_{n-j})^{*}\in\ker\Delta_{1,n-j-1}^{*}$
(and $p$ is nonzero because $\bar{i}>1$), and thus, by Proposition
\ref{prop:doobtransform-efns}.R and Theorem \ref{thm:tob2r-evectors-cocommutative},
a right eigenfunction is given by 
\begin{align*}
\f_{\tau}(\sigma) & =m_{1,\dots1,n-j}(\bullet^{*}\otimes\dots\otimes\bullet^{*}\otimes p)\mbox{ evaluated on }\mbox{\ensuremath{\sigma}}\\
 & =(\bullet^{*}\otimes\dots\otimes\bullet^{*}\otimes p)\mbox{ evaluated on }\Delta_{1,\dots,1,n-j}\mbox{\ensuremath{\sigma}}\\
 & =p(\std(\sigma_{j+1}\sigma_{j+2}\dots\sigma_{n}))\\
 & =\begin{cases}
1 & \mbox{if }\std(\sigma_{j+1}\sigma_{j+2}\dots\sigma_{n})=\bar{\tau};\\
-1 & \mbox{if }\std(\sigma_{j+1}\sigma_{j+2}\dots\sigma_{n})=1\bar{\tau}_{1}\dots\bar{\tau}_{\bar{i}-1}\bar{\tau}_{\bar{i}+1}\dots\bar{\tau}{}_{n};\\
0 & \mbox{otherwise.}
\end{cases}
\end{align*}
To show that $\{\f_{\tau}|\tau\in\sn\}$ form a basis of eigenfunctions,
it suffices to show that the $\f_{\tau}$ for a fixed $j$ are linearly
independent. Indeed, $\f_{\tau}(\tau)=1$ and $\f_{\tau'}(\tau)=0$
for any $\tau'\neq\tau$ with the ``same $j$'' as $\tau$ (i.e.
the smallest number not fixed by $\tau'$ is also the smallest number
not fixed by $\tau$.)
\end{proof}

\newcommand{\etalchar}[1]{$^{#1}$}

%\bibliography{descentoperator_ref} 

 %\printbibliography[heading=bibintoc]
\end{document}